\documentclass[11pt]{amsart}
\usepackage{geometry}
\usepackage{color}
\usepackage[colorlinks=true,linkcolor=red,urlcolor=blue]{hyperref}
\usepackage{amsthm}
\usepackage{graphicx}
\usepackage{times}
\usepackage{xcolor}
\graphicspath{ {./images/} }
\usepackage{tikz}
\usepackage{multirow}
\usepackage{caption}
\usepackage{subcaption}
\usepackage[square,numbers]{natbib}
\usepackage{hyperref}
\usepackage{comment}
\hypersetup{
    colorlinks=true,
    linkcolor=blue,
    filecolor=magenta,      
    urlcolor=cyan,
    pdftitle={Overleaf Example},
    pdfpagemode=FullScreen,
}

\usepackage[title]{appendix}
\usepackage{fancyhdr}	

\numberwithin{equation}{section}
\newtheorem{theorem}{Theorem}[section]

\newtheorem{lemma}{Lemma}[section]
\newtheorem{proposition}{Proposition}[section]
\newtheorem{remark}{Remark}[section]

\newtheorem{assumption}{Assumption}[section]
\newtheorem{definition}{Definition}[section]

\DeclareMathOperator*{\supp}{supp}

\title[Ergodic MFGs of Singular Control with Regime-Switching]{Ergodic Mean-Field Games of Singular Control with Regime-Switching (Extended Version)}
\author[Dianetti]{Jodi Dianetti}
\author[Ferrari]{Giorgio Ferrari}
\author[Tzouanas]{Ioannis Tzouanas}
\date{\today}

\address{J.~Dianetti: Department of Economics and Finance, University of Rome Tor Vergata, Rome, Italy}
\email{\href{mailto:jodi.dianetti@uniroma2.it}{jodi.dianetti@uniroma2.it}}
\address{G.~Ferrari: Center for Mathematical Economics (IMW), Bielefeld University, Universit\"atsstrasse 25, 33615, Bielefeld, Germany}
\email{\href{mailto:giorgio.ferrari@uni-bielefeld.de}{giorgio.ferrari@uni-bielefeld.de}}
\address{I.~Tzouanas: Center for Mathematical Economics (IMW), Bielefeld University, Universit\"atsstrasse 25, 33615, Bielefeld, Germany}
\email{\href{mailto:ioannis.tzouanas@uni-bielefeld.de}{ioannis.tzouanas@uni-bielefeld.de}}

\begin{document}
\setcitestyle{numbers}
\maketitle

\begin{abstract}
This paper studies a class of ergodic mean-field games of singular
stochastic control with regime-switching. The representative agent adjusts the
dynamics of a Markov-modulated It\^o-diffusion via a two-sided singular
stochastic control and faces a long-time-average expected profit criterion. The
mean-field interaction is of scalar type and it is given through the stationary
distribution of the population. Via a constructive approach, we prove the
existence and uniqueness of the stationary mean-field equilibrium. Furthermore,
we show that this realizes a symmetric $\varepsilon_N$-Nash equilibrium for a
suitable ergodic $N$-player game with singular controls. The proof hinges on
the characterization of the optimal solution to the representative player's
ergodic singular stochastic control problem with regime switching in terms of an auxiliary Dynkin game, which is of
independent interest and appears here for the first time. 
\end{abstract}
    \subsection*{Keywords:} stationary mean-field games; singular control; regime-switching; ergodic criterion; $\varepsilon$-Nash equilibrium.
    \subsection*{MSC subject classification:} 49L20, 91A15, 91A16, 60G40, 35R35, 93C30.

\section{Introduction}
Mean-field games (MFGs in short) have been introduced independently in 2006 by Larsy-Lions \cite{Lions} and Caines et al.\ \cite{Caines2006LargePS} as limit models for symmetric $N$-player differential games, where the interaction is through the empirical distribution of the states (and possibly of the actions) of the players. In MFGs, a representative agent determines her best reply to a given flow of probability measures -- e.g., representing the distribution of the states of the indistinguishable rivals -- so that the counterpart to the Nash equilibrium concept arising in $N$-player games takes now the form of a consistency condition: The law of the optimally controlled state process of the representative agent must agree with the given flow of probability measures. Since their introduction, because of their tractability, their relation to the theory of propagation of chaos and of forward-backward systems, and their ability to reproduce $\varepsilon_N$-Nash equilibria for suitably related symmetric $N$-player games, MFGs have attracted large attention in the mathematical and applied literature. We refer to the two-volume book by Carmona and Delarue \cite{carmona2018probabilistic} for a comprehensive presentation of results, approaches, and techniques, as well as to the paper by Carmona \cite{carmona2020applications} for a review of applications of MFGs in Economics, Finance, and Engineering.

In stationary MFGs, the representative player interacts with the long-run distribution of the population. Such a concept has a long tradition in economic theory: Stationary equilibria appeared already in the 1980s in the context of games with a continuum of players (see \cite{Hopeynan} and \cite{JOVANOVIC198877}), and also play an important role in the analysis of competitive market models with heterogeneous agents (see, e.g., \cite{PDE-Macro} and \cite{Luttmer}, amongst many others). Closely connected is also the concept of stationary oblivious equilibria, introduced by Adlakha et al.\ in  \cite{ADLAKHA2015269}.
Within the mathematical literature, stationary MFGs have been approached both via analytic and probabilistic methods. Among those papers adopting a partial differential equations (PDE) approach, we refer to the works of Bardi and Feleqi \cite{bardi2016nonlinear} for the study of the forward-backward system arising in stationary MFGs with regular controls, Gomes et al.\ \cite{GOMES201449} for extended stationary MFGs, Cardaliaguet and Porretta \cite{cardaliaguetprotera} for the study of the long-term behavior of the master equation arising in MFG theory, and to Bertucci \cite{BERTUCCI2018165} for the study of stationary mean-field optimal stopping games. On the other hand, a probabilistic approach is followed in a series of recent contributions dealing with stationary MFGs with singular and impulsive controls, see A\"id et al.\ \cite{Ferrari-Basei}, Cao and Guo \cite{CAO2022995}, Cao et al.\ \cite{CaoDianettiFerrari}, and Christensen et al.\ \cite{Christensen}.

\subsection{Our results} In this paper, we study a class of stationary MFGs where the underlying state process is a general singularly controlled one-dimensional diffusion whose coefficients are modulated by a continuous-time Markov chain with $d\geq 2$ states. More precisely, the representative agent optimally controls a Markov-modulated real-valued It\^o-diffusion through a two-sided singular control in order to maximize an ergodic reward functional. This is given by the long-time-average of the time-integral of a running profit function, net of the proportional costs of actions. The mean-field interaction is of scalar type and comes through a real-valued parameter denoted by $\theta$, which, at equilibrium, has to identify with a suitable generalized moment of the stationary distribution of the optimally controlled state process. From the economic point of view, $\theta$ can be thought of as a stationary price index arising from the aggregate productivity through an isoelastic demand function \`a la Spence-Dixit-Stiglitz (see pp.\ 7-8 in \cite{PDE-Macro}), or of as a stationary demand due to aggregate advertising (see pp.\ 595-596 in \cite{LightWeintraub}). We refer to Remark \ref{remark: benchmark example} below for details.

Under suitable assumptions on the problem's data, by employing mainly probabilistic means, we prove the existence and uniqueness of a stationary MFG equilibrium for the considered game. Furthermore, we show that this realizes an $\varepsilon_N$-Nash equilibrium for a related symmetric ergodic $N$-player game with singular controls.

Our first contribution consists in studying the representative player's optimal control problem and thus in providing, for the first time in the literature, the complete solution to a two-sided ergodic singular stochastic control problem with regime-switching (cf.\ Proposition \ref{Prop  3.5} below).\ This is accomplished through the study of an auxiliary optimal stopping game, which is then tackled through probabilistic arguments similarly to Ferrari and Rodosthenous in \cite{doi:10.1137/19M1245049} (where, however, the considered Dynkin game was related to a discounted singular stochastic control problem). To the best of our knowledge, connections between ergodic singular stochastic control problems and optimal stopping are still not well understood, contrary to the discounted or finite horizon case. Two exceptions are Karatzas \cite{karatzas1984} and Cao et al.\ \cite{CaoDianettiFerrari} where one-dimensional ergodic singular stochastic control problems are addressed via optimal stopping (in \cite{karatzas1984} the underlying process is a Brownian motion). However, when the dimension of the state process is larger than one, as in our case, no result is available in the literature. Through a  verification argument (see Propositions \ref{eq:Prop 3.4} and \ref{Prop  3.5} below), we show that, for a given and fixed mean-field parameter $\theta$, the optimal control is of barrier-type. That is, the optimal control uniquely solves a Skorokhod reflection problem (see e.g., Burdzy et al.\ \cite{burdzy2009skorokhod}) at endogenously determined barriers (the free boundaries of the underlying optimal stopping game), which depend on the underlying Markov chain and on the given and fixed mean-field parameter $\theta$. As a byproduct, we also show that the optimal upwards and downwards reflection policies satisfy a couple of functional equations resembling those in Theorem 2 of \cite{karatzas1984}.

The next step deals with the construction of the MFG equilibrium and with the proof of its uniqueness. To that end, we first show that the joint process constituted by the optimally controlled Markov-modulated diffusion process and the Markov chain admits a stationary distribution (cf.\ Proposition \ref{Prop 4.1} below). As a matter of fact, we prove that its cumulative distribution function is the unique classical solution to a weakly-coupled system of ordinary differential equations (cf.\ Equation \ref{eq:(4.1)} below). This generalizes the result for a Markov-modulated Brownian motion obtained by D'Auria and Kella in Theorem 1 of \cite{DAURIA20121566}. Clearly, the stationary distribution and its cumulative distribution depend on the fixed mean-field parameter $\theta$, since the optimally controlled state does. In order to proceed with the equilibrium analysis, we thus study the stability of the stationary distribution with respect to $\theta$ and actually prove its continuity with respect to such a parameter (cf.\ Theorem \ref{Thm 4.1}). Further exploiting the connection to the aforementioned Dynkin game of optimal stopping, we are then able to determine an invariant compact set where any equilibrium value of $\theta$ (if one exists) should lie. Combining those continuity and compactness results, an application of the Schauder-Tychonof fixed point theorem allows us to prove that there exists a unique stationary equilibrium (cf.\ Theorem \ref{Thm 4.1}).

The analysis of the considered stationary MFG is finally justified by the fact that its unique stationary mean-field equilibrium is able to realize an $\varepsilon_N$-Nash equilibrium for an $N$-player symmetric game with singular controls in which each player faces an ergodic net profit functional. It is worth noticing that in the $N$-player game the interaction comes through a suitable time-dependent average of the players' states (see Eqs.\ \ref{eq:functJn} and \ref{eq:thetaN} below), and it is therefore given in terms of the empirical distribution of players' states at the current time.

\subsection{Related literature} Ergodic singular stochastic control problems for one-dimensional diffusions have been treated in general settings, including state-dependent costs of actions, and with different applications, in \cite{AlvarezHenning}, \cite{LokkaZervosI}, \cite{Menaldi} and \cite{ZervosErgodic}, \cite{Kunwai}, among others. However, in all those papers, no regime switching is included.

Our paper is placed within the recent bunch of literature dealing with MFGs with singular controls by following a probabilistic approach; see A\"id et al.\ \cite{Ferrari-Basei}, Cao and Guo \cite{CAO2022995}, Cao et al.\ \cite{CaoDianettiFerrari}, Campi et al.\ \cite{Campi_et.al}, Cohen and Sun \cite{cohen2024existenceoptimalstationarysingular}, Dianetti et al.\ \cite{Unifying_Sub_MFG}, Fu \cite{Fu}, Fu and Horst \cite{Fu-Horst}, and Guo and Xu \cite{GuoXu}. Amongst those, the works that most relate to ours are those by A\"id et al.\ \cite{Ferrari-Basei} and by Cao et al.\ \cite{CaoDianettiFerrari}. Cao et al.\ consider in \cite{CaoDianettiFerrari} ergodic MFGs involving a one-dimensional singularly controlled It\^o-diffusion. However, differently to us, the control in \cite{CaoDianettiFerrari} can be exerted only upwards and no regime-switching process is considered therein. In our work, similarly to \cite{Ferrari-Basei}, we consider a stationary MFG involving a singularly controlled one-dimensional diffusion whose coefficients are modulated by a continuous-time Markov chain. However, differently to \cite{Ferrari-Basei}, here the control is two-sided, rather than only increasing, the performance criterion is of ergodic type, rather than of discounted type, the dynamics of the underlying state process are general, rather than geometric, and the Markov chain has $d\geq 2$ states, rather than only two regimes. 

We also clearly relate to those works dealing with MFGs involving regime-switching regular control models. Wang and Zhang \cite{WangFeng} consider social optima of mean-field linear-quadratic-Gaussian control models with Markov jump parameters, while distributed games for large-population multi-agent systems with random time-varying parameters are investigated in \cite{WangZhang}. Bensoussan et al.\ \cite{bensoussan2020mean} focus on MFGs of risk-sensitive type with jump-diffusions and regime-switching. Furthermore, due to the application in networks with switching mechanism, mean-field control problems with regime-switching became recently of particular interest: See, among others, Bayraktar et al.\ \cite{bayraktar2023mean}, Zhang et al.\ \cite{ZhangSun} and Nguyen et al.\ \cite{NguyenYin}.

\subsection{Organization of the paper}
The rest of the paper is organized as follows.\ In Section \ref{Section 2}, we introduce the probabilistic setting and the MFG under study. Next, in Section \ref{Chapter 3}, for a given and fixed mean-field parameter, we solve the ergodic stochastic control problem faced by the representative player. In Section \ref{Mean field analysis} we then prove the existence and uniqueness of the mean-field equilibrium, while in Section \ref{Approximation} we provide the approximation result for a related $N$-player symmetric game. Finally, technical proofs are collected in Appendices \ref{Appendix A}, \ref{Appendix B} and \ref{Appendix C}.

\section{Probabilistic setting}
\label{Section 2}
Let $(\Omega,\mathcal{F}, \mathbb{F}:=\{\mathcal{F}_{t}\}_{t\geq 0}, \mathbb{P})$ be a filtered probability space which satisfies the usual conditions, on which it is defined a one-dimensional $\mathbb{F}$-Brownian motion $\{W_{t}\}_{t\geq 0}$ and an independent irreducible $\mathbb{F}$-adapted continuous-time Markov chain $\{Y_{t}\}_{t\geq 0}$. The Markov chain $Y$ has state space $\mathbb{Y}:=\{1,...,d\}$ and transition matrix $\mathbb{Q}:=\{q_{ij}\}_{1\leq i,j\leq d}$. The transition rates are such that $\kappa_{i}:=-q_{ii}>0$ and the condition $\sum_{j\in\mathbb{Y}}q_{ij}=0$ holds for every $i\in\mathbb{Y}$. Accordingly, 
setting
$P_{ij}(t):=\mathbb{P}(Y_t=j|Y_0=i)$, we have $P'_{ij}(t)=\sum_{k\in \mathbb{Y}}q_{i\ell}P_{\ell j}(t)$, for $i,j\in \mathbb{Y}$. It also follows that $Y$ admits a stationary distribution, whose $i$-th component is given by
\begin{equation}
    p(i):=\frac{\kappa_{i}}{\sum_{j=1}^{d}\kappa_{j}}.
\end{equation}

For given Borel-measurable functions $b:\mathbb{R}\times \mathbb{Y}\to \mathbb{R}$, $\sigma:\mathbb{R}\times \mathbb{Y}\to (0,\infty)$, we introduce the process $X^{0}$ with state space $\mathcal{I}:=(\underline{x},\overline{x})$ and dynamics
\begin{equation}
\label{eq:SDE-X0}
dX^{0}_{t}=b(X^{0}_{t},Y_{t})dt+\sigma(X^{0}_{t},Y_{t})dW_{t},\quad (X^{0}_{0},Y_{0})=(x,i)\in\mathcal{I}\times \mathbb{Y}.
\end{equation}
Here, $-\infty \leq \underline{x} < \overline{x} \leq + \infty$ are the boundary points of the state space of $X^0$. Let then
\begin{equation*}
    \mathcal{A}:=\{\{\xi_{t}\}_{t\geq 0},\;\mathbb{F}\text{-adapted, with bounded-variation,  left-continuous, }\xi_{0}=0,\text{ a.s.}\}
\end{equation*}
and notice that any $\xi\in\mathcal{A}$ admits the Jordan decomposition $\xi=\xi^{+}-\xi^{-}$, for $\xi^{\pm}$ nondecreasing. 
Also, let $\{ |\xi|_{t} \}_{t\geq0}:=\xi^{+}+\xi^{-}$ denote the variation of $\xi\in\mathcal{A}$. Then, for given $\xi\in\mathcal{A}$ and $x\in \mathcal{I}$, we introduce the process $X^{\xi}$ with dynamics
\begin{equation}
\label{eq:(2.2)}
dX^{\xi}_{t}=b(X^{\xi}_{t},Y_{t})dt+\sigma(X^{\xi}_{t},Y_{t})dW_{t}+d\xi^{+}_{t}-d\xi^{-}_{t},\quad (X^{\xi}_{0},Y_{0})=(x,i)\in\mathcal{I}\times \mathbb{Y}.
\end{equation}

The following assumption in particular ensures that there exists a unique strong solution to (\ref{eq:(2.2)}), for every $\xi\in\mathcal{A}$ and $(x,i)\in\mathcal{I}\times\mathbb{Y}$ (hence, in particular, for $\xi\equiv0$; see Theorem 7 Chapter V in \cite{protter2005stochastic}). In the following, we shall denote such strong solutions by $(X^{x,\xi},Y^{i})$ and $(X^{x,0},Y^{i})$, when needed.

\begin{assumption}
\label{eq:Ass 2.1}
The following hold:
        \begin{enumerate}
            \item \label{eq:(2.1-1)} The functions $b(\cdot,i)$ and $\sigma(\cdot,i)$ are twice continuously differentiable, for every $i\in\mathbb{Y}$.
            \item \label{eq:(2.1-2)} There exists $C>0$, such that
                        $|b(x,i)|+|\sigma(x,i)|\leq C(1+|x|)$, for any $(x,i)\in\mathcal{I}\times \mathbb{Y}$.
            \item \label{eq:(2.1-3)} There exists $c>0$, such that
               $ b_{x}(x,i)\leq -c$, for any $(x,i)\in\mathcal{I}\times \mathbb{Y}$.
            \item \label{eq:(2.1-4)} For any $(x,i)\in \mathcal{I}\times \mathbb{Y}$,
            $
            \sigma(x,i)>0.
            $
        \end{enumerate}
\end{assumption}
Conditions (\ref{eq:(2.1-1)}) and (\ref{eq:(2.1-2)}) in Assumption \ref{eq:Ass 2.1} imply that $(X^{0},Y)$ is regular, meaning that there exists a sequence of stopping times $\{\beta_{n}\}_{n\geq 0}$, with $\beta_{n}:=\inf\{t\geq 0: |X^{0}_{t}|=n\}$, such that $\beta_{\infty}:=\lim_{n\to \infty}\beta_{n}=\infty,\; \mathbb{P}$-a.s.; for further details see Section 2.3 in \cite{yin2009hybrid}.

For our subsequent analysis, we introduce the $\mathbb{F}$-adapted process $\{\widehat{X}_{t}\}_{t\geq 0}$, which evolves as
\begin{equation}
    \label{eq:(2.3)}
    d\widehat{X}_{t}=(b(\widehat{X}_{t},Y_{t})+\sigma\sigma_{x}(\widehat{X}_{t},Y_{t}))dt+\sigma(\widehat{X}_{t},Y_{t})d\widehat{W}_{t},\quad (\widehat{X}_{0},Y_{0})=(x,i)\in\mathcal{I}\times \mathbb{Y},
\end{equation}
for an $\mathbb{F}$-adapted Brownian motion $\widehat{W}$. Notice that Equation (\ref{eq:(2.3)}) also admits a unique strong solution $(\widehat{X}^{x},Y^{i})$ which is regular, due to Assumption \ref{eq:Ass 2.1}.
Process $\widehat{X}$ will play an important role in Section \ref{Section 3.1} below since the couple $(\widehat{X},Y)$ will be the state process of the auxiliary Dynkin game that we will introduce to solve the ergodic mean-field game under consideration.

For $g:\mathbb{R}\times \mathbb{Y}\to \mathbb{R}$ such that $g(\cdot,i)\in C^{2}(\mathbb{R}),\text{ for any }i\in\mathbb{Y}$, the infinitesimal generator of the  uncontrolled process $(X^{0},Y)$ is denoted by $\mathcal{L}_{(X,Y)}$ and it is such that
\begin{equation}\label{eq:(2.4)}
    \big(\mathcal{L}_{(X,Y)}g\big)(x,i)=\frac{1}{2}\sigma^{2}(x,i)g_{xx}(x,i)+b(x,i)g_{x}(x,i)+\sum_{j\neq i}q_{ij}(g(x,j)-g(x,i)),
\end{equation}
while the infinitesimal generator $\mathcal{L}_{(\widehat{X},Y)}$ for the process $(\widehat{X},Y)$ is such that
\begin{equation}
    \label{eq:(2.5)}
    \big(\mathcal{L}_{(\widehat{X},Y)}g\big)(x,i)=\frac{1}{2}\sigma^{2}(x,i)g_{xx}(x,i)+(b(x,i)+\sigma\sigma_{x}(x,i))g_{x}(x,i)+\sum_{j\neq i}q_{ij}(g(x,j)-g(x,i)).
\end{equation}

\begin{remark}
    Our setting accommodates (uncontrolled) affine diffusion processes where $X^{0}$ is either a geometric Brownian motion with $b(x,i)=-\delta_{i}x$ and $\sigma(x,i)=\sigma_{i}x$ or a mean-reverting process with $b(x,i)=\rho_{i}(\kappa_{i}-x)$ and $\sigma(x,i)=\sigma_{i}x$, for constants $\delta_{i},\rho_{i},\kappa_{i},\sigma_{i}>0$  for any $i\in\mathbb{Y}$. 
\end{remark}
In the rest of the paper, we adopt the following notation: $\mathbb{P}_{(x,i)}[\,\cdot\,]:=\mathbb{P}[\,\cdot\,|X_{0}^{\xi}=x,Y_{0}=i]$ and  $\mathbb{E}_{(x,i)}[\,\cdot\,]:=\mathbb{E}^{\mathbb{P}}[\,\cdot\, |X^{\xi}_{0}=x,Y_{0}=i]$ for the hybrid-diffusion process $(X^{\xi},Y)$, and $\widehat{\mathbb{P}}_{(x,i)}[\,\cdot\,]:=\widehat{\mathbb{P}}[\,\cdot\,|\widehat{X}_{0}=x,Y_{0}=i]$ and $\widehat{\mathbb{E}}_{(x,i)}[\,\cdot\,]:=\mathbb{E}^{\widehat{\mathbb{P}}}[\,\cdot\, |\widehat{X}_{0}=x,Y_{0}=i]$ for $(\widehat{X},Y)$. We also set $\mathbb{P}_{i}[\,\cdot\,]:=\mathbb{P}[\,\cdot\, |Y_{0}=i]$ and denote by $\mathbb{E}_{i}$ the corresponding expectation.

\subsection{The ergodic mean-field game}
\label{sec:ErgodicMFG}

Within the previous probabilistic setting, we now introduce the ergodic mean-field game (ergodic MFG for short) which will be the main object of our study. 

We restrict to those $\xi\in\mathcal{A}$ belonging to
\begin{equation}
    \label{admissible controls} \mathcal{A}_{e}:=\bigg\{\xi\in\mathcal{A}:\, X^{\xi}_{t} \in \mathcal{I} \,\,\forall t \geq 0\,\,\text{a.s.},\,\mathbb{E}\big[|\xi|_{T}\big]<\infty\,\,\forall T<\infty,\,\limsup_{T\uparrow\infty}\frac{1}{T}\mathbb{E}\big[|X^{\xi}_{T}|\big]=0\bigg\},
\end{equation}
and, for $(x,i)\in\mathcal{I}\times \mathbb{Y}=:\mathcal{O}$, $\xi\in\mathcal{A}_{e}$ and $\theta\in\mathbb{R}_{+}$, we introduce the ergodic profit functional
\begin{equation}
    \label{eq:(2.6)}
    J(x,i;\xi,\theta):=\limsup_{T\uparrow \infty}\frac{1}{T}\mathbb{E}_{(x,i)}\bigg[ \int_{0}^{T}\pi(X^{\xi}_{t},\theta)dt-k_{1}\xi^{+}_{T}+k_{2}\xi^{-}_{T} \bigg],
\end{equation}
where $0<k_{2}<k_{1}$. 
In (\ref{eq:(2.6)}), $\pi$ is the instantaneous profit function satisfying Assumption \ref{eq:Ass 2.2} below, and $\theta$ is (for the moment) a fixed nonnegative number, which drives the mean-field interaction (cf.\ Definition \ref{eq: Def 1} below).
Notice that $J$ is well defined for $\xi \in \mathcal{A}_e$, although potentially infinite.

We are now ready to provide the notion of stationary mean-field equilibrium. 
To that end, consider functions $F:\mathbb{R}_{+}\to\mathbb{R}_{+}$ and $f:\mathcal{I}\to\mathbb{R}_{+}$, 
and introduce the following notion of equilibrium.
\begin{definition}[Ergodic MFG Equilibrium]
        \label{eq: Def 1}
        For an initial condition $(x,i)\in \mathcal{O}$, a couple $(\xi^{*}(\theta^{*}),\theta^{*})\in\mathcal{A}_{e}\times\mathbb{R}_{+}$ is said to be an \textbf{equilibrium of the ergodic MFG} for the initial condition $(x,i)$ if
        \begin{enumerate}
            \item \label{eq:Def 1-1} $J(x,i;\xi^{*}(\theta^{*}),\theta^{*})\geq J(x,i;\xi,\theta^{*}),\text{ for any }\xi\in\mathcal{A}_{e}$.
            \item \label{eq:Def 1-2} The optimally controlled state process $(X^{\xi^{*}(\theta^{*})},Y)$ admits a limiting stationary distribution $\mu^{\theta^{*}}$ and $\theta^{*}=F\big(\sum_{i=1}^{d} \int_{\mathcal{I}}f(x)\mu^{\theta^{*}}(dx,i)\big)$.
        \end{enumerate}
\end{definition}

In the sequel, our solution plan will be as following:
\begin{enumerate}
    \item For a fixed mean-field parameter $\theta\in\mathbb{R}_{+}$, we solve the ergodic control problem aiming at maximizing (\ref{eq:(2.6)}) over $\mathcal{A}_{e}$.
    \item We determine the stationary distribution of the optimally controlled state process, we impose the consistency condition (\ref{eq:Def 1-2}) in Definition \ref{eq: Def 1}, and we prove the existence and uniqueness of the mean-field parameter $\theta^{*}$ via a fixed-point argument.
    \item We finally show that the constructed ergodic MFG equilibrium realizes an $\epsilon_{N}$-Nash equilibrium for a suitable ergodic $N$-player game of singular control with regime-switching.
\end{enumerate}
\subsection{Assumptions and examples}
The instantaneous profit function $\pi$ fulfills the following conditions.
\begin{assumption}
    \label{eq:Ass 2.2}
    The function $\pi:\mathbb{R}\times \mathbb{R}_{+}\mapsto \mathbb{R}_{+}$ is such that:
    \begin{enumerate}
    	\item \label{eq:2.2-1} $\pi(\cdot,\theta)\in C^{2}(\mathcal{I})$, for any $\theta\in\mathbb{R}_{+}$; 
        \item \label{eq:2.2-2} $\pi(\cdot,\theta)$ is non-decreasing and concave, for any $\theta\in\mathbb{R}_{+}$;
        \item \label{eq:2.2-3}$\pi_{x\theta}$ is continuous and it is such that $\pi_{x\theta}(x,\theta)<0$, for any $(x,\theta)\in\mathcal{I}\times \mathbb{R}_{+}$;
        \item \label{eq:2.2-4} for every 
        $(i,\theta)\in\mathbb{Y}\times \mathbb{R}_{+}$,
        \begin{equation*}
            \lim_{x\downarrow \underline{x}}\widehat{\mathbb{E}}_{(x,i)}\bigg[\int_{0}^{\infty}e^{\int_{0}^{t}b_{x}(\widehat{X}_{s},Y_{s})ds}\pi_{x}(\widehat{X}_{t},\theta)dt\bigg]=\infty,
        \end{equation*}
        and,
        \begin{equation*}
            \lim_{x\uparrow \overline{x}}\widehat{\mathbb{E}}_{(x,i)}\bigg[\int_{0}^{\infty}e^{\int_{0}^{t}b_{x}(\widehat{X}_{s},Y_{s})ds}\pi_{x}(\widehat{X}_{t},\theta)dt\bigg]=0;
        \end{equation*}
        \item \label{eq:2.2-5} for every 
        $(x,i,\theta)\in\mathcal{O}\times \mathbb{R}_{+}$,
        \begin{equation*}
            \widehat{\mathbb{E}}_{(x,i)}\bigg[ \int_{0}^{\infty}e^{-ct}|\pi_{x}(\widehat{X}_{t},\theta)|dt\bigg]<\infty,
        \end{equation*}
         and, for some $\epsilon_{0}:=\epsilon_{0}(x)\in (0,1)$,
        $$
        \begin{aligned}
            &\widehat{\mathbb{E}}\bigg[ \int_{0}^{\infty}e^{-ct}\sup_{x'\in \mathcal I _x^{\epsilon_{0}} }\Big(|\pi_{xx}(\widehat{X}^{x'}_{t},\theta)|\partial_{x}\widehat{X}^{x}_{s}\big|_{x=x'}\Big)dt\bigg]+ \\
            &+\widehat{\mathbb{E}}\bigg[ \int_{0}^{\infty}e^{-ct}|\pi_{x}(\widehat{X}^{x}_{t},\theta)
    |\sup_{x'\in\mathcal I _x^{\epsilon_0} }\bigg(\int_{0}^{t}|b_{xx}(\widehat{X}^{x'}_{s},Y^{i}_{s})|\partial_{x}\widehat{X}^{x}_{s}\big|_{x=x'}ds\bigg)dt\bigg]<\infty,
        \end{aligned}
        $$
        where $ \mathcal I _x^{\epsilon_{0}}:=(x,x+\varepsilon_0)$, 
        $c>0$ is the same constant as in Assumption \ref{eq:Ass 2.1}-(\ref{eq:(2.1-3)})
         and $\partial_{x}\widehat{X}^{x}_{t}:=\exp{\big( \int_{0}^{t}\big(b_{x}+\partial_{x}(\sigma\sigma_{x})-\frac{1}{2}\sigma^{2}_{x}\big)(\widehat{X}_{s}^{x},Y_{s})ds+\int_{0}^{t}\sigma_{x}(\widehat{X}^{x}_{s},Y_{s})d\widehat{W}_{s} \big)}$;
        \item \label{assumption bxx} for every $(x,i)\in\mathcal{O}$, and for some $\epsilon_0:=\epsilon_0(x) \in (0,1)$, it holds that
            \begin{equation}
                \label{int condition for stopping}
                \widehat{\mathbb{E}}\bigg[ \int_{0}^{\infty}e^{-ct}\, \sup_{x'\in (x,x+\epsilon_0)}\Big(|b_{xx}(\widehat{X}^{x'}_{t},Y^{i}_{t})\big| \partial_{x}\widehat{X}^{x}_{t}\big|_{x=x'}\Big) dt \bigg]<\infty.
            \end{equation}
    \end{enumerate}
\end{assumption}

In the sequel, we enforce the following requirements on the functions $f,F$.
\begin{assumption}
    \label{eq:Ass 2.3}
    $F:\mathbb{R}_{+}\to\mathbb{R}_{+}$, $f:\mathcal{I}\to\mathbb{R}_{+}$ are such that:
    \begin{enumerate}
        \item \label{eq:Ass 2.3-1} $F$ and $f$ are strictly increasing continuously differentiable functions;
        \item  for $\beta\in(0,1)$, there exists $C>0$ such that:
        \begin{enumerate}
            \item \label{eq:Ass 2.3-2-a}
            $
                |f(x)|\leq C(1+|x|^{\beta}),\quad |F(x)|\leq C(1+|x|^{\frac{1}{\beta}}),
            $
            \item \label{eq:Ass 2.3-2-b}
            $
                \big| F(x)-F(y) \big|\leq C(1+|x|+|y|)^{\frac{1}{\beta}-1}|x-y|;
            $
        \end{enumerate}
        \item \label{eq:Ass 2.3-3} $\lim_{x\uparrow\infty}F(x)=\lim_{x\uparrow\overline{x}}f(x)=\infty$.
    \end{enumerate}
\end{assumption}
\begin{remark}
    \label{remark: benchmark example} As a benchmark example we take a profit function
    \begin{equation}
        \pi(x,\theta)=x^{\beta}(\theta^{-(1+\beta)}+\kappa_*),\quad F(x)=x^{1/\beta},\quad f(x)=x^{\beta}, \quad \kappa_*>0, \ \beta\in (0,1),
    \end{equation}
    and dynamics $b(x,i)=-\delta_{i}x$ and $\sigma(x,i)=\sigma_{i}x$, for constants $\delta_{i}, \sigma_{i}>0$  for any $i\in\mathbb{Y}$.
    All the assumptions are easily verified to hold.
    Such an example relates to a MFG of partially reversible investment. In this regard, $X$ represents the productivity of a representative company. Such productivity can be instantaneously increased via the cumulative investment process $\xi^+$ and decreased via the cumulative disinvestment process $\xi^-$. The function $\pi$ is the instantaneous profit accrued from the production of a final good, and this depend on the company's productivity level, as well as on the equilibrium stationary price $\theta$ arising from a competing market where symmetric producing companies face isoelastic demand functions \`a la Spence-Dixit-Stiglitz (cf.\ Section 3 in \cite{PDE-Macro}, as well as Section 2.1 in \cite{Calvia-etal} for a microeconomic foundation in a setting with regular controls). %Furthermore, our framework can also be applied to a continuous time-space ergodic mean-field game with singular controls of dynamic advertising games with monopolistic competition (see also Section 4.2 in \cite{LightWeintraub} for the discrete counterpart).
\end{remark}

\section{The Ergodic optimal control problem}
\label{Chapter 3}

Recalling (\ref{eq:(2.6)}), in this section we fix $\theta\in\mathbb{R}_{+}$ and solve the ergodic control problem. In particular, we want to find 
\begin{equation}
    \label{eq:(3.1)}
    \bar{\lambda}(\theta):=\sup_{\xi\in\mathcal{A}_{e}}J(x,i;\xi,\theta).
\end{equation}
To the best of our knowledge, this is the first paper that solves a singular stochastic control problem with regime switching and ergodic performance criterion. In order to solve (\ref{eq:(3.1)}), we let $V:\mathcal{O}\times \mathbb{R}_{+}\to \mathbb{R}$ and $\lambda:\mathbb{R}_{+}\times \mathbb{Y}\to \mathbb{R}$ to be determined such that $V(\cdot,i;\theta)\in C^{2}(\mathcal{I}),\text{ for any } (i,\theta)\in\mathbb{Y}\times \mathbb{R}_{+}$, and the pair $(V,\lambda)$ solves the variational inequality
\begin{equation}
\label{eq:(3.2)}
    \max\big\{ \mathcal{L}_{(X,Y)}V(x,i;\theta)+\pi(x,\theta)-\lambda(\theta,i),V_{x}(x,i;\theta)-k_{1},k_{2}-V_{x}(x,i;\theta) \big\}=0.
\end{equation}
It will be shown in Proposition \ref{eq:Prop 3.4} below that a solution $(V,\lambda)$ to (\ref{eq:(3.2)}) allows to obtain $\bar{\lambda}$ in the sense that
\begin{equation}
\label{eq:(3.3)}
	\bar{\lambda}(\theta)=\sum_{i=1}^{d}p(i)\lambda(\theta,i),
\end{equation} 
where $(p(i))_{i\in \mathbb{Y}}$ denotes the stationary distribution of the Markov chain $Y$.
The following additional assumption on $\pi$ holds throughout the rest of this paper.
\begin{assumption}
    \label{eq:Ass 3.1}
    For  $x_{-}(\theta):=(x_{-}(i,\theta))_{i \in \mathbb Y}$ and $x_{+}(\theta):=(x_{+}(i,\theta))_{i \in \mathbb Y}$, with $x_{-}(i,\theta)<x_{+}(i,\theta),\,i\in\mathbb{Y}$, it holds:
        \begin{equation*}
            \pi_{x}(x,\theta)+k_{1}b_{x}(x,i)\begin{cases}
                                                   >0,\quad x\in (\underline{x}, x_{-}(i,\theta)), \\
                                                   =0,\quad x=x_{-}(i,\theta), \\
                                                   <0,\quad x\in(x_{-}(i,\theta),\overline{x}),
                                                \end{cases}
        \end{equation*}
        \begin{equation*}
            \pi_{x}(x,\theta)+k_{2}b_{x}(x,i)\begin{cases}
                                                   >0,\quad x\in(\underline{x}, x_{+}(i,\theta)), \\
                                                   =0,\quad x=x_{+}(i,\theta), \\
                                                   <0,\quad x\in (x_{+}(i,\theta),\overline{x}).
                                                \end{cases}
        \end{equation*}
        
\end{assumption}
Assumption \ref{eq:Ass 3.1} together with Assumption \ref{eq:Ass 2.2}-(\ref{eq:2.2-4}) guarantee that a solution to (\ref{eq:(3.2)}) will be of threshold type, meaning that there shall exist $\alpha(i,\theta)<\beta(i,\theta),\;i\in\mathbb{Y}$, such that $V_{x}(x,i;\theta)<k_{1}$ on $\{(x,i)\in\mathcal{O}:x<\alpha(i,\theta)\}$ and $V_{x}(x,i;\theta)>k_{2}$ on $\{(x,i)\in\mathcal{O}:x>\beta(i,\theta)\}$. In order to find a solution to (\ref{eq:(3.2)}) and to characterize free-boundaries $\alpha(i,\theta)$ and $\beta(i,\theta)$, we examine a connection of stochastic control problem with a zero-sum optimal stopping game. While the connection singular stochastic control-optimal stopping is well known in the discounted or finite horizon setting, the result is not yet well understood in the ergodic setting, and, in particular, no literature is available in the regime-switching case.

\subsection{An auxiliary optimal stopping game}
\label{Section 3.1}
Mimicking what it is typically done in the finite horizon or discounted setting (see, e.g., \cite{doi:10.1137/19M1245049} and references therein), we guess that $V_x$ identifies with the value of an auxiliary Dynkin game. This has stopping functional
\[
	\widehat{J}(x,i;\tau,\sigma,\theta):=\widehat{\mathbb{E}}_{(x,i)}\bigg[ \int_{0}^{\tau\wedge \sigma}e^{\int_{0}^{t}b_{x}(\widehat{X}_{s},Y_{s})ds}\pi_{x}(\widehat{X}_{t},\theta)dt+k_{1}e^{\int_{0}^{\tau}b_{x}(\widehat{X}_{t},Y_{t})dt}\boldsymbol{1}_{\{\tau<\sigma\}}
\]
\begin{equation}
   \label{eq:(3.4)}
   +k_{2}e^{\int_{0}^{\sigma}b_{x}(\widehat{X}_{t},Y_{t})dt}\boldsymbol{1}_{\{\sigma<\tau\}} \bigg],
\end{equation}
where $(\widehat{X}_{t},Y_{t})_{t\geq 0}$ is the unique strong solution to (\ref{eq:(2.3)}) and $\tau,\sigma\in\mathcal{T}$, with
$\mathcal{T}:=\{\rho:\Omega\to [0,\infty]:\rho \text{ is an }\mathbb{F}\text{-stopping time}\}$. In (\ref{eq:(3.4)}), we use the convention $e^{\int_{0}^{\rho}b_{x}(\widehat{X}^{x}_{s},Y^{i}_{s})ds}=0$ on $\{\rho=\infty\}$. The Dynkin game is such that Player 1 aims at minimizing (\ref{eq:(3.4)}) over $\tau\in\mathcal{T}$, while Player 2 at maximizing (\ref{eq:(3.4)}) over $\sigma\in \mathcal{T}$. Theorem 2.1 in \cite{doi:10.1137/S0040585X97983821} allows to show that such a game indeed admits a value.

\begin{theorem}
\label{eq:Thm 3.1}
    Let $(x,i,\theta)\in\mathcal{O}\times \mathbb{R}_{+}$. Then, 
\begin{equation}
    \label{eq:(3.5)}
    \inf_{\tau\in\mathcal{T}}\sup_{\sigma\in\mathcal{T}}\widehat{J}(x,i;\tau,\sigma,\theta)=\sup_{\sigma\in\mathcal{T}}\inf_{\tau\in\mathcal{T}}\widehat{J}(x,i;\tau,\sigma,\theta),
\end{equation}
and we define the value function
\begin{equation}
\label{eq:(3.6)}
    v(x,i;\theta):=\inf_{\tau\in\mathcal{T}}\sup_{\sigma\in\mathcal{T}}\widehat{J}(x,i;\tau,\sigma,\theta)=\sup_{\sigma\in\mathcal{T}}\inf_{\tau\in\mathcal{T}}\widehat{J}(x,i;\tau,\sigma,\theta).
\end{equation}
Moreover, letting the \textit{continuation region} be
\[
    \mathcal{C}^{\theta}:=\{(x,i)\in\mathcal{O}:k_{2}<v(x,i;\theta)<k_{1}\},
\]
and the \textit{stopping regions} be 
\[
    \mathcal{S}^{\theta}_{\inf}:=\{(x,i)\in\mathcal{O}:v(x,i;\theta)\geq k_{1}\}, \quad \mathcal{S}^{\theta}_{\sup}:=\{(x,i)\in\mathcal{O}:v(x,i;\theta)\leq k_{2}\},
\]
one has that the stopping times
\[
    \tau^{*}(x,i;\theta):=\inf\{t\geq 0:(\widehat{X}_{t},Y_{t})\in\mathcal{S}^{\theta}_{\inf}\}, \ \sigma^{*}(x,i;\theta):=\inf\{t\geq 0:(\widehat{X}_{t},Y_{t})\in\mathcal{S}^{\theta}_{\sup}\},%\;\widehat{\mathbb{P}}_{(x,i)}\text{-a.s.}, 
\]
realize a saddle-point. In particular,
\[
    v(x,i;\theta)=\widehat{J}(x,i;\tau^{*}(x,i;\theta),\sigma^{*}(x,i;\theta),\theta).
\]
\end{theorem}
\begin{proof}
     See Appendix \ref{Proof of Theorem 3.1}.
\end{proof}
The following proposition is easily proved thanks to Conditions (\ref{eq:2.2-2}) and (\ref{eq:2.2-3}) in Assumption \ref{eq:Ass 2.2}.
\begin{proposition}
\label{eq:Prop 3.1}
    It holds:
    \begin{enumerate}
        \item The mapping $x\mapsto v(x,i;\theta)$ is decreasing for every $(i,\theta)\in\mathbb{Y}\times \mathbb{R}_{+}$.
        \item The mapping $\theta\mapsto v(x,i;\theta)$ is decreasing for every $(x,i)\in\mathcal{O}$.
    \end{enumerate}
\end{proposition}
The next result excludes the possibility of empty stopping regions.
\begin{proposition}
\label{eq:Prop 3.2}
        It holds $\mathcal{S}^{\theta}_{inf}\neq \emptyset\;\text{ and }\; \mathcal{S}^{\theta}_{sup}\neq \emptyset$.
\end{proposition}
\begin{proof}
     See Appendix \ref{Proof of Proposition 3.2}.
\end{proof}
We define the free-boundaries
\begin{equation}
   \label{eq:(3.7)}
    \alpha_i (\theta):=\sup\{x\in\mathcal{I}:v(x,i;\theta)\geq k_{1}\},\quad \beta_i(\theta):=\inf\{x\in\mathcal{I}:v(x,i;\theta)\leq k_{2}\},
\end{equation}
for any $(i,\theta)\in\mathbb{Y}\times \mathbb{R}_{+}$, with the conventions $\sup\emptyset=\underline{x}$ and $\inf\emptyset=\overline{x}$. Then, thanks to the monotonicity of $v(\cdot,i;\theta),\; (i,\theta)\in\mathbb{Y}\times \mathbb{R}_{+}$, we have
\begin{align}
     \label{eq:(3.8)}
    & \mathcal{C}^{\theta}=\{(x,i)\in\mathcal{O}: x\in (\alpha_i (\theta),\beta_i (\theta))\}, \\
    \label{eq:(3.9)}
    & \mathcal{S}^{\theta}_{\inf}=\{(x,i)\in\mathcal{O}: x\in (\underline{x},\alpha_i(\theta)]\}\;\text{  and  }\;\mathcal{S}^{\theta}_{\sup}=\{(x,i)\in\mathcal{O}: x\in [\beta_i(\theta),\overline{x})\}.
\end{align}
Notice that, due to Assumption \ref{eq:Ass 3.1}, $x_{-}(i,\theta)<\alpha(i,\theta)$ and $\beta(i,\theta)<x_{+}(i,\theta)$, for any $(i,\theta)\in\mathbb{Y}\times \mathbb{R}_{+}$.

The next proposition provides regularity results about the value of the Dynkin game (\ref{eq:(3.6)}). The proof is in the same spirit as that of Theorem 4.3 in \cite{doi:10.1137/19M1245049}.
\begin{proposition}
\label{eq:Prop 3.3}
    For any $(i,\theta)\in\mathbb{Y}\times \mathbb{R}_{+}$ we have $v(\cdot,i;\theta)\in C^{1}(\mathcal{I})\cap C^{2}(\mathcal{I}\setminus \{\alpha_{i}(\theta),\beta_{i}(\theta)\})$.
\end{proposition}
\begin{proof}
     See Appendix \ref{Proof of Proposition 3.3}.
\end{proof}
By Proposition \ref{eq:Prop 3.3} and Proposition \ref{Prop A.1} (see also Theorem 2.1 in \cite{doi:10.1137/S0040585X97983821}) and using the fact that $\alpha(i,\theta)$ and $\beta(i,\theta)$ are finite, we are allowed to conclude that $(v(x,i;\theta),\alpha_{i}(\theta),\beta_{i}(\theta))_{i\in\mathbb{Y}}$ solves in the classical sense the following free-boundary problem:
\begin{equation}
    \label{eq:(3.10)}
    \begin{cases}
        \mathcal{L}_{(\widehat{X},Y)}v(x,i;\theta)+b_{x}(x,i)v(x,i;\theta)+\pi_{x}(x,\theta)= 0,\: x\in (\alpha_{i}(\theta),\beta_{i}(\theta)),\;\quad i\in\mathbb{Y}, \\
        \mathcal{L}_{(\widehat{X},Y)}v(x,i;\theta)+b_{x}(x,i)v(x,i;\theta)+\pi_{x}(x,\theta)\leq 0,\quad x\in (\alpha_{i}(\theta),\overline{x}),\qquad i\in\mathbb{Y}, \\
        \mathcal{L}_{(\widehat{X},Y)}v(x,i;\theta)+b_{x}(x,i)v(x,i;\theta)+\pi_{x}(x,\theta)\geq 0,\quad x\in (\underline{x},\beta_{i}(\theta)),\qquad i\in\mathbb{Y}, \\
        v(x,i;\theta)=k_{1},\quad x\in (\underline{x},\alpha_{i}(\theta)],\quad i\in\mathbb{Y}, \\
        v(x,i;\theta)=k_{2},\quad x\in [\beta_{i}(\theta),\overline{x}),\quad i\in\mathbb{Y}, \\
        v_{x}(x,i;\theta)=0,\quad x\in (\underline{x},\alpha_{i}(\theta)],\quad i\in\mathbb{Y}, \\
        v_{x}(x,i;\theta)=0,\quad x\in [\beta_{i}(\theta),\overline{x}),\quad i\in\mathbb{Y}.
    \end{cases}
\end{equation}

\subsection{The solution to the ergodic singular control problem}
\label{Section 3.2}
In this section, via a verification theorem, we establish the relation between the Dynkin game of Section \ref{Section 3.1} and the ergodic singular control problem (\ref{eq:(3.1)}). 

\begin{proposition}
\label{eq:Prop 3.4}
        Recall $v$ as in (\ref{eq:(3.6)}), the free boundaries $(\alpha_{i}(\theta))_{i\in \mathbb{Y}}$, $(\beta_{i}(\theta))_{i\in\mathbb{Y}}$ as in (\ref{eq:(3.7)}), and define 
        \begin{equation}
            \label{eq:(3.11)}
            U(x,i;\theta):=K_{i}(\theta)+\int_{\alpha_{i}(\theta)}^{x}v(y,i;\theta)dy,\quad (x,i,\theta)\in\mathcal{O}\times\mathbb{R}_{+},
        \end{equation}            
%            \quad \tilde{U}(x,i;\theta):=-\int_{x}^{\beta_{i}(\theta)}v(y,i;\theta)dy,\;  
        with $K_{i}(\theta):=-\int_{\alpha_{i}(\theta)}^{\beta_{i}(\theta)}v(y,i;\theta)dy$. Then we have
        \begin{equation}
            \label{eq:(3.12)}
            \begin{cases}
                \mathcal{L}_{(X,Y)}U(x,i;\theta)+\pi(x,\theta)=\lambda(\theta,i),\quad x\in(\alpha_{i}(\theta),\beta_{i}(\theta)),\quad i\in\mathbb{Y}, \\
                \mathcal{L}_{(X,Y)}U(x,i;\theta)+\pi(x,\theta)\leq \lambda(\theta,i),\quad x\notin (\alpha_{i}(\theta),\beta_{i}(\theta)),\quad i\in\mathbb{Y}, \\
                U_{x}(x,i;\theta)=k_{1},\quad x\in (\underline{x},\alpha_{i}(\theta)],\quad \quad \quad i\in\mathbb{Y}, \\
                U_{x}(x,i;\theta)\leq k_{1},\quad x \in (\alpha_{i}(\theta),\beta_{i}(\theta)),\quad \; i\in\mathbb{Y}, \\
                U_{x}(x,i;\theta)=k_{2},\quad x\in [\beta_{i}(\theta),\overline{x}),\quad \quad \quad i\in\mathbb{Y}, \\
                U_{x}(x,i;\theta)\geq k_{2},\quad x\in (\alpha_{i}(\theta),\beta_{i}(\theta)),\quad \; i\in\mathbb{Y},
            \end{cases}
        \end{equation}
        where
        \begin{equation}
            \label{eq:(3.13)}
            \lambda(\theta,i):=\sum_{j\neq i}q_{ij} \bigg(\int_{\alpha_{j}(\theta)}^{\alpha_{i}(\theta)}v(y,j;\theta)dy+K_{i}(\theta)-K_{j}(\theta)\bigg)+\big(b(\alpha_{i}(\theta),i)k_{1}+\pi(\alpha_{i}(\theta),\theta) \big).
        \end{equation}

\end{proposition}

\begin{proof}
        By Proposition \ref{eq:Prop 3.3} and the definition of $U$, we obtain that $U(\cdot,i;\theta)\in C^{2}(\mathcal{I})$ for $i\in\mathbb{Y}$. Then, recalling $\mathcal{L}_{(X,Y)}$ as in (\ref{eq:(2.4)}), we have
        \[
            \mathcal{L}_{(X,Y)}U(x,i;\theta)+\pi(x,\theta)=\underbrace{\big(\frac{1}{2}\sigma^{2}(x,i)v_{x}(x,i;\theta)+b(x,i)v(x,i;\theta)+\pi(x,\theta)\big)}_{(1)}
        \]
        \begin{equation}
            \label{eq:(3.14)}
            +\sum_{j\neq i}q_{ij}\bigg(K_{j}(\theta)+\int_{\alpha_{j}(\theta)}^{x}v(y,j;\theta)dy-K_{i}(\theta)-\int_{\alpha_{i}(\theta)}^{x}v(y,i;\theta)dy\bigg).
        \end{equation}
        Using that $v$ satisfies (\ref{eq:(3.10)}) (see in particular the fourth to seventh display equations), via an integration by parts we have that
        $$
        \begin{aligned}
            (1)=  \int_{\alpha_{i}(\theta)}^{x}\bigg( & \frac{1}{2}\sigma^{2}(y,i)v_{xx}(y,i;\theta) 
            +( b(y,i)+\sigma\sigma_{x}(y,i))v_{x}(y,i;\theta) \\
             & +b_{x}(y,i)v(y,i;\theta)+\pi_{x}(y,\theta)\bigg)dy +  b(\alpha_{i}(\theta),i)k_{1}+\pi(\alpha_{i}(\theta),\theta).
        \end{aligned}
        $$
        Then, using the second equation in (\ref{eq:(3.10)}) we find
        \[
            (1)\leq \int_{\alpha_{i}(\theta)}^{x}-\sum_{j\neq i}q_{ij}\big(v(y,j;\theta)-v(y,i;\theta)\big)dy+\big(b(\alpha_{i}(\theta),i)k_{1}+\pi(\alpha_{i}(\theta),\theta) \big).
        \]
        Substituting the latter inequality into (\ref{eq:(3.14)}) yields
        \begin{align*}
           &\mathcal{L}_{(X,Y)}U(x,i;\theta)+\pi(x,\theta) \\
            &=\sum_{j\neq i}q_{ij}\bigg( \int_{\alpha_{i}(\theta)}^{x}v(y,j;\theta)dy-\int_{\alpha_{j}(\theta)}^{x}v(y,j;\theta)dy \bigg)+  b(\alpha_{i}(\theta),i)k_{1}+\pi(\alpha_{i}(\theta),\theta)  \\
            &=\sum_{j\neq i}q_{ij}\bigg( \int_{\alpha_{j}(\theta)}^{\alpha_{i}(\theta)}v(y,j;\theta)dy\bigg)+  b(\alpha_{i}(\theta),i)k_{1}+\pi(\alpha_{i}(\theta),\theta).
        \end{align*}
        Thus, defining
        \begin{equation*}
            \lambda(\theta,i):=\sum_{j\neq i}q_{ij}\bigg( \int_{\alpha_{j}(\theta)}^{\alpha_{i}(\theta)}v(y,j;\theta)dy\bigg)+ b(\alpha_{i}(\theta),i)k_{1}+\pi(\alpha_{i}(\theta),\theta) 
        \end{equation*}
        \begin{equation}
            +\sum_{j\neq i}q_{ij}\bigg( \int_{\alpha_{i}(\theta)}^{\beta_{i}(\theta)}v(y,i;\theta)dy-\int_{\alpha_{j}(\theta)}^{\beta_{j}(\theta)}v(y,j;\theta)dy\bigg),
        \end{equation}
        \newline
        we see that $(U,\lambda)$ solves the variational inequality (\ref{eq:(3.12)}).
\end{proof}

 In order to derive the optimal control rule we introduce the following definition. We denote by $\mathcal{D}[0,\infty)$ the space of c\`adl\`ag functions on $[0,\infty)$, which we equip with the Skorokhod topology.
\begin{definition}[Skorokhod reflection regime-switching problem]
    \label{eq:Def 3.1}
   Given $x\in\mathcal{I}$, $i\in\mathbb{Y}$, $X\in\mathcal{D}[0,\infty)$, and a vector $(\alpha,\beta) = (\alpha_j, \beta_j)_{j \in \mathbb Y} \in \mathbb R ^{2d}$ with
   $\alpha_{j}<\beta_{j},\text{ for any }j\in\mathbb{Y}$, the process $(X^{\xi},\xi)\in\mathcal{D}[0,\infty)\times \mathcal{A}$ is said to be a solution to the Skorokhod refection problem $\mathbf{SP}(\alpha,\beta,x,i)$ for the noise $(W,Y)$ if it satisfies the following properties:
   \begin{enumerate}
       \item  \label{eq:Def 3.1-1} Letting $I:\mathcal{D}[0,\infty)\to\mathcal{D}[0,\infty)$ be such that
       \begin{equation}
       \label{eq:(3.18)}
           (X)_{t\geq 0}\mapsto \bigg(x+\int_{0}^{t}b(X_{s},Y_{s})ds+\int_{0}^{t}\sigma(X_{s},Y_{s})dW_{s}\bigg)_{t\geq 0},
       \end{equation}
       then
       \begin{equation}
       \label{eq:(3.19)}
           X^{\xi}_{t}=I(X^{\xi})_{t}+\xi_{t}^{+}-\xi^{-}_{t},\; \mathbb{P}\otimes dt\text{-a.s.};
       \end{equation}
       \item  \label{eq:Def 3.1-2} $X_{t}^{\xi}\in[\alpha_{Y_{t}},\beta_{Y_{t}}],\quad\mathbb{P}\otimes dt\text{-a.s.}$;
       \item \label{eq:Def 3.1-3} $\int_{0}^{T}\boldsymbol{1}_{\{X^{\xi}_{t}> \alpha_{Y_{t}} \}}d\xi^{+}_{t}=\int_{0}^{T}\boldsymbol{1}_{\{X^{\xi}_{t}< \beta_{Y_{t}}\}}d\xi^{-}_{t}=0,\text{ for any }T>0,\; \mathbb{P}\text{-a.s.}$
   \end{enumerate}
   
\end{definition}

We then have the following result. Its proof is somewhat classical (it is based on a Picard iteration). 

\begin{theorem}
    \label{Thm 3.2}
    There exists a unique solution $\xi^{*}(\theta)\in\mathcal{A}_{e}$ to $\mathbf{SP}(\alpha(
    \theta),\beta(\theta),x,i)$. Moreover, $\xi^{*}$ admits decomposition $\xi^{*}(\theta)=\xi^{*,+}(\theta)-\xi^{*,-}(\theta)$, where
    \begin{gather}
        \xi^{*,+}_{t}(\theta):=\sup_{0\leq s\leq t}\bigg( \alpha_{Y_{s}}(\theta)-I(X^{\xi^{*}(\theta)})_{s}+\xi^{*,-}_{s}(\theta) \bigg)^{+},\nonumber \\ 
        \label{eq:(3.20)} \xi^{*,-}_{t}(\theta):=\sup_{0\leq s\leq t}\bigg( I(X^{\xi^{*}(\theta)})_{s}+\xi^{*,+}_{s}(\theta)-\beta_{Y_{s}}(\theta) \bigg)^{+}.
    \end{gather}
\end{theorem}
\begin{proof}
    See Appendix \ref{Proof of Theorem 3.2}.
\end{proof}
We now prove that $\xi^{*}(\theta)$ solving $\mathbf{SP}(\alpha(
    \theta),\beta(\theta),x,i)$ is optimal for (\ref{eq:(3.1)}). 
\begin{proposition}
\label{Prop  3.5}
    For every $(x,i)\in \mathcal{O}$, the solution $\xi^{*}(\theta)=\xi^{*,+}(\theta)-\xi^{*,-}(\theta)$  to $\mathbf{SP}(\alpha( \theta),\beta(\theta),x,i)$ is optimal for (\ref{eq:(3.1)}). Moreover, it holds
    \begin{equation}
        \overline{\lambda}(\theta)=\sum_{i\in\mathbb{Y}}p(i)\lambda(\theta,i)
    \end{equation}
    with $\lambda$ as in (\ref{eq:(3.13)}).
\end{proposition}

\begin{proof}
     Let $T>0$ and $(x,i)\in\mathcal{O}$. Recall that $U$ as in (\ref{eq:(3.11)}) is such that $U(\cdot,i)\in C^{2}(\mathcal{I}),\; i\in\mathbb{Y}$, and $(U,\lambda)$ solves (\ref{eq:(3.12)}), with $\lambda$ as in (\ref{eq:(3.13)}). Fixing $\xi\in\mathcal{A}_{e}$, we introduce the sequence of $\mathbb{F}$-stopping times $\{\eta_{n}\}_{n\in\mathbb{N}}$ such that $\eta_{n}:=\inf\{t\geq0:\big|X^{\xi}_{t}\big|>n\}$. Applying It\^o-Meyer's formula (in the sense of \cite{Bjork}) to $(U(X_{T\wedge \eta_{n}},Y_{T\wedge \eta_{n}};\theta))_{T\geq 0}$, we have:
     \begin{equation}
     \label{eq:(3.21)}
        U(X_{T\wedge \eta_{n}}^{\xi},Y_{T\wedge \eta_{n}};\theta)=U(x,i;\theta)+\int_{0}^{T\wedge \eta_{n}}\mathcal{L}_{(X,Y)}U(X_{s}^{\xi},Y_{s};\theta)ds
     \end{equation}
     \begin{equation*}
        +\int_{0}^{T\wedge \eta_{n}}U_{x}(X^{\xi}_{s},Y_{s};\theta)\sigma(X^{\xi}_{s},Y_{s})dW_{s}+\int_{0}^{T\wedge \eta_{n}}U_{x}(X_{s}^{\xi},Y_{s};\theta)
        (d\xi^{+,c}_{s}-d\xi^{-,c}_{s})
     \end{equation*}
     \begin{equation*}
         +\sum_{0\leq s\leq T\wedge \eta_{n},s\in \Lambda^{\xi}}\big(U(X_{s^{+}}^{\xi},Y_{s};\theta)-U(X_{s}^{\xi},Y_{s};\theta)\big)+M_{T\wedge \eta_{n}}-M_{0}
     \end{equation*}
     where $\Lambda^{\xi}:=\{t\geq 0:\xi_{t+}\neq \xi_{t}\}$, $\xi^{\pm,c}$ denotes the continuous part of $\xi^{\pm}$, and
     \begin{equation}
         M_{t}:=\int_{0}^{t}\int_{\mathbb{R}_{+}}\big(U(X_{s}^{\xi},h(Y_{s},z)+Y_{0};\theta)-U(X_{s}^{\xi},Y_{s};\theta)\big)\nu(ds,dz)
     \end{equation}
     is an $\mathbb{F}$-local martingale, where $\nu(dt,dz):=\wp(dt,dz)-dt\otimes m(dz)$. In particular, following the construction in \cite{yin2009hybrid}, $\wp(dt,dz)$ denotes the Poisson random measure with intensity $dt\otimes m(dz)$ for Lebesgue measure $m$, and the function $h:\mathbb{Y}\times \mathbb{R}_{+}$ has the form
     $
            h(i,z)=\sum_{j=1}^{d}(j-i)\boldsymbol{1}_{\{z\in\Delta_{i,j}\}}, 
     $
     with $\Delta_{i,j}$ being the partition of the $\mathbb{R}_{+}$ where each interval has length $q_{ij}$. We then define the $\mathbb{F}$-local martingale $\{\widetilde{M}_{t}\}_{t\geq0}$ as
     \begin{equation}
         \widetilde{M}_{t}:=M_{t}+\int_{0}^{t}U_{x}(X^{\xi}_{s},Y_{s};\theta)\sigma(X^{\xi}_{s},Y_{s})dW_{s}
     \end{equation}
     and notice that $\mathbb{E}_{(x,i)}\big[ \widetilde{M}_{T\wedge \eta_{n}} \big]=0,\forall T\geq 0,n\in\mathbb{N}$. Then, taking expectations in (\ref{eq:(3.21)}), we obtain
     \begin{align*}
       \mathbb{E}_{(x,i)}\big[ U(X_{T\wedge \eta_{n}}^{\xi},Y_{T\wedge \eta_{n}};\theta) \big]=
       & U(x,i;\theta)+\mathbb{E}_{(x,i)}\bigg[ \int_{0}^{T\wedge \eta_{n}}\mathcal{L}_{(X,Y)}U(X_{s}^{\xi},Y_{s};\theta)ds \bigg]   \\
        &+\mathbb{E}_{(x,i)}\bigg[\int_{0}^{T\wedge \eta_{n}}U_{x}(X_{s}^{\xi},Y_{s};\theta)
        (d\xi^{+,c}_{s}-d\xi^{-,c}_{s})\bigg] \\
        & +\mathbb{E}_{(x,i)}\bigg[\sum_{0\leq s\leq T\wedge \eta_{n},s\in \Lambda^{\xi}}\big(U(X_{s^{+}}^{\xi},Y_{s};\theta)-U(X_{s}^{\xi},Y_{s};\theta)\big)\bigg].
     \end{align*}
     Employing the second display inequality in (\ref{eq:(3.12)}), we find
     \begin{align}
        \label{eq:(3.24)}
        \mathbb{E}_{(x,i)}\big[ U(X_{T\wedge \eta_{n}}^{\xi},Y_{T\wedge \eta_{n}};\theta) \big]\leq
        & U(x,i;\theta)+\mathbb{E}_{(x,i)}\bigg[ \int_{0}^{T\wedge \eta_{n}}\big(\lambda(\theta,Y_{s})-\pi(X_{s}^{\xi},\theta)\big)ds \bigg] \\ \notag
        & +\mathbb{E}_{(x,i)}\bigg[\int_{0}^{T\wedge \eta_{n}}U_{x}(X_{s}^{\xi},Y_{s};\theta)
        (d\xi^{+,c}_{s}-d\xi^{-,c}_{s})\bigg] \\ \notag
        & +\mathbb{E}_{(x,i)}\bigg[\sum_{0\leq s\leq T\wedge \eta_{n},s\in \Lambda^{\xi}}\big(U(X_{s+}^{\xi},Y_{s};\theta)-U(X_{s}^{\xi},Y_{s};\theta)\big)\bigg]. \notag
     \end{align}
     Noticing that
     \begin{align*}
     U(X^{\xi}_{s^{+}},Y_{s};\theta)-U(X^{\xi}_{s},Y_{s};\theta)=
     &\boldsymbol{1}_{\{\Delta \xi^{+}_{s}>0\}}\int_{0}^{\Delta \xi^{+}_{s}}U_{x}(X^{\xi}_{s}+z,Y_{s};\theta)dz \\
     & -\boldsymbol{1}_{\{\Delta \xi^{-}_{s}>0\}}\int_{0}^{\Delta \xi^{-}_{s}}U_{x}(X^{\xi}_{s}-z,Y_{s};\theta)dz,
     \end{align*}
    using that $k_2 \leq U_x \leq k_1$ (cf.\ fourth and sixth conditions in (\ref{eq:(3.12)})) in the latter display equation, and then using the obtained inequality in \eqref{eq:(3.24)}, we obtain
     \begin{align}
        \label{eq:(3.25)}
        \mathbb{E}_{(x,i)}\big[ U(X_{T\wedge \eta_{n}}^{\xi},Y_{T\wedge \eta_{n}};\theta) \big]\leq
        & U(x,i;\theta)+\mathbb{E}_{(x,i)}\bigg[ \int_{0}^{T\wedge \eta_{n}}\big(\lambda(\theta,Y_{s})-\pi(X_{s}^{\xi},\theta)\big)ds \bigg] \\ \notag
        & +\mathbb{E}_{(x,i)}\bigg[ \int_{0}^{T\wedge \eta_{n}}(k_{1}d\xi^{+}_{t}-k_{2}d\xi^{-}_{t})\bigg]. \notag
     \end{align}
     
     If now the dominated convergence theorem can be applied (we will verify this later), taking limits as $n\uparrow\infty$ in (\ref{eq:(3.25)}) we find
     \begin{equation}
            \label{eq:(3.26)}
            \mathbb{E}_{(x,i)}\big[ U(X_{T}^{\xi},Y_{T};\theta) \big]\leq U(x,i;\theta)+\mathbb{E}_{(x,i)}\bigg[ \int_{0}^{T}\bigg(\lambda(\theta,Y_{s})-\pi(X_{s}^{\xi},\theta)\bigg)ds +k_{1}\xi^{+}_{T}-k_{2}\xi^{-}_{T} \bigg].
     \end{equation}

     Rearranging the terms, dividing by $T$, and sending $T\uparrow \infty$, we conclude that
     \begin{equation}
        \label{eq:(3.27)}
        \limsup_{T\uparrow\infty}\bigg(\frac{1}{T}\mathbb{E}_{(x,i)}\big[ U(X^{\xi}_{T},Y_{T};\theta)\big]+\frac{1}{T}\mathbb{E}_{(x,i)}\bigg[ \int_{0}^{T}\pi(X^{\xi}_{t},\theta)dt-k_{1}\xi^{+}_{T}+k_{2}\xi^{-}_{T}\bigg] \bigg)
     \end{equation}
     \begin{equation*}
         \leq \limsup_{T\uparrow\infty}\bigg(\frac{1}{T}U(x,i;\theta)+\frac{1}{T}\mathbb{E}_{(x,i)}\bigg[ \int_{0}^{T}\lambda(\theta,Y_{t})dt \bigg]\bigg),
     \end{equation*}
     which implies 
    \begin{equation}
        \label{eq:(3.28)}
        \limsup_{T\uparrow \infty}\bigg(\frac{1}{T} \mathbb{E}_{(x,i)}\big[ U(X_{T}^{\xi},Y_{T};\theta) \big]+\frac{1}{T}\mathbb{E}_{(x,i)}\bigg[ \int_{0}^{T}\pi(X_{s}^{\xi},\theta)ds-k_{1}\xi^{+}_{T}+k_{2}\xi^{-}_{T} \bigg]\bigg)
    \end{equation}
    \begin{equation*}
        \leq \limsup_{T\uparrow \infty}\frac{1}{T} U(x,i;\theta)+\limsup_{T\uparrow \infty}\frac{1}{T}\mathbb{E}_{(x,i)}\bigg[ \int_{0}^{T}\lambda(\theta,Y_{s})ds\bigg].
    \end{equation*}
    Hence, using the inequality $\liminf_{n\to\infty}(f_{n})+\limsup_{n\to\infty}(g_{n})\leq \limsup_{n\to\infty}(f_{n}+g_{n})$,
     on the left-hand side of (\ref{eq:(3.28)}), we find that
     \begin{equation}\label{eq:(3.29)}
     \begin{aligned}
         \liminf_{T\uparrow \infty}\frac{1}{T} & \mathbb{E}_{(x,i)}\big[ U(X_{T}^{\xi},Y_{T};\theta) \big]+\limsup_{T\uparrow\infty}\frac{1}{T}\mathbb{E}_{(x,i)}\bigg[ \int_{0}^{T}\pi(X_{s}^{\xi},\theta)ds-k_{1}\xi^{+}_{T}+k_{2}\xi^{-}_{T} \bigg] \\
        & \leq \limsup_{T\uparrow \infty}\frac{1}{T} U(x,i;\theta)+\limsup_{T\uparrow \infty}\frac{1}{T}\mathbb{E}_{(x,i)}\bigg[ \int_{0}^{T}\lambda(\theta,Y_{s})ds\bigg].
         \end{aligned}
     \end{equation}
    Since $U_{x}\in [k_{2},k_{1}]$ by (\ref{eq:(3.12)}), we have that there exists $C>0$ such that for every $(x,i,\theta)\in\mathcal{O}\times \mathbb{R}_{+}$ one has
    \begin{equation}
        \label{eq:(3.30)}
        \big|U(x,i,\theta)\big|\leq C(1+|x|). 
    \end{equation}
    This fact gives
    \begin{equation}
        \label{eq:(3.31)}
        \lim_{T\uparrow\infty}\frac{1}{T}U(x,i;\theta)=0,
    \end{equation}
    and, thanks to the fact that $\xi\in\mathcal{A}_{e}$ (cf.\ (\ref{admissible controls})), we also find
    \begin{align} 
    \label{eq:(3.32)}
        \liminf_{T\uparrow \infty}\frac{1}{T} \mathbb{E}_{(x,i)}\big[ U(X_{T}^{\xi},Y_{T};\theta) \big] & \geq \liminf_{T\uparrow\infty}\frac{1}{T}\mathbb{E}_{(x,i)}[-C(1+|X^{\xi}_{T}|)] \\ \notag
        &=-C\limsup_{T\uparrow\infty}\frac{1}{T}\mathbb{E}_{(x,i)}\big[ |X^{\xi}_{T}| \big]=0, \notag
    \end{align}
    as well as,
    \begin{align}
    \label{eq:(3.32-1)}
        \limsup_{T\uparrow \infty}\frac{1}{T} \mathbb{E}_{(x,i)}\big[ U(X_{T}^{\xi},Y_{T};\theta) \big]& \leq C\limsup_{T\uparrow\infty}\frac{1}{T}\mathbb{E}_{(x,i)}[C(1+|X^{\xi}_{T}|)]\\ \notag
        &\leq \limsup_{T\uparrow\infty}\frac{1}{T}\mathbb{E}_{(x,i)}[|X^{\xi}_{T}|]=0. \notag
    \end{align}
   Hence, using (\ref{eq:(3.31)}) and (\ref{eq:(3.32)}) in (\ref{eq:(3.29)}), we obtain
    \begin{align}
         \label{eq:(3.33)}
        & \limsup_{T\uparrow\infty}\frac{1}{T}\mathbb{E}_{(x,i)}\bigg[ \int_{0}^{T}\pi(X_{s}^{\xi},\theta)ds-k_{1}\xi^{+}_{T}+k_{2}\xi^{-}_{T} \bigg] \\ \notag 
        & \quad \leq    
         \limsup_{T\uparrow \infty}\frac{1}{T}\mathbb{E}_{(x,i)}\bigg[ \int_{0}^{T}\lambda(\theta,Y_{s})ds\bigg]=\sum_{i\in\mathbb{Y}}\lambda(\theta,i)p(i). \notag
    \end{align}
    Here, the last equality is due to the ergodicity of $Y$
    (see, for instance, Corollary 25.9 in 
    \cite{kallenberg2002foundations}). Since (\ref{eq:(3.33)}) holds for every $\xi\in\mathcal{A}_{e}$, we get 
    \begin{equation}
        \label{eq:(3.34)}
        \overline{\lambda}(\theta)=\sup_{\xi\in\mathcal{A}_{e}}J(x,i;\xi,\theta)\leq \sum_{i\in\mathbb{Y}}p(i)\lambda(\theta,i).
    \end{equation}
    Letting now $\xi^{*}(\theta)$ be the solution to $\mathbf{SP}(\alpha(
    \theta),\beta(\theta),x,i)$, we notice that in (\ref{eq:(3.25)}), and thus in (\ref{eq:(3.27)}), the inequality becomes an equality. Then, using that $\limsup_{n\uparrow\infty}(f_{n}+g_{n})\leq \limsup_{n\uparrow\infty}(f_{n})+\limsup_{n\uparrow\infty}(g_{n})$ on the left-hand side of (\ref{eq:(3.27)}) (with equality) and $\liminf_{n\uparrow\infty}(f_{n})+\limsup_{n\uparrow\infty}(g_{n})\leq \limsup_{n\uparrow\infty}(f_{n}+g_{n})$ on the right-hand side of (\ref{eq:(3.27)}) (with equality) we obtain
    \begin{equation}
        \label{eq:(3.35)}
        \begin{aligned}
        & \limsup_{T\uparrow\infty}\frac{1}{T}\mathbb{E}_{(x,i)}\bigg[ \int_{0}^{T}\pi(X_{s}^{\xi^{*}(\theta)},\theta)ds-k_{1}\xi^{*,+}_{T}(\theta)+k_{2}\xi^{*,-}_{T}(\theta) \bigg]\\
        & + \limsup_{T\uparrow\infty}\frac{1}{T}\mathbb{E}_{(x,i)}\big[ U(X_{T}^{\xi^{*}(\theta)},Y_{T};\theta) \big]\\
        & \geq \liminf_{T\uparrow\infty}\frac{1}{T}U(x,i;\theta)+\limsup_{T\uparrow \infty}\frac{1}{T}\mathbb{E}_{(x,i)}\bigg[ \int_{0}^{T}\lambda(\theta,Y_{s})ds\bigg].
        \end{aligned}
    \end{equation}
    By means of (\ref{eq:(3.31)}), (\ref{eq:(3.32-1)}) and the fact that $\xi^{*}(\theta)\in\mathcal{A}_{e}$ we thus conclude from (\ref{eq:(3.35)})
    $$
    \begin{aligned}
        &\limsup_{T\uparrow\infty}\frac{1}{T}\mathbb{E}_{(x,i)}\bigg[ \int_{0}^{T}\pi(X_{s}^{\xi^{*}(\theta)},\theta)ds-k_{1}\xi^{*,+}_{T}(\theta)+k_{2}\xi^{*,-}_{T}(\theta) \bigg] \\
        &\geq \limsup_{T\uparrow \infty}\frac{1}{T}\mathbb{E}_{(x,i)}\bigg[ \int_{0}^{T}\lambda(\theta,Y_{s})ds\bigg].
    \end{aligned}
    $$
    
    This, together with (\ref{eq:(3.34)}), provides the claimed optimality of $\xi^{*}(\theta)$ and that
    \begin{equation}
        \overline{\lambda}(\theta)=\sum_{i\in\mathbb{Y}}\lambda(\theta,i)p(i).
    \end{equation}
    In order to complete the proof, it remains to show that expectations and limits as $n\uparrow\infty$ can be interchanged in (\ref{eq:(3.25)}). By writing (\ref{eq:(3.25)}) as 
    \[
        \mathbb{E}_{(x,i)}\big[ U(X^{\xi}_{T\wedge \eta_{n}},Y_{T\wedge \eta_{n}};\theta) \big]+\mathbb{E}_{(x,i)}\bigg[ \int_{0}^{T\wedge \eta_{n}}\big(\pi(X^{\xi}_{s},\theta(Y_{s}))ds+k_{2}d\xi^{-}_{t}\big)\bigg]
    \]
    \[
        \leq U(x,i;\theta)+ \mathbb{E}_{(x,i)}\bigg[\int_{0}^{T\wedge \eta_{n}}\big(\lambda(\theta,Y_{s})ds+\kappa_{1}d\xi^{+}_{s}\big)\bigg],
    \]
    it is easy to see that the limit as $n\uparrow\infty$ can be interchanged with the expectations of the integral terms by the monotone convergence theorem. We are thus left to show that
    \begin{equation}
        \label{eq:(3.39)}
        \lim_{n\uparrow\infty}\mathbb{E}_{(x,i)}\big[ U(X^{\xi}_{T\wedge \eta_{n}},Y_{T\wedge \eta_{n}};\theta) \big]=\mathbb{E}_{(x,i)}\big[ U(X^{\xi}_{T},Y_{T};\theta) \big].
    \end{equation}
    By (\ref{eq:(3.30)}), $|U(X^{\xi}_{T\wedge \eta_{n}},Y_{T\wedge \eta_{n}};\theta)|\leq C(1+|X^{\xi}_{T\wedge\eta_{n}}|)\leq C(1+\sup_{t\leq T}|X^{\xi}_{t}|)$.
    Hence, by standard estimates based on the Lipschitz property of $b(\cdot,i)$ and $\sigma(\cdot,i)$ and the fact that $\xi\in\mathcal{A}_{e}$, we conclude that $\mathbb{E}_{(x,i)}\big[\sup_{t\leq T}|X^{\xi}_{t}|\big]<\infty$, thus allowing the application of the dominated convergence theorem in order to derive (\ref{eq:(3.39)}).
\end{proof}

\section{Mean-Field Equilibrium}
\label{Mean field analysis}

In the following, we prove the existence and uniqueness of the mean-field equilibrium (cf.\ Definition \ref{eq: Def 1}) by an application of Schauder-Tychonof fixed-point theorem. Let $
\mathcal{P}(\mathcal{O})$ be the space of probability measures on $\mathcal{O}$, endowed with the weak topology.

\subsection{Continuity of the free-boundaries w.r.t.\ the mean-field term}
In this subsection we establish some regularity of the free-boundaries with respect to the variable $\theta$. For the subsequent analysis we introduce the following assumptions.
\begin{assumption}
    \label{Ass 4.1}
    The following hold.
    \begin{enumerate}
        \item \label{Ass 4.1-1} There exists $\kappa:\mathcal{I} \to\mathbb{R}_{+}$ such that
        \[
            \lim_{\theta\downarrow 0}\pi_{x}(x,\theta)=\infty,\quad \lim_{\theta\uparrow \infty}\pi_{x}(x,\theta)=\kappa(x),\;\text{ for any }x\in\mathcal{I}
        \]
        and, for some $\widehat{x}_{-}:=(\widehat{x}_{-}(1),...,\widehat{x}_{-}(d))$ and $\widehat{x}_{+}:=(\widehat{x}_{+}(1),...,\widehat{x}_{+}(d))$, we have 
        \[        
            \kappa(x)+k_{1}b_{x}(x,i)
            \begin{cases}
                                            >0, \quad x<\widehat{x}_{-}(i), \\
                                                        =0,\quad x=\widehat{x}_{-}(i), \\
                                                        <0,\quad x<\widehat{x}_{-}(i),
            \end{cases}
            \quad \kappa(x)+k_{2}b_{x}(x,i)
            \begin{cases}
                                                        >0,\quad x<\widehat{x}_{+}(i), \\
                                                        =0,\quad x=\widehat{x}_{+}(i), \\
                                                        <0,\quad x<\widehat{x}_{+}(i).
            \end{cases}
        \]
       \item \label{Ass 4.1-2} It holds that $b_{xx}(x,i) \leq 0$  for any $(x,i)\in\mathcal{O}$, and, for any compact  $\mathcal X \times \Theta  \subset  \mathcal I \times \mathbb R _+ $, there exists a constant $C_{(\mathcal X, \Theta)} > 0$ such that
       $$
            \widehat{\mathbb{E}}_{(x,i)}\bigg[ \int_{0}^{\infty}e^{-ct}\big| \pi_{x\theta}(\widehat{X}_{t},\theta) \big|dt\bigg] < C_{(\mathcal X, \Theta)} <\infty, 
        $$
        for any $(x,i, \theta)\in\mathcal{X} \times \mathbb Y \times \Theta$, with $c$ as in Assumption \ref{eq:Ass 2.1}.
       \item For any $i\in\mathbb{Y}$ we have
        \label{Ass 4.1-3}
        $
            \sum_{j\in\mathbb{Y}}q_{ji}\leq 0.
        $
\end{enumerate}
\end{assumption}
\begin{remark}
    As an example, as in Remark \ref{remark: benchmark example} consider a profit function
        $\pi(x,\theta)=x^{\beta}(\theta^{-(1+\beta)}+\kappa_*),$ for $  \kappa_*>0, \beta\in (0,1)$, and dynamics
     $b(x,i)=-\delta_{i}x$ and $\sigma(x,i)=\sigma_{i}x$, for constants $\delta_{i}, \sigma_{i}>0$  for any $i\in\mathbb{Y}$. 
     By choosing $\kappa(x)=\kappa_* \beta x^{\beta  -1 }$, both Assumptions \ref{eq:Ass 2.2} and \ref{Ass 4.1} are satisfied
     (see also Remark \ref{remark: benchmark example}).
     Notice that the case $\kappa_*=0$ fulfills Assumption \ref{eq:Ass 3.1}, but it does not satisfy Condition \ref{Ass 4.1-1} in Assumption \ref{Ass 4.1}.
\end{remark}

\begin{lemma}
    \label{Lemma 4.1}
    For any $i\in\mathbb{Y}$, the maps $\theta\mapsto \alpha_{i}(\theta)$ and $\theta\mapsto \beta_{i}(\theta)$ are nonincreasing.
\end{lemma}

\begin{proof}
    Take arbitrary $i\in\mathbb{Y}$ and let $\theta_{1},\theta_{2}\in\mathbb{R}_{+}$ such that $\theta_{1}\leq \theta_{2}$. Then, by Proposition \ref{eq:Prop 3.1}, we have that the map $\theta\mapsto v(x,i;\theta)$ is decreasing, for any $(x,i)\in\mathcal{O}$. Hence,
    \[
        \{x\in\mathcal{I}:v(x,i;\theta_{2})\geq k_{1}\} \subseteq \{x\in \mathcal{I}:v(x,i;\theta_{1})\geq k_{1}\},
    \]
    which, by (\ref{eq:(3.7)}), implies that
    \[
        \alpha_{i}(\theta_{2})\leq \alpha_{i}(\theta_{1}).
    \]
    The monotonicity of $\theta\mapsto \beta_{i}(\theta),\; i\in\mathbb{Y}$, can be proved similarly.
\end{proof}

\begin{lemma}
    \label{lemma signs}
    The following hold.
    \begin{enumerate}
        \item \label{Lipschitz con on theta} For any compact $\mathcal X$, the map $\theta\mapsto v(x,i;\theta)$ is locally Lipschtiz continuous, uniformly for $(x,i)\in\mathcal{X} \times \mathbb Y$.
        \item \label{Strictly neg} The function $v_{x}(x,i;\theta)$ is nonpositive for any $(x,i,\theta)\in\mathcal{O}\times\mathbb{R}_{+}$, and, for fixed $\theta\in\mathbb{R}_{+}$, it is strictly negative for any $(x,i)\in\mathcal{C}^{\theta}$, where $\mathcal{C}^{\theta}$ is the continuation region as in Theorem \ref{eq:Thm 3.1}.
    \end{enumerate}
\end{lemma}

\begin{proof}
    Take a compact $\mathcal X \times \Theta \subset \mathcal I \times \mathbb R _+$ and generic $(x,i,\theta)\in\mathcal{X} \times \mathbb Y \times \Theta$. Following the proof of Proposition \ref{eq:Prop 3.3} (cf. Appendix \ref{Proof of Proposition 3.3}), for $\epsilon>0$ small enough, by mean value theorem, we have that
    \begin{align*}
        \big|v(x,i;\theta+\epsilon)-v(x,i;\theta)\big|&\leq\widehat{\mathbb{E}}_{(x,i)}\bigg[ \int_{0}^{\infty}e^{\int_{0}^{t}b_{x}(\widehat{X}_{s},Y_{s})ds} \big|\pi_{x}(\widehat{X}_{t},\theta+\epsilon)-\pi_{x}(\widehat{X}_{t},\theta)\big|dt\bigg] \\
        &=\epsilon \, \widehat{\mathbb{E}}_{(x,i)}\bigg[ \int_{0}^{\infty}e^{\int_{0}^{t}b_{x}(\widehat{X}_{s},Y_{s})ds}\big|\pi_{x\theta}(\widehat{X}_{t},\tilde{\theta})\big|dt\bigg] \\
        &\leq \epsilon \, \widehat{\mathbb{E}}_{(x,i)}\bigg[ \int_{0}^{\infty}e^{-ct}\big|\pi_{x\theta}(\widehat{X}_{t},\tilde{\theta})\big|dt\bigg]\leq \epsilon \,  C_{(\mathcal X, \Theta^*)} ,
    \end{align*}
    for $\tilde{\theta}\in(\theta,\theta+\epsilon)$ and a compact $\Theta^*$ such that $\theta + \epsilon \in \Theta^*$ for any $\theta \in \Theta$.
    Thus, since $C_{(\mathcal X, \Theta^*)} < \infty$ by Assumption \ref{Ass 4.1}-(\ref{Ass 4.1-2}), we deduce that the map $\theta\mapsto v(x,i;\theta)$ is locally Lipschtiz continuous, proving Claim (\ref{Lipschitz con on theta}).

    We next prove Claim (\ref{Strictly neg}).
    Using the representation of $v_{x}$ in Lemma \ref{repr vx}, Assumption \ref{eq:Ass 2.2}-(\ref{eq:2.2-2}) and Assumption \ref{Ass 4.1}-(\ref{Ass 4.1-2}), one can easily check
     (by observing that $\partial _x \widehat X^{x}_{t} > 0,\;\widehat{\mathbb{P}}\otimes dt$-a.s., since $\partial _x \widehat X ^x$ is a solution to a linear equation) that the function $v_{x}$ is nonpositive for any $(x,i,\theta)\in\mathcal{O}\times \mathbb{R}_{+}$.
    For arbitrary fixed $\theta\in\mathbb{R}_{+}$ and for any $(x,i)\in\mathcal{C}^{\theta}$ we know that $\tau^{*}(x,i;\theta)\wedge\sigma^{*}(x,i;\theta)>0,\; \widehat{\mathbb{P}}\text{-a.s.}$, hence by representation of $v_{x}$ (cf. Lemma \ref{repr vx}) we conclude that $v_{x}(x,i,;\theta)<0$. 
\end{proof}

The next proposition establishes continuity of the free boundaries with respect to the parameter $\theta$, a result which will play an important role in the proof of Theorem \ref{Thm: existence MFGE} below, when applying the Schauder-Tychonoff fixed point theorem. The result is interesting per se since it proves continuity with respect to a parameter (not a state variable) of the free boundaries arising in a Dynkin game.
\begin{proposition}
    \label{Cont of free-boundaries}
    The map $\theta\mapsto (\alpha_{i}(\theta),\beta_{i}(\theta))_{i\in\mathbb{Y}}$ is continuous.
\end{proposition}

\begin{proof}
    For fixed $i\in\mathbb{Y}$ we, first prove that the map $\theta\mapsto \alpha_{i}(\theta)$ is left-continuous.     
    Fix $\theta\in\mathbb{R}_{+}$, a sequence $\{\theta^{n}\}_{n\in\mathbb{N}}$, and a compact $ \Theta \subset \mathbb{R}_{+}$ such that $\theta^{n}\nearrow \theta$, and $\theta,\theta^{n}\in \Theta$ for any $n\in\mathbb{N}$. 
Since the map $\theta\mapsto\alpha_{i}(\theta)$ is nonincreasing (cf.\ Lemma \ref{Lemma 4.1}), arguing by contradiction, we assume that there exists $\delta>0$ such that $\alpha_{i}(\theta)+\delta \leq \alpha_{i}(\theta^{n})$ for any $n\in\mathbb{N}$. 
Observe that, since the function $\alpha_i$ is finite and monotone (cf.\ Lemma \ref{Lemma 4.1}), one has
\begin{equation}
   \label{eq compact X}
 \alpha_i(\theta^n) \in \Big[ \inf_{y \in \Theta} \alpha_i(y) , \sup_{y \in \Theta} \alpha_i(y) \Big]=:\mathcal{K}, \quad \text{ for any $n$},
\end{equation}
with $\mathcal{K}$ being compact.
Thanks to the fourth display equation in (\ref{eq:(3.10)}) we have that $v(\alpha_{i}(\theta),i;\theta)=v(\alpha_{i}(\theta^{n}),i;\theta^{n})=k_{1}$ for any $n\in\mathbb{N}$. Hence, using Proposition \ref{eq:Prop 3.3} and Lemma \ref{lemma signs} we can write
    \begin{align*}
        0&=v(\alpha_{i}(\theta^{n}),i;\theta^{n})-v(\alpha_{i}(\theta),i;\theta) \\
        &=\big( v(\alpha_{i}(\theta^{n}),i;\theta^{n})-v(\alpha_{i}(\theta^{n}),i;\theta)\big)+\big( v(\alpha_{i}(\theta^{n}),i;\theta)-v(\alpha_{i}(\theta),i;\theta) \big) \notag \\
        &= - \int^{\theta}_{\theta^{n}}v_{\theta}(\alpha_{i}(\theta^{n}),i;y)dy+\int_{\alpha_{i}(\theta)}^{\alpha_{i}(\theta^{n})}v_{x}(y,i;\theta)dy \notag
    \end{align*}
 or,   equivalently,
    \begin{equation}\label{equation in theta}
        \int_{\alpha_{i}(\theta)}^{\alpha_{i}(\theta^{n})}\big(-v_{x}(y,i;\theta)\big)dy= \int_{\theta^{n}}^{\theta}\big(-v_{\theta}(\alpha_{i}(\theta^{n}),i;y)\big)dy.
    \end{equation} 
    By Lemma \ref{lemma signs} we have $v_{x}(x,i;\theta)< 0$ for fixed $\theta\in \Theta$ and for any $(x,i)\in \mathcal{C}^{\theta}$, then for $\epsilon\in (0,\delta)$ small enough we have that
    \begin{align}
        \label{contr left-cont}
        \int_{\alpha_{i}(\theta)}^{\alpha_{i}(\theta^{n})}\big(-v_{x}(y,i;\theta)\big)dy & \geq \int_{\alpha_{i}(\theta)+\epsilon}^{\alpha_{i}(\theta)+\delta}\big(-v_{x}(y,i;\theta)\big)dy \\
        &\geq \min_{y\in [\alpha_{i}(\theta)+\epsilon,\alpha_{i}(\theta)+\delta]}|v_{x}(y,i;\theta)|(\delta-\epsilon)>0, \nonumber
    \end{align}
    and
    \begin{equation}
        \label{contr II left-cont}
        \int_{\theta^{n}}^{\theta}\big(-v_{\theta}(\alpha_{i}(\theta^{n}),i;y)\big)dy\leq \max_{y \in  \Theta}|v_{\theta}(\alpha_{i}(\theta^{n}),i;y)|(\theta-\theta^{n}).
    \end{equation}
    Combining (\ref{contr left-cont}) and (\ref{contr II left-cont}) with (\ref{equation in theta}) we arrive at the inequality,
    \begin{equation*}
        0<\min_{y\in [\alpha_{i}(\theta)+\epsilon,\alpha_{i}(\theta)+\delta]}|v_{x}(y,i;\theta)|(\delta-\epsilon)
        \leq \max_{(x,y)\in \mathcal K\times \Theta} |v_{\theta}(x,i;y)|(\theta-\theta^{n}) \leq C (\theta-\theta^{n}),
    \end{equation*}
    for a constant $C$ not depending on $n$ (thanks to Lemma \ref{lemma signs}). 
    Thus, letting $n\to \infty$ leads to a contradiction.

 We next show that the map $\theta\mapsto \alpha_{i}(\theta)$ is right-continuous.
    Fix a sequence $\{\theta^{n}\}_{n\in\mathbb{N}}\subseteq \Theta$ such that $\theta^{n}\searrow\theta$.
 Again by the monotonicity of the map $\theta\mapsto\alpha_{i}(\theta)$, arguing by contradiction we assume that there exists $\delta>0$ such that $\alpha_{i}(\theta^{n})\leq\alpha_{i}(\theta)-\delta$ for any $n\in\mathbb{N}$. 
   Similarly to the first part of the proof, we find 
    \begin{align*}
        \int_{\theta}^{\theta^{n}}\big(-v_{\theta}(\alpha_{i}(\theta^{n}),i;y)\big)dy=\int_{\alpha_{i}(\theta^{n})}^{\alpha_{i}(\theta)}\big(-v_{x}(y,i;\theta^{n})\big)dy.
    \end{align*}
    which leads to the following inequality
    \begin{equation*}
        0<\min_{y\in [\alpha_{i}(\theta)-\delta,\alpha_{i}(\theta)-\epsilon]}|v_{x}(y,i;\theta)|(\delta-\epsilon)\leq\max_{(x,y) \in \mathcal K \times \Theta}|v_{\theta}(x,i;y)|(\theta^{n}-\theta) \leq C (\theta-\theta^{n}), 
    \end{equation*}
    for  $\mathcal K$ as in \eqref{eq compact X} and a constant $C$ not depending on $n$.
    Thus, letting $n\to \infty$, we obtain a contradiction. 

 By repeating the same argument, one can show that the map $\theta\mapsto \beta_{i}(\theta)$ is continuous, thus completing the proof. 
\end{proof}

\subsection{Analysis of the stationary distribution}

According to the road-map discussed at the end of Section \ref{sec:ErgodicMFG}, having solved the ergodic singular stochastic control problem parametrized by $\theta \in \mathbb{R}_+$, the next step consists in characterizing the stationary distribution of $(X^{\xi^{*}(\theta)}_{t},Y_{t})_{t\geq 0}$. First of all, we show that a stationary distribution for such a process does exist.

\begin{proposition}
\label{Prop 4.1}
     For any $\theta\in\mathbb{R}_{+}$, there exists a unique stationary distribution $\nu^{\theta}\in\mathcal{P}(\mathcal{O})$ for  $(X_{t}^{\xi^{*}(\theta)},Y_{t})_{t\geq 0}$.
\end{proposition}

\begin{proof}
In light of  Theorem 2.1 in \cite{Sigman}, in order to establish the existence and uniqueness of a stationary distribution, it is sufficient to show that $(X^{\xi^{*}(\theta)}_{t},Y_{t})_{t\geq 0}$ is a regenerative process (cf. Definition (D1) in \cite{Sigman}) with finite length regenerative epochs.  
In particular, for $(x,i)\in \mathcal{O}$, setting $\eta :=\inf\{ t>0 : X^{\xi^{*}(\theta)}_{t} = \alpha_{Y_t}(\theta)\},\;\mathbb{P}_{(x,i)}$\text{-a.s.}, by strong Markov property we have that the processes $(X^{\xi^{*}(\theta)}_{t +\eta },Y_{t + \eta})_{t\geq 0}$ and the collection of random variables $ ( (X^{\xi^{*}(\theta)}_{t},Y_{t})_{t < \eta}, \eta)$ are independent.
Moreover, when starting from points in the set $\{ (\alpha_i (\theta) , i) | i \in \mathbb Y \}$, the process  $(X^{\xi^{*}(\theta)}_{t},Y_{t})_{t\geq 0}$ has the same distribution as the process $(X^{\xi^{*}(\theta)}_{t + \eta },Y_{t + \eta})_{t\geq 0}$.
Thus, $(X^{\xi^{*}(\theta)}_{t},Y_{t})_{t\geq 0}$ is a regenerative process, with regenerative epoch $\eta$. 
In order to conclude the proof, it remains to show that $\mathbb{E}_{(x,i)} [\eta] < \infty$, for any $(x,i)\in\mathcal{O}$.

For $(x,i)\in\mathcal{O}$ and $y\in [\min_{j\in\mathbb{Y}}\alpha_{j}(\theta),\max_{j\in\mathbb{Y}}\beta_{j}(\theta)]$, we define the $\mathbb{F}\text{-hitting time}$ $\eta(y):=\inf\{t>0:X^{\xi^{*}(\theta)}_{t}=y\}$.
If $Y_{0}=i$ by comparison principle (see Proposition B.2) we have that $X^{x,\xi^{*}(\theta)}_{t}\leq X^{\beta_{i}(\theta),\xi^{*}(\theta)}_{t},\;\mathbb{P}_{i}\text{-a.s.}$.
Thus, we have
\begin{equation}
    \label{eq estimate eta}
    \mathbb E_{(x,i)} [\eta] \leq  \mathbb E _{(\beta_i(\theta),i)} [ \eta ] \leq \mathbb E _{(\beta_i(\theta),i)} [ \eta (\alpha_i (\theta)) ].
\end{equation}

We now estimate $\mathbb E _{(\beta_i(\theta),i)} [ \eta (\alpha_i (\theta)) ]$. Let $\tau$ be the first jump of the process $Y$, and, denoting by $Z$ the solution to the SDE
\begin{equation}\label{eq SDE Z fixed i}
 dZ_{t}=b(Z_{t},i)dt+\sigma(Z_{t},i)dW_{t}, \quad Z_{0}=\beta_{i}(\theta),
\end{equation}
for $\tilde \eta:=\inf\{t>0:Z_{t}= \alpha_{i}(\theta) \}$, by comparison principle (a slightly version of Proposition B.2) we also have $ X^{\beta_{i}(\theta),\xi^{*}(\theta)}_{t} \leq Z_t,\text{ for any }t\in [0,\tau\wedge \tilde{\eta}),\;\mathbb{P}_{i}\text{-a.s.}$,
which in turn implies
    \begin{equation*}
        \label{prob of hitting times}
         \mathbb{P}_{(\beta_{i}(\theta),i)}\big( \eta(\alpha_{i}(\theta))<\tau \big) \geq \mathbb{P}_{(\beta_{i}(\theta),i)}\big( \tilde \eta <\tau \big).
    \end{equation*} 
Therefore, we find
\begin{equation}\label{eq estimate hotting probability from below}
  \mathbb{P}_{(\beta_{i}(\theta),i)}\big(  \eta (\alpha_i(\theta)) <\tau \big) \geq \mathbb{P}_{(\beta_{i}(\theta),i)}\big( \tilde \eta <\tau \big)  
      = \int_0^\infty \mathbb E [\boldsymbol{1}_{\{ \tilde \eta < t \} }] \,   F_\tau (dt) := \rho >0, 
\end{equation} 
where $F_{\tau}$ denotes the distribution function of $\tau$ and the last inequality follows from the fact that the SDE \eqref{eq SDE Z fixed i} induces a regular diffusion (cf.\ Assumption \ref{eq:Ass 2.1}).

Using \eqref{eq estimate hotting probability from below}, we have
\begin{align}
        \label{inequality stopping}
         \mathbb{E}_{(x,i)}\big[ \eta(\alpha_{i}(\theta)) \big] 
         &=\mathbb{E}_{(x,i)}\big[ \eta(\alpha_{i}(\theta))\boldsymbol{1}_{\{\eta(\alpha_{i}(\theta))<\tau\}}\big]+\mathbb{E}_{(x,i)}\big[ \eta(\alpha_{i}(\theta))\boldsymbol{1}_{\{\eta(\alpha_{i}(\theta))\geq\tau\}}\big] \\ 
         &\leq  \mathbb{E}_{i}\big[ \tau\big]+\mathbb{E}_{(x,i)}\big[\mathbb{E}_{(X^{\xi^{*}(\theta)}_{\tau},Y_{\tau})}\big[\eta(\alpha_{i}(\theta))(1-\boldsymbol{1}_{\{\eta(\alpha_{i}(\theta))<\tau\}})\big]\big] \nonumber \\
         &= \mathbb{E}_{i}\big[\tau\big] +\mathbb{E}_{(x,i)}\big[(1-\boldsymbol{1}_{\{\eta(\alpha_{i}(\theta))<\tau\}})\big]\mathbb{E}_{(x,i)}\big[\mathbb{E}_{(X^{\xi^{*}(\theta)}_{\tau},Y_{\tau})}\big[\eta(\alpha_{i}(\theta))\big]\big] \nonumber \\
         &= \mathbb{E}_{i}\big[ \tau\big]+\big(1-\mathbb{P}_{(x,i)}\big( \eta(\alpha_{i}(\theta))<\tau \big)\big)\mathbb{E}_{(x,i)}\big[\eta(\alpha_{i}(\theta))\big], \nonumber \\
         &\leq \mathbb{E}_{i}\big[ \tau\big]+(1-\rho) \mathbb{E}_{(x,i)}\big[\eta(\alpha_{i}(\theta))\big], \nonumber
    \end{align}
    where in the second equality we use the strong Markov property of $(X^{\xi^{*}(\theta)}_{t},Y_{t})_{t\geq 0}$. 
    Then, since (\ref{inequality stopping}) holds for any $(x,i)\in\mathcal{O}$, we obtain
    $$
        \sup_{(x,i)\in\mathcal{O}}\mathbb{E}_{(x,i)}\big[\eta(\alpha_{i}(\theta))\big]\leq \max_{i\in \mathbb{Y}} \, \mathbb{E}_{i}\big[ \tau\big]+(1-\rho)\sup_{(x,i)\in\mathcal{O}}\mathbb{E}_{(x,i)}\big[\eta(\alpha_{i}(\theta))\big].
    $$
    Equivalently, we write
    \begin{equation}
        \label{finite stopping time}
        \sup_{(x,i)\in\mathcal{O}}\mathbb{E}_{(x,i)}\big[\eta(\alpha_{i}(\theta))\big]\leq \frac{1}{\rho}\max_{i\in \mathbb{Y}}\mathbb{E}_{i}\big[ \tau\big]<\infty,
    \end{equation}
    where, in the last inequality, we have used that the process $(Y_{t})_{t\geq 0}$ is irreducible and positive recurrent.
    
    Finally, plugging \eqref{finite stopping time} into \eqref{eq estimate eta}, we conclude that $\mathbb{E}_{(x,i)}[\eta] < \infty $, which in turn implies that $(X^{\xi^{*}(\theta)}_{t},Y_{t})_{t\geq 0}$ is a regenerative process with  finite mean regenerative epochs. 
    Hence, by Theorem 2.1 in \cite{Sigman}, there exists a unique stationary distribution, concluding the proof. 
\end{proof}

Next, we characterize the stationary distribution and we study its stability with respect to the changes in the parameter $\theta$.

\begin{theorem}
    \label{Thm 4.1}
    Let $\nu^{\theta}\in\mathcal{P}(\mathcal{O})$ be the stationary distribution of $(X_{t}^{\xi^{*}(\theta)},Y_{t})_{t\geq 0}$. The following hold:
    \begin{enumerate}
        \item \label{(1) in Thm 4.1} %\jo{The measure $\nu^{\theta}$ is absolutely continuous with respect to Lebesgue measure $m$,} 
        For any $i\in\mathbb{Y}$, the cumulative distribution function $\mu^{\theta}(x,i):=\nu^{\theta}((\underline{x},x],i),$ $x \in \mathcal{I}$ is $ C^{1}([\alpha_{i}(\theta),\beta_{i}(\theta)])\cap C^{2}((\alpha_{i}(\theta),\beta_{i}(\theta))\setminus \bigcup_{j\in\mathbb{Y}}\{\alpha_{j}(\theta),\beta_{j}(\theta)\})$ and it is the unique nondecreasing solution of the equation
    \begin{equation}
    \label{eq:(4.1)}
        \frac{1}{2}\sigma^{2}(x,i)\,\mu^{\theta}_{xx}(x,i)-(b(x,i)-\sigma\sigma_{x}(x,i))\mu^{\theta}_{x}(x,i)+\sum_{j\in\mathbb{Y}}q_{ji}\mu^{\theta}(x,j)=0,
    \end{equation}
    for $x\in (\alpha_{i}(\theta),\beta_{i}(\theta))$,
    satisfying the boundary conditions $\mu^{\theta}(x,i)=0,\,x\leq\alpha_{i}(\theta)$ and $\mu^{\theta}(x,i)=p(i),\,x\geq \beta_{i}(\theta)$.
    \item \label{Continuity of a->mu} The map $\nu:\mathbb{R}_{+}\to \mathcal{P}(\mathcal{O}),\; \theta\mapsto\nu^{\theta}$ is continuous.
    \end{enumerate}
\end{theorem}

\begin{proof}
For the first claim see Appendix \ref{Appendix C}. 
We only provide here a proof of Claim \ref{Continuity of a->mu}.

For any $(\alpha, \beta) = (\alpha_{i}, \beta_{i})_{i \in \mathbb Y}$, we denote by $\mu(\cdot,\cdot;\alpha,\beta)$ the solution of (\ref{eq:(4.1)}) with $(\alpha(\theta),\beta(\theta))=(\alpha,\beta)$. 
Then for $\Theta\subset \mathbb{R}_{+}$ compact such that $\theta\in\Theta$ and for $\{\theta^{n}\}_{n\in\mathbb{N}}\subset \Theta$ such that $\theta^{n}\to \theta$, as $n\uparrow\infty$, we know by Proposition \ref{Cont of free-boundaries} that $\alpha_{i}(\theta^{n})\to\alpha_{i}(\theta)=\alpha_{i}$ and $\beta_{i}(\theta^{n})\to\beta_{i}(\theta)=\beta_{i}$ for any $i\in\mathbb{Y}$ as $n\uparrow\infty$. For simplicity we set $(\alpha_{i}^{n},\beta_{i}^{n})_{i\in\mathbb{Y}}:=(\alpha_{i}(\theta^{n}),\beta_{i}(\theta^{n}))_{i\in\mathbb{Y}}$ for any $n\in\mathbb{N}$.
%Unless repeating the subsequent argument with minor modifications, we assume that there exists $n_{0}\in\mathbb{N}$ such that, for any $n\geq n_{0}$ and for any $i\in\mathbb{Y}$,  we have $(\alpha_{i}^{n},\beta_{i}^{n})\subseteq (\alpha_{i},\beta_{i})$. 
For any $n\in\mathbb{N}$, let $\mu^n(x,i):=\nu^{n}((\underline{x},x],i),\;(x,i)\in\mathcal{O}$ be the cumulative distribution function of the stationary distribution of the solution to $\mathbf{SP}(\alpha^n,\beta^n,x,i)$ (cf.\ Proposition \ref{Prop 4.1}).
According to Claim \ref{(1) in Thm 4.1}, $\mu^n$ satisfies (\ref{eq:(4.1)}) with boundary condition on $(\alpha^{n},\beta^{n})$. 
Next, for any $n\in\mathbb{N}$ and for any $(x,i)\in\bigcup_{i=1}^{d}(\alpha_{i}^{n},\beta^{n}_{i})\times \{i\}$, define $v^{n}(x,i):=\mu(x,i)-\mu^{n}(x,i)$. 

Since $\mu$ and $\mu^{n}$ are solutions to the   system (\ref{eq:(4.1)}), for each $n\in\mathbb{N}$, $v^{n}$ is a solution to the elliptic system
     \begin{equation*}
     \label{n-BVP}
     \begin{cases}
        \mathcal L^{*} _{(\widehat X,i)}v^n (x,i) +\sum_{j\in\mathbb{Y}}q_{ji}v^{n}(x,j)=0,\; x\in(\max \{ \alpha_i , \alpha_{i}^{n} \}, \min \{ \beta_i, \beta^{n}_{i} \}), \ i \in \mathbb Y, \\
         v^{n}( \max \{ \alpha_i , \alpha_{i}^{n} \} ,i)
            = \mu(\alpha^{n}_{i},i) \boldsymbol{1}_{ \alpha_i < \alpha_i^n }
              - \mu^n(\alpha_{i},i) \boldsymbol{1}_{ \alpha_i \geq \alpha_i^n },
         \\ 
         v^{n}( \min \{ \beta_i, \beta^{n}_{i} \},i)
            = (\mu(\beta^{n}_{i},i)-p(i)) \boldsymbol{1}_{ \beta_i \geq \beta^{n}_{i} }
              + (p(i) - \mu^{n} (\beta_{i},i)) \boldsymbol{1}_{ \beta_i < \beta^{n}_{i} },
     \end{cases}
     \end{equation*}
     where $\mathcal L^{*} _{(\widehat X,i)}v^n (x,i):= \frac{1}{2}\sigma^{2}(x,i)\,v^{n}_{xx}(x,i)-(b(x,i)-\sigma\sigma_{x}(x,i))v^{n}_{x}(x,i)$.
    Now, thanks to Assumption \ref{eq:Ass 2.1} and Assumption \ref{Ass 4.1}-(\ref{Ass 4.1-3}), we can apply Theorem 1 in \cite{maxprinc} to the previous system, which gives us
     \begin{equation*}
         \sup_{x\in \mathcal I _n}|v^{n}(x,i)|
         \leq C 
         \max\big\{ |\mu(\alpha^{n}_{i},i)|, |\mu^n(\alpha_{i},i)|,|\mu(\beta^{n}_{i},i)-p(i)|, | p(i) - \mu^{n} (\beta_{i},i)) | \big\}, 
     \end{equation*}
     where $ \mathcal I _n := (\max \{ \alpha_i , \alpha_{i}^{n} \}, \min \{ \beta_i, \beta^{n}_{i} \})$.
     Further, since the functions $\mu(\cdot,i), \mu^n (\cdot,i)$ are uniformly bounded solutions to elliptic equations, the functions $\mu_x^n (\cdot,i) $ are uniformly bounded, so that
     $$
     \begin{aligned}
      & \mu(\alpha^{n}_{i},i) = \mu(\alpha_{i},i) + \mu(\alpha^{n}_{i},i) - \mu(\alpha_{i},i) = \int_{\alpha_i}^{\alpha_i^n} \mu_x (x,i) dx \leq C |{\alpha_i}-{\alpha_i^n}|, \\
      & p(i) - \mu^{n} (\beta_{i},i)) = \int_{\beta_i}^{\beta_i^n} \mu_x^n(x,i) dx \leq C | {\beta_i^n}-{\beta_i}|.     \end{aligned}
     $$
     This allows us to improve the previous estimate for $\sup_{x\in \mathcal I _n}|v^{n}(x,i)|$, and to use the continuity of $\mu (\cdot,i)$ to find  
     \begin{equation*}
         \sup_{x\in \mathcal I _n}|v^{n}(x,i)|
         \leq C 
         \max\big\{ |\mu(\alpha^{n}_{i},i)|, | {\alpha_i}-{\alpha_i^n}|,|\mu(\beta^{n}_{i},i)-p(i)|, | {\beta_i^n}-{\beta_i}| \big\} \to 0,
     \end{equation*}
     as  $n \to \infty$,  which in turn implies that $\mu^n (x,i) \to \mu(x,i)$ as $n \to \infty$, for any $x \in \mathcal I$. 
     \\ \indent
     To conclude, we recall that $\mu^n ( \cdot,i)$ was defined as the cumulative distribution function associated to the measure $\nu^{\theta^n} ( dx , i)$. 
     Having shown that $\mu^n ( \cdot,i) \to \mu ( \cdot,i)$, with $\mu ( \cdot,i)$ being the cumulative distribution function associated to the measure $\nu^{\theta} ( dx , i)$, by an application of Theorem 5.25 in \cite{kallenberg2002foundations} we obtain that $\nu^{n}\rightharpoonup \nu^{\theta}$ as $n\to\infty$, thus completing the proof of Claim \ref{Continuity of a->mu}.
\end{proof}

\subsection{Existence and uniqueness of the MFG equilibrium}
Next, we turn our attention to the main result of this section. 
To this end, we introduce the operator $\mathcal{T}:\mathbb{R}_{+}\to \mathbb{R}_{+}$, as
\begin{equation}
\label{eq:(4.2)}
    \mathcal{T}\theta:=F\bigg(\sum_{i=1}^{d}\int_{\mathcal{I}}f(x)\nu^{\theta}(dx,i)\bigg)=:F\big( \langle f,\nu^{\theta} \rangle \big),\quad \theta\in\mathbb{R}_{+}
\end{equation}
where $\langle f,\nu^{\theta}\rangle:=\sum_{i\in\mathbb{Y}}\int_{\mathcal{I}}f(x)\nu^{\theta} (dx,i)$. Thanks to the previous results, we can now prove the existence and uniqueness of a stationary mean-field equilibrium.

\begin{theorem}
    \label{Thm: existence MFGE}
    There exists a unique MFG equilibrium $\theta^{*}$; i.e., a unique $\theta^{*}\in\mathbb{R}_{+}$ such that $\mathcal{T}\theta^{*}=\theta^{*}$.
\end{theorem}
\begin{proof}
    \textbf{Step 1: Set of relevant $\theta$.} We introduce the following auxiliary Dynkin game
    \begin{align}
        \label{aux Dynkin game}
        \underline{v}(x,i)  :=\inf_{\tau\in\mathcal{T}}\sup_{\sigma\in\mathcal{T}}\widehat{\mathbb{E}}_{(x,i)}\bigg[ & \int_{0}^{\tau\wedge \sigma}e^{\int_{0}^{t}b_{x}(\widehat{X}_{s},Y_{s})ds}\kappa(\widehat{X}_{s})ds  \\
        & +k_{1}e^{\int_{0}^{\tau}b_{x}(\widehat{X}_{s},Y_{s})ds}\boldsymbol{1}_{\{\tau<\sigma\}}+k_{2}e^{\int_{0}^{\sigma}b_{x}(\widehat{X}_{s},Y_{s})ds}\boldsymbol{1}_{\{\sigma<\tau\}} \bigg]. \nonumber
    \end{align}
    Arguing as in Section \ref{Section 3.1}, the Dynkin game (\ref{aux Dynkin game}) admits a value and a saddle point. 
    In particular, we can define $\underline{\alpha}_{i}:=\sup\{x\in\mathcal{I}:\underline{v}(x,i)\geq k_{1}\}$ and $\underline{\beta}_{i}:=\inf\{x\in\mathcal{I}:\underline{v}(x,i)\leq k_{2}\}$, as well as the stopping regions $\underline{\mathcal{S}}_{\inf}:=\{(x,i)\in\mathcal{O}:\underline{v}(x,i)\geq k_{1}\}$ and $\underline{\mathcal{S}}_{\sup}:=\{(x,i)\in\mathcal{O}:\underline{v}(x,i)\leq k_{2}\}$. 
    Thus, we have that the stopping times $\underline{\tau}^{*}(x,i):=\inf\{t\geq 0:(\widehat{X}_{t},Y_{t})\in \underline{\mathcal{S}}_{\inf}\}$ and $\underline{\sigma}^{*}(x,i):=\inf\{t\geq 0:(\widehat{X}_{t},Y_{t})\in \underline{\mathcal{S}}_{\sup}\}$ realize a saddle-point for (\ref{aux Dynkin game}). 
    Also by Condition (\ref{eq:2.2-3}) in Assumption \ref{eq:Ass 2.2} and Condition (\ref{Ass 4.1-1}) in Assumption \ref{Ass 4.1} we conclude that $\pi_{x}(x,\theta)\geq \kappa(x)$, for any $\theta\in\mathbb{R}_{+}$, so that (cf. (\ref{eq:(3.4)}) and (\ref{eq:(3.6)})) 
    \begin{equation}
        \underline{v}(x,i)\leq v(x,i;\theta),\;\text{ for any } (x,i;\theta)\in\mathcal{O}\times \mathbb{R}_{+}.
    \end{equation}
    Then,  for any $(i, \theta)\in \mathbb{Y}\times \mathbb{R}_{+}$, we have
    \begin{equation}
        \alpha_{i}(\theta)=\sup\{x\in\mathcal{I}: v(x,i;\theta)\geq k_{1}\} \geq \sup\{x\in \mathcal{I}: \underline{v}(x,i)\geq k_{1}\}=\underline{\alpha}_{i}.
    \end{equation}
     Analogously, we find
    \begin{equation}
        \label{eq:(4.5)}
        \beta_{i}(\theta)\geq\underline{\beta}_{i},\;\text{ for any } (i,\theta)\in \mathbb{Y}\times \mathbb{R}_{+}.
    \end{equation}
    We then define
    \begin{equation}
        \underline{\theta}:=F\bigg(\sum_{i=1}^{d}\int_{\mathcal{I}}f(x)\underline{\nu}(dx,i)\bigg)
    \end{equation}
    where $\underline{\nu}$ is the unique stationary distribution of the process $(\underline{X}_{t},Y_{t})_{t\geq 0}$ and  $\underline{X}$ is the solution to the $\mathbf{SP}((\underline{\alpha}_{i},\underline{\beta}_{i})_{i\in\mathbb{Y}}; x, i)$. Then, by comparison principle (see Proposition B.2), since $f$ is increasing, we obtain
    \begin{equation}
        \mathbb{E}_{(x,i)}\big[ f(\underline{X}_{t})\big]\leq\mathbb{E}_{(x,i)}\big[ f(X^{\xi^{*}(\theta)}_{t})\big],\;\text{ for any }t\geq 0.
    \end{equation}
    Hence, by Proposition \ref{Prop 4.1}, we can employ the ergodic theorem (see Theorem 4.4 in \cite{yin2009hybrid}) in order to deduce that
    $$
    \langle f, \underline{\nu} \rangle = 
   \lim_{t \to \infty} \frac1t \int_0^t \mathbb{E}\big[ f(\underline{X}_{t})\big] dt \leq  \lim_{t \to \infty} \frac1t \int_0^t \mathbb{E}\big[ f(X^{\xi^{*}(\theta)}_{t})\big] dt = 
   \langle f, {\nu^\theta} \rangle, 
    $$
    which in turn, by monotonicity of $F$, leads to
    \begin{equation}
        \label{eq:(4.8)}
        \underline{\theta}\leq\mathcal{T}\theta.
    \end{equation}
    To find an upper bound for $\mathcal{T}\theta$, we proceed as follows. 
    Since, $X^{\xi^{*}(\theta)}_{t}\in [\alpha_{Y_{t}}(\theta),\beta_{Y_{t}}(\theta)]$, for any $t\geq 0,\;\mathbb{P}$-a.s.\ for any $\theta\in\mathbb{R}_{+}$, thanks to the monotonicity of $F$ and $f$ (see Condition (\ref{eq:Ass 2.3-1}) Assumption \ref{eq:Ass 2.3}), we have
    \begin{equation*}
    \mathcal{T}\theta   =F\bigg(\sum_{i=1}^{d}\int_{\mathcal{I}}f(x)\nu^{\theta}(dx,i)\bigg)\leq F\bigg(\sum_{i=1}^{d}p(i)f(\beta_{i}(\theta))\bigg) \leq F\bigg(\sum_{i=1}^{d}p(i)f(\beta_{i}(\underline{\theta}))\bigg), 
    \end{equation*}
    for any $\theta\geq \underline{\theta}$. 
    In the last inequality above we have used the monotonicity of $\theta \mapsto (\beta_{i}(\theta))_{i\in\mathbb{Y}}$ (cf. Lemma \ref{Lemma 4.1}).
    Hence, for $\overline{\theta}:=F\big(\sum_{i=1}^{d}p(i)f(\beta_{i}(\underline{\theta}))\big)$   we find that
    \begin{equation}
        \label{eq:(4.10)}
        \mathcal{T}\theta\leq \overline{\theta}, \quad \text{for any $\underline{\theta}\leq \theta$.}
    \end{equation}
	Thus, combining (\ref{eq:(4.8)}) and (\ref{eq:(4.10)}), we conclude that any potential fixed point of $\mathcal{T}$ must lie in the convex, compact set 
 \begin{equation}\label{eq compact set}
 K:=[\underline{\theta},\overline{\theta}]\subset \mathbb{R}_{+}.
 \end{equation}
    
    \textbf{Step 2: Continuity of $\mathcal{T}$.} 
    We begin by observing  that the barriers related to $\theta \in K$, belong to a compact $\widehat{K}$, defined in terms of the compact set $K$ in \eqref{eq compact set}. 
    Indeed, by Step 1 we know that $\underline{\alpha}_{i}\leq \alpha_{i}(\theta)$ for any $(i,\theta)\in\mathbb{Y}\times K$, and, by monotonicity of the map $\theta\mapsto \beta_{i}(\theta)$ for $i\in\mathbb{Y}$ we have $\beta_{i}(\theta)\leq \beta_{i}(\underline{\theta})$ for $i \in\mathbb{Y}$. 
    Thus, for each $\theta\in K$ we have,
    \begin{equation}
        \label{remark compact set} 
        \alpha_{i}(\theta),\beta_{i}(\theta) \in \Big[ \min_{j\in\mathbb{Y}}\underline{\alpha}_{j},\max_{j\in\mathbb{Y}}\beta_{j}(\underline{\theta}) \Big]
        =:\widehat{K}, \; \text{for any $i\in\mathbb{Y}$.}
    \end{equation}
    
    Define now the map $\mathcal{T}_{1}:K\to \mathcal{P}(\widehat K \times \mathbb Y)$ by
        $\mathcal{T}_{1}(\theta):=\nu^\theta,$ $ \theta\in K$.
    By \eqref{remark compact set}, such a map is well defined.
   Moreover, thanks to Claim \ref{Continuity of a->mu} in Theorem \ref{Thm 4.1}, the map $\mathcal T _1$ is continuous. 
   Next, we denote by $\mathcal{T}_{2}:\mathcal{P}(\widehat K \times \mathbb Y) \to K$ the map 
       $ \mathcal{T}_{2}(\nu):=F\big( \sum_{i=1}^{d}\int_{\mathcal{I}}f(x)\nu(dx,i)\big)$.
   Since the functions $f$ and $F$ are continuous and the probability measures have compact support, the map $\mathcal T _2$ is clearly continuous.
Concluding, the map $\mathcal{T}:=\mathcal{T}_{2}\circ\mathcal{T}_{1}:K \to K$ is continuous in the convex compact set $K$ and, by Schauder-Tychonof fixed point theorem ( Corollary 17.56 in \cite{aliprantis1999infinite}), there exist $\theta^{*}\in K$, such that $\mathcal{T}\theta^{*}=\theta^{*}$.
    
        \textbf{Step 3: Uniqueness.} Let $\theta^{*}\in K$ be the fixed-point of $\mathcal{T}$ and let $\theta'\in K$ another fixed-point of $\mathcal{T}$ such that $\theta^{*}\neq \theta'$. Without loss of generality we can assume that $\theta^{*}<\theta'$ (the opposite inequality can be treated similarly). Following the same arguments as in Step 1 we conclude that
    	$\alpha_{i}(\theta^{*})\geq \alpha_{i}(\theta'), $ $ \beta_{i}(\theta^{*})\geq \beta_{i}(\theta'),\;\forall i\in\mathbb{Y}$,
    so that $\theta'=\mathcal{T}\theta'\leq \mathcal{T}\theta^{*}=\theta^{*}$, which leads to a contradiction.
\end{proof}

\section{\emph{N}-Player game and approximate Nash equilibria}
\label{Approximation}

In this section we establish the classical connection between MFG and $N$-player game, by constructing approximate Nash equilibria starting from the MFG equilibrium.

\subsection{\emph{N}-player game} 
Let $W$ and $Y$ be as in Section \ref{Section 2} and
assume the filtered probability space $(\Omega,\mathcal{F},\mathbb{F}=\{\mathcal{F}_{t}\}_{t\geq 0},\mathbb{P})$ to be large enough to accommodate a sequence of independent and identically distributed $\mathbb{F}$-adapted processes $\{(W^{n},Y^{n})\}_{n\in\mathbb{N}}$
as well as independent and identically distributed $\mathcal I \times \mathbb Y$-random variables $\{ (x^n_0, i^n_0) \}_{n \in \mathbb N}$, $(x_0,i_0)$.
Each $(W^n,Y^n)$ has the same distribution as $(W,Y)$ and the random variables $(W^n,Y^n)$, $(W,Y)$, $\{ (x^n_0, i^n) \}_{n \in \mathbb N}$ and $(x_0,i_0)$ are assumed to be independent.

When player $n$ does not intervene, its state process $X^n$ evolves accordingly to the SDE
 $$
 dX^{n}_{t}=b(X^{n}_{t},Y^{n}_{t})dt+\sigma(X^{n}_{t},Y^{n}_{t})dW^{n}_{t},\quad (X^{n}_{0},Y^{n}_{0})=(x^n_0,i^n_0).
 $$
Given a boundary vector $(\alpha^{n},\beta^{n})=(\alpha^{n}_j,\beta^{n}_j)_{j \in\mathbb Y} \in \mathbb R  ^{2d}$, the $(\alpha^{n},\beta^{n})$-barrier-type strategy for player $n$ is the reflection process for the state of player $n$ in the regime-switching domain $(\alpha^{n}_{Y_{t}},\beta^{n}_{Y_{t}})_{t\geq 0}$; that is, the $\xi^n$-component of the solution $(X^n,\xi^n)$ to the Skorokhod problem $\mathbf{SP}(\alpha^n,\beta^n,x,i)$ for the noise $(W^{n},Y^{n})$, accordingly to Definition \ref{eq:Def 3.1}.  
Without carrying the dependence on the index of the player, we denote by $\mathcal A _b$ the set of barrier-type strategies.
The inclusion $\mathcal {A}_b \subset \mathcal{A}_e$ is given  by the fact that any $X$ controlled through a process belonging to $\mathcal {A}_b$ is bounded (see \eqref{admissible controls}).
Thus, when player $n$ choose a strategy $\xi^{n}\in\mathcal{A}_{b}$, its state process $X^{n,\xi^{n}}$ evolves following the regime-switching (reflected) SDE
\begin{equation*}
    dX^{n,\xi^{n}}_{t}=b(X^{n,\xi^{n}}_{t},Y^{n}_{t})dt+\sigma(X^{n,\xi^n}_{t},Y^{n}_{t})dW^{n}_{t}+d\xi^{n,+}_{t}-d\xi^{n,-}_{t},\quad (X^{n,\xi^{n}}_{0-},Y^{n}_{0})=(x^n_0,i^n_0).
\end{equation*}
The (generic) vector $\boldsymbol{\xi}:=(\xi^{1},...,\xi^{N})\in\mathcal{A}_{b}^{N}$ represents a tuple of control policies of the  $N$ players.
We denote by $\boldsymbol{\xi}^{-n} =( \xi^\ell) _{\ell \ne n}$ a vector of strategies of the opponents of player $n$, and we use the notation  $\boldsymbol{\xi}=(\xi^{n},\boldsymbol{\xi}^{-n})$. 
The profit functional of player $n$ is defined as
\begin{equation}
\label{eq:functJn}
    J^n(\xi^{n},\boldsymbol{\xi}^{-n}):=\limsup_{T\uparrow \infty}\frac{1}{T}\mathbb{E}\bigg[ \int_{0}^{T}\pi(X^{n,\xi^{n}}_{t},\theta^{N}_{\boldsymbol{\xi}^{-n}}(t))dt-k_{1}\xi^{n,+}_{T}+k_{2}\xi^{n,-}_{T} \bigg],
\end{equation} 
where $\theta^{N}_{\boldsymbol{\xi}^{-n}}$ denotes the mean-field interaction term between the players and has the form
\begin{equation}
\label{eq:thetaN}
    \theta^{N}_{\boldsymbol{\xi}^{-n}}(t):=
        F\bigg(\frac{1}{N-1}\sum_{\ell\neq n}f(X^{\ell,\xi^{\ell}}_{t}) \bigg), \quad t\geq0.
\end{equation}

For $\widehat K$ being the compact set as in \eqref{remark compact set}, define the set of restricted barrier strategies
\begin{equation}
    \label{eq restricted strategies}
    \widehat {\mathcal A}_b := \{ \text{$(\alpha^{n},\beta^{n})$-barrier-type strategy with $(\alpha^n_j,\beta^n_j)_{j\in\mathbb{Y}} \in \widehat{K}^{2d}$}\} \,\subset \mathcal A_b \subset \mathcal A _e,
\end{equation}
 we give the following definition of $\epsilon$-Nash equilibrium (see also Remark \ref{remark extension admissible deviations} below). 
\begin{definition}[$\epsilon$-Nash equilibrium]
\label{Nash eq}
   For $\epsilon>0$, $\bar{\boldsymbol{\xi}}=(\bar{\xi}^{1},...,\bar{\xi}^{N})\in\widehat{\mathcal{A}}_{b}^{N}$ is called $\epsilon$-Nash equilibrium for the $N$-player game if, for any $n=1,...,N$, one has  
   \begin{equation}
       J^{n}(\bar{\xi}^{n},\bar{\boldsymbol{\xi}}^{-n})\geq J^{n}(\xi^{n},\bar{\boldsymbol{\xi}}^{-n})-\epsilon,\quad \forall \xi^{n}\in\widehat{\mathcal{A}}_{b}.
   \end{equation}
\end{definition}

%\textcolor{red}{\begin{remark}[On the set $\widehat{\mathcal{A}}_{b}$]
%    The set $\widehat{\mathcal{A}}_{b}$ consists the best response strategies of player n when the opponents act optimally. Indeed, since $X^{k,\xi^{k}}_{t}\in \widehat{K},\;\mathbb{P}\otimes dt$-a.s. for any $k\neq n$ (cf. (\ref{remark compact set})) we have that $\theta^{N}_{\boldsymbol{\xi}^{-n}}(t)\in K$ for any $t\geq 0$ (cf. (\ref{eq compact set})), hence the best response strategy of player n, denoted by $\xi^{n}$, has to be of barrier-type and belongs to $\mathcal{A}_{e}$ (see Theorem 1 in \cite{DharmaKwon}). Moreover, since players are symmetric it holds that $X^{n,\xi^{n}}_{t}\in \widehat{K},\;\mathbb{P}\otimes dt$-a.s. 
%\end{remark}}

\begin{remark}[On the initial distribution] 
We point out that all the results in the previous sections hold true also if the deterministic initial condition $(X_{0-}, Y_0) = (x,i) \in \mathcal I \times \mathbb Y$ is replaced by the random initial condition $(X_{0-},Y_0)=(x_0,i_0)$. 
Indeed, by considering (with slight abuse of notation) the payoff
$$
J(x_0,i_0;\xi,\theta):=  \limsup_{T\uparrow \infty} \frac{1}{T} \mathbb E \bigg[ \mathbb{E}_{(x_0,i_0)}\bigg[ \int_{0}^{T}\pi(X^{\xi}_{t},\theta)dt-k_{1}\xi^{+}_{T}+k_{2}\xi^{-}_{T} \bigg]\bigg], 
$$
by the Markov property of the solution to the reflected Skorokhod problem, 
the MFG equilibrium $(\xi(\theta^{*}),\theta^*)$ is still given by Theorem \ref{Thm: existence MFGE}, with $\xi(\theta^{*})$ characterized as in in Proposition  \ref{Prop  3.5}. 
\end{remark}

We state the following auxiliary result, which will be useful for our subsequent analysis regarding the approximation.
\begin{proposition}
    \label{Prop 4.2}
    For any $N\geq 1$ and $\{\xi^{n}\}_{n\leq N}\in\widehat{\mathcal{A}}_{b}^{N}$, the ergodic limit\begin{equation*}
                    \lim_{T\uparrow\infty}\frac{1}{T}\int_{0}^{T}G(X^{1,\xi^{1}}_{t},Y^{1}_{t},...,X^{N,\xi^{N}}_{t},Y^{N}_{t})dt=\int_{\mathcal{O}^{N}}G(x_{1},i_{1},...,x_{N},i_{N})\otimes_{\ell=1}^{N}
                    \nu^{\ell}(dx_{\ell},i_{\ell}),
        \end{equation*}
     holds $\mathbb{P}$-a.s.\ for any bounded function $G:\mathcal{O}^{N}\to\mathbb{R}$.        \end{proposition}

\begin{proof}

Since $\{(W^{n}_{t},Y^{n}_{t})_{t\geq0}\}_{n\in\mathbb{N}}$ is i.i.d., the family  $\{(X^{n,\xi^n}_{t},Y^{n}_{t})_{t\geq0}\}_{n\leq N}$ is i.i.d. For any fixed $1\leq n\leq N$, by the proof of Proposition \ref{Prop 4.1} we know that the process $(X^{n,\xi^{n}}_{t},Y^{n}_{t})_{t\geq 0}$ has a unique stationary distribution $\nu^n$. 
Thus, by independence of the processes $(X^{n,\xi^{n}}_{t},Y^{n}_{t})_{t\geq 0}$, 
the stationary distribution for $(X^{1,\xi^{1}},Y^{1},...,X^{N,\xi^{N}},Y^{N})$, denoted by $\bar{\nu}^{N}\in\mathcal{P}(\mathcal{O}^{N})$, exists and it is given by
    \begin{equation*}
        \bar{\nu}^{N}(dx_{1},i_{1},...,dx_{N},i_{N})=\otimes_{\ell=1}^{N}\nu^{\ell}(dx_{\ell},i_{\ell}).
    \end{equation*}
Clearly, this distribution is unique.\ Thus, by Theorems 3.2.6 and 3.3.1 in \cite{daprato_zabczyk_1996}, the ergodic limit holds.
\end{proof}

\subsection{Approximation result}
Our aim is to use the solution of the ergodic MFG as an approximating solution of the $N$-player game. 
To this end, we enforce the following condition.
\begin{assumption}
	\label{Ass 4.2}
    There exists $C>0$ such that,
        $
            |\pi(x,\theta_{1})-\pi(x,\theta_{2})|\leq C(1+|x|^{\beta})|\theta_{1}-\theta_{2}|,
        $
        for any $\theta_{1},\theta_{2}\in \mathbb{R}_{+}$ and $x\in\mathcal{I}$.
\end{assumption}

Let $\theta^{*}$ be the unique MFG equilibrium as in Theorem \ref{Thm: existence MFGE} and  $(\alpha(\theta^{*}),\beta(\theta^{*}))=(\alpha_{i}(\theta^{*}),\beta_{i}(\theta^{*}))_{i\in\mathbb{Y}}$ be the related boundary (see (\ref{eq:(3.7)})). 
We now use the MFG equilibrium $\theta^*$ in order to construct profile strategies for the $N$-player games.
For any $N\geq 1$ and $n=1,...,N$, define the strategy $\bar{\xi}^{n}$ as the $(\alpha(\theta^{*}),\beta(\theta^{*}))$-barrier strategy for the noise $(W^{n},Y^{n})$. 
In light of the definition of $\widehat{K}$ in \eqref{remark compact set}, we have $\bar \xi ^n \in \widehat{\mathcal A}_b$.
Let $\bar{\boldsymbol{\xi}}$ be the related profile strategy; i.e., set $\bar{\boldsymbol{\xi}} : = (\bar \xi ^1 ,..., \bar \xi ^N)$.

In the spirit of \cite{CaoDianettiFerrari} and \cite{CAO2022995}, we can now state the final and main result of this section.

\begin{theorem}
    \label{idiosyn N player}
    The profile strategy $\bar{\boldsymbol{\xi}}:=(\overline{\xi}^{1},...,\overline{\xi}^{N})\in \widehat{\mathcal{A}}_{b}^{N}$ is an $\epsilon_{N}$-Nash equilibrium for the ergodic $N$-player game, with $\epsilon_{N}\to 0$ as $N\to \infty$. 
\end{theorem}

\begin{proof}
We first introduce some notation.
Take barriers $(\alpha,\beta)=(\alpha_{i},
\beta_{i})_{i\in\mathbb{Y}}\in \widehat{K}^{2d}$, with $\widehat{K}$ as in (\ref{remark compact set}).
 Denote by $(X^{n,\xi^{n}},\xi^{n})$ the solution to $\mathbf{SP}(\alpha,\beta,x,i)$ for the noise $(W^{n},Y^{n})$ of player $n$, and denote by  $(\tilde{X},\tilde{\xi})$ the solution to  $\mathbf{SP}(\alpha,\beta,x,i)$ for the noise $(W,Y)$ of the MFG.
 Notice that the solution to $\mathbf{SP}(\alpha,\beta,x,i)$ can be constructed by a Picard iteration (see Proposition \ref{eq:Prop B.1}) and thus it is progressively measurable with respect to the noises.
 This implies that such a solution is unique in distribution, 
 so that the processes $(X^{n,\xi^{n}},\xi^n)$ and $(\tilde X ^{n,\xi^{n}}, \tilde \xi ^n)$ have the same law, and
 \begin{align}
    \label{eq identity J uniqueness in law}
 J^n(\xi^n,\theta^{*}) &:= \limsup_{T\uparrow\infty}  \frac{1}{T} \mathbb{E}\bigg[ \int_{0}^{T}  \pi(X^{n,\xi^{n}}_{t},\theta^{*})dt-k_{1}\xi^{n,+}_{T}+k_{2}\xi^{n,-}_{T}\bigg]  \\ \notag
& =\limsup_{T\uparrow\infty}  \frac{1}{T} \mathbb{E}\bigg[ \int_{0}^{T}  \pi(\tilde X ^{n,\xi^{n}}_{t},\theta^{*})dt-k_{1}\tilde{\xi}^{n,+}_{T}+k_{2} \tilde{\xi}^{n,-}_{T}\bigg] =: J(\tilde \xi, \theta^* ). \notag
 \end{align} 

Next, introduce the error estimate between the functionals of the $N$-player game and the MFG. 
 For  $N\in\mathbb{N}$ and $1\leq n\leq N$, set
$$
R_{N}(\xi^n):=J^n(\xi^n,\bar{\xi}^{-n})-J^n(\xi^n,\theta^{*}),\quad \xi^n \in \widehat{\mathcal{A}} _b.
$$
Since the control $\xi(\theta^{*})$ is optimal for $\theta^{*}$ (cf. Proposition \ref{Prop 3.5}),   using \eqref{eq identity J uniqueness in law}, we have that
     \begin{align*}
        J^{n}(\bar{\xi}^{n},\bar{\xi}^{-n}) -J^{n}(\xi^n,\bar{\xi}^{-n})
        & =R_{N}(\bar{\xi}^{n})-R_{N}(\xi^n)+J^n(\bar{\xi}^{n},\theta^{*})- J^n(\xi^n,\theta^{*}) \\
        & =R_{N}(\bar{\xi}^{n})-R_{N}(\xi^n)+J(\xi(\theta^{*}),\theta^{*})- J(\tilde \xi ^n,\theta^{*}) \\
        & \geq R_{N}(\bar{\xi}^{n})-R_{N}(\xi^n), \; \text{ for any $\xi ^n \in \widehat{\mathcal{A}}_{b}$.}
     \end{align*}
Therefore, in order to complete the proof it is sufficient to show that
\begin{equation}\label{eq limit of R N}
    \lim_{N\uparrow\infty} \sup_{\xi ^n \in\widehat{\mathcal{A}}_{b}} R_{N}(\xi^n )=0.
\end{equation}
    
By using that $\overline{\lim}_{T\uparrow\infty}\big(\alpha_{T}+\beta_{T}\big)\geq \overline{\lim}_{T\uparrow \infty}(\alpha_{T})+\underline{\lim}_{T\uparrow \infty}(\beta_{T}) $, write
\begin{equation}\label{resid}
\begin{aligned}
         R_{N}(\xi^n) 
    %& =\limsup_{T\uparrow\infty}  \frac{1}{T}\mathbb{E}\bigg[ \int_{0}^{T}  \pi(X^{n,\xi^{n}}_{t},\theta^{N}_{\boldsymbol{\bar{\xi}^{-n}}} (t))dt-k_{1}\xi^{n,+}_{T}+k_{2}\xi^{n,-}_{T}\bigg] \\  
    %& \quad \quad \quad \quad -\limsup_{T\uparrow\infty}\frac{1}{T}\mathbb{E}\bigg[ \int_{0}^{T}  \pi(X^{n,\xi^{n}}_{t},\theta^{*})dt-k_{1}\xi^{n,+}_{T}+k_{2}\xi^{n,-}_{T}\bigg]  \\  
    & \leq \limsup_{T\uparrow\infty}  \frac{1}{T}\mathbb{E}\bigg[ \int_{0}^{T}  \Big( \pi(X^{n,\xi^{n}}_{t},\theta^{N}_{\boldsymbol{\bar{\xi}^{-n}}} (t)) - \pi(X^{n,\xi^{n}}_{t},\theta^{*}) \Big) dt \bigg].  
\end{aligned}
\end{equation}
Next, by Proposition \ref{Prop 4.1}, there exists a stationary distribution $\nu^{\xi^{n}}$ for the process $(X^{n,\xi^{n}}_{t},Y_{t})_{t\geq 0}$.
Moreover, Proposition \ref{Prop 4.2} ensures that the process $\big( (X^{n,\xi^{n}}, Y^n),  \{(X^{\ell,\bar{\xi}^{\ell}},Y^{\ell})_{t\geq 0}\}_{\ell\neq n} \big)$ admits an ergodic distribution  $\otimes_{\ell\neq n}\nu^{\ell} \otimes \nu^{\xi^{n}}$ with support in the compact set $(\widehat K \times \mathbb Y )^N$.
Thus, by continuity of the functions $f,F$ and $\pi$, Proposition \ref{Prop 4.2} also allows to rewrite \eqref{resid} as
\begin{equation}
\label{res}
 R_{N}(\xi^n) \nonumber \leq \int_{\mathcal{O}} \int_{\mathcal{O}^{N-1}}\Big( \pi \Big( x_{n}, F \big( \sum_{\ell \ne n} \frac{f(x_\ell)}{N-1} \big) \Big) - \pi( x_n, \theta^{*}) \Big) \otimes_{\ell\neq n}\nu^{\ell}(dx_{\ell},i_{\ell})\otimes \nu^{\xi^{n}}(dx_{n},i_{n}).    
\end{equation}

We proceed by further estimating $\sup_{\xi^n \in \widehat{A} _b} R_{N}(\xi^n)$.
From the latter inequality, since $\otimes_{\ell\neq n}\nu^{\ell} \otimes \nu^{\xi^{n}}$ has compact support, we can use Assumption \ref{Ass 4.2} to obtain 
\begin{equation}\label{eq term in epsilon approximation}
\begin{aligned}
    R_{N}(\xi^n)
    &\leq C \int_{\mathcal{O}^N}  (1+|x_{n}|^{\beta})\bigg|F\bigg( \sum_{\ell \ne n} \frac{f(x_\ell)}{N-1} \bigg)  -F\big( \langle f,\nu^{\theta^{*}}\rangle\big) \bigg|\otimes_{\ell\neq n}\nu^{\ell}(dx_{\ell},i_{\ell})\otimes\nu^{\xi^{n}}(dx_{n},i_{n})  \\
%    &\leq C \bigg(\int_{\mathcal{O}}(1+|x_{n}|^{\beta})\nu^{\xi^{n}}(dx_{n},i_{n})\bigg) \\ 
%    & \qquad \times \bigg( \int_{\mathcal{O}^{N-1}}\bigg|F\bigg( \sum_{\ell \ne n} \frac{f(x_\ell)}{N-1} \bigg)-F\big( \langle f,\nu^{\theta^{*}}\rangle\big) \bigg|\otimes_{\ell\neq n}\nu^{\ell}(dx_{\ell},i_{\ell})\bigg) \\
    &\leq C_{\widehat K} \bigg( \int_{\mathcal{O}^{N-1}}\bigg|F\bigg( \sum_{\ell \ne n} \frac{f(x_\ell)}{N-1} \bigg) -F\big( \langle f,\nu^{\theta^{*}}\rangle\big) \bigg|\otimes_{\ell\neq n}\nu^{\ell}(dx_{\ell},i_{\ell})\bigg),
\end{aligned}
\end{equation}
where the constant $C_{\widehat K}$ depends only on the compact $\widehat K$, since $\nu^{\xi^n}$ ha support in the compact $\widehat K$ ($\xi ^n \in \widehat A _b$ as in \eqref{eq restricted strategies}).
Next, using the local Lipschitz property of $F$ (cf. Assumption (\ref{eq:Ass 2.3-2-b})-(\ref{eq:Ass 2.3})), and again the fact that  $\nu^{\ell}$ are supported in the compact $\widehat K \times \mathbb Y$, we obtain
\begin{align*}
     \sup_{\xi^n \in \widehat{A} _b} R_{N}(\xi) 
    &  \leq C_{\widehat K} \int_{\mathcal{O}^{N-1}}\bigg(1+\sum_{\ell\neq n}\frac{|f(x_\ell)|}{N-1}+\langle f,\nu^{\theta^{*}}\rangle\bigg)^{\frac{1}{\beta}-1} \bigg|\sum_{\ell \ne n} \frac{f(x_\ell)}{N-1} - \langle f,\nu^{\theta^{*}}\rangle\bigg|\underset{\ell\neq n}{\otimes} \nu^{\ell}(dx_{\ell},i_{\ell}) \\
    &  \leq \widehat{C}_{\widehat K} \int_{\mathcal{O}^{N-1}}\bigg|\frac{1}{N-1}\sum_{\ell \ne n} f(x_\ell) - \langle f,\nu^{\theta^{*}}\rangle\bigg|\otimes_{\ell\neq n}\nu^{\ell}(dx_{\ell},i_{\ell}).
\end{align*}
Finally, by using the strong law of large numbers (cf.\ Theorem 4.23 in \cite{kallenberg2002foundations}) and the dominated convergence theorem, a limit as $N \to \infty$ in the latter inequality leads to \eqref{eq limit of R N}. 
This completes the proof of the theorem.
\end{proof}

We conclude this section with a technical remark on the set of deviations (cf.\ \eqref{eq restricted strategies}). 
\begin{remark}
    \label{remark extension admissible deviations}
    Limiting players to use restricted deviations (see \eqref{eq restricted strategies}) is a way to get a uniform bound in the error term $\epsilon_N$; this is used in \eqref{eq term in epsilon approximation} in the proof of Theorem \ref{idiosyn N player}.
    It seems possible to extend Theorem \ref{idiosyn N player} also to the case in which players' admissible deviations are not  assumed to be bounded a priori; that is, to the case case in which $\widehat{\mathcal{A}}_{b}^{N}$ in Definition \ref{Nash eq} is replaced with $\widehat{\mathcal{A}}^{N}$.
    In order to do so, one has to show that, when player $n$'s opponents choose reflection strategies $\bar{\boldsymbol{\xi}}^{-n} \in \widehat{\mathcal{A}}_{b}^{N-1}$ of the MFG equilibrium, player $n$'s optimal strategy over $\widehat{\mathcal{A}}_e$ is an element of $\widehat{\mathcal{A}}_{b}$ (i.e., without restricting to it ex ante). 
    However this intermediate result comes with a certain technical cost, since one has to: 
    (1) Repeat the analysis of Section 3 by replacing the constant parameter $\theta$ by the process $\{ \theta^{N}_{ \bar{\boldsymbol{\xi}}^{-n} }(t)\}_t$ defined in \eqref{eq:thetaN}; 
    (2) Show a monotonicity of optimal reflection boundaries with respect the stochastic parameter $\theta^{N}_{\bar{\boldsymbol{\xi}}^{-n}}$; 
    (3) Obtain bounds for the optimal reflection strategies in the spirit of \eqref{eq:(4.8)} and \eqref{eq:(4.10)}.
\end{remark}

\appendix

\section{Results on the Optimal Stopping Game}
\label{Appendix A}

    \begin{proof}[Proof of Theorem \ref{eq:Thm 3.1}]
    \label{Proof of Theorem 3.1}
    Let $(x,i,\theta)\in\mathcal{O}\times \mathbb{R}_{+}$ and $\tau,\sigma\in\mathcal{T}$ given and fixed. We aim at applying Theorem 2.1 in \cite{doi:10.1137/S0040585X97983821}, and for this we first notice that 
    
    \begin{multline}
        \widehat{J}(x,i;\tau,\sigma,\theta) 
             = \widehat{\mathbb{E}}_{(x,i)}\bigg[ \int_{0}^{\tau\wedge \sigma}e^{\int_{0}^{t}b_{x}(\widehat{X}_{s},Y_{s})ds}\pi_{x}(\widehat{X}_{t},\theta)dt 
              +k_{1}e^{\int_{0}^{\tau}b_{x}(\widehat{X}_{s},Y_{s})ds}\boldsymbol{1}_{\{\tau\leq\sigma\}} \nonumber \\
              +k_{2}e^{\int_{0}^{\sigma}b_{x}(\widehat{X}_{s},Y_{s})ds}\boldsymbol{1}_{\{\sigma<\tau\}} \bigg]  
             =\widehat{\mathbb{E}}_{(x,i)}\bigg[ \int_{0}^{\infty}e^{\int_{0}^{t}b_{x}(\widehat{X}_{s},Y_{s})ds}\pi_{x}(\widehat{X}_{t},\theta)dt \nonumber \\
             -\int_{\tau\wedge \sigma}^{\infty}e^{\int_{0}^{t}b_{x}(\widehat{X}_{s},Y_{s})dt}\pi_{x}(\widehat{X}_{t},\theta)dt  
             +k_{1}e^{\int_{0}^{\tau}b_{x}(\widehat{X}_{s},Y_{s})ds}\boldsymbol{1}_{\{\tau\leq\sigma\}}+k_{2}e^{\int_{0}^{\sigma}b_{x}(\widehat{X}_{s},Y_{s})ds}\boldsymbol{1}_{\{\sigma<\tau\}} \bigg] \nonumber \\
            =\widehat{\mathbb{E}}_{(x,i)}\bigg[ \int_{0}^{\infty}e^{\int_{0}^{t}b_{x}(\widehat{X}_{s},Y_{s})ds}\pi_{x}(\widehat{X}_{t},\theta)dt\bigg] -\widehat{\mathbb{E}}_{(x,i)}\bigg[\widehat{\mathbb{E}}_{(x,i)}\bigg[ \int_{\tau\wedge \sigma}^{\infty}e^{\int_{0}^{t}b_{x}(\widehat{X}_{s},Y_{s})ds}\pi_{x}(\widehat{X}_{t},\theta)dt \bigg|\mathcal{F}_{\tau\wedge \sigma}\bigg]\bigg] \nonumber \\
            +\widehat{\mathbb{E}}_{(x,i)}\bigg[k_{1}e^{\int_{0}^{\tau}b_{x}(\widehat{X}_{s},Y_{s})ds}\boldsymbol{1}_{\{\tau\leq\sigma\}}
            +k_{2}e^{\int_{0}^{\sigma}b_{x}(\widehat{X}_{s},Y_{s})ds}\boldsymbol{1}_{\{\sigma<\tau\}} \bigg].
    \end{multline}
    Then, defining 
        \[
            G_{0}(x,i;\theta):=\widehat{\mathbb{E}}_{(x,i)}\bigg[\int_{0}^{\infty}e^{\int_{0}^{t}b_{x}(\widehat{X}_{s},Y_{s})ds}\pi_{x}(\widehat{X}_{t},\theta)dt \bigg],
        \]
        by the strong Markov property of $(\widehat{X}_{t},Y_{t})_{t\geq 0}$ we obtain
        \begin{multline}
            \widehat{J}(x,i;\tau,\sigma,\theta)=G_{0}(x,i;\theta)-\widehat{\mathbb{E
            }}_{(x,i)}\bigg[e^{\int_{0}^{\tau\wedge \sigma}b_{x}(\widehat{X}_{s},Y_{s})ds}G_{0}(\widehat{X}_{\tau\wedge\sigma},Y_{\tau\wedge\sigma};\theta) \bigg] \\
            +\widehat{\mathbb{E}}_{(x,i)}\bigg[k_{1}e^{\int_{0}^{\tau}b_{x}(\widehat{X}_{s},Y_{s})ds}\boldsymbol{1}_{\{\tau\leq\sigma\}}+k_{2}e^{\int_{0}^{\sigma}b_{x}(\widehat{X}_{s},Y_{s})ds}\boldsymbol{1}_{\{\sigma<\tau\}} \bigg].
        \end{multline}
        Therefore, 
        \begin{multline}
            \widehat{J}(x,i;\tau,\sigma,\theta)=G_{0}(x,i;\theta)+\widehat{\mathbb{E}}_{(x,i)}\bigg[ e^{\int_{0}^{\tau}b_{x}(\widehat{X}_{s},Y_{s})ds}\big(k_{1}-G_{0}(\widehat{X}_{\tau},Y_{\tau};\theta) \big)\boldsymbol{1}_{\{\tau\leq\sigma\}}\bigg] \\
            +\widehat{\mathbb{E}}_{(x,i)}\bigg[e^{\int_{0}^{\sigma}b_{x}(\widehat{X}_{s},Y_{s})ds}\big(k_{2}-G_{0}(\widehat{X}_{\sigma},Y_{\sigma};\theta) \big)\boldsymbol{1}_{\{\sigma<\tau\}} \bigg];
        \end{multline}
        that is,
        \[
          \widehat{J}(x,i;\tau,\sigma,\theta)=G_{0}(x,i;\theta)+\widehat{\mathbb{E}}_{(x,i,1)}\big[ G_{1}(\bar{X}_{t};\theta)\boldsymbol{1}_{\{\tau\leq\sigma\}}+G_{2}(\bar{X}_{t};\theta)\boldsymbol{1}_{\{\sigma<\tau\}} \big],
        \]
        where we have set $G_{1}(x,i,z;\theta):=z(k_{1}-G_{0}(x,i;\theta))$, $G_{2}(x,i,z;\theta):=z(k_{2}-G_{0}(x,i;\theta))$, we have introduced the 3-dimensional right-continuous strong Markov process
        \begin{equation}
            \bar{X}_{t}:=(\widehat{X}_{t},Y_{t},Z_{t}),
        \end{equation}
        with $Z_{t}:=z\cdot e^{\int_{0}^{t}b_{x}(\widehat{X}_{s},Y_{s})ds}$, and $\widehat{\mathbb{E}}_{(x,i,1)}$ is the expectation with the respect to $\widehat{\mathbb{P}}_{(x,i,1)}[ \ \cdot \ ]:=\widehat{\mathbb{P}}[ \ \cdot\ |\widehat{X}_{0}=x,Y_{0}=i,Z_{0}=1 ]$. We observe that by Condition (\ref{eq:(2.1-3)}) in Assumption \ref{eq:Ass 2.1} and Condition (\ref{eq:2.2-5}) in Assumption \ref{eq:Ass 2.2}, 
        \[
            \widehat{\mathbb{E}}_{(x,i,1)}\big[\sup_{t\geq 0}\big|G_{i}(\bar{X}_{t};\theta)\big|\big]<\infty,\; i=0,1,2,
        \]
        where as well as, $G_{2}(\bar{X}_{t};\theta)\leq G_{1}(\bar{X}_{t};\theta)$ for any $t\geq 0$. Also, 
        \[
            \lim_{t\uparrow\infty}G_{1}(\bar{X}_{t};\theta)=0\;\text{ and }\lim_{t\uparrow\infty}G_{2}(\bar{X}_{t};\theta)=0,\quad \widehat{\mathbb{P}}_{(x,i,1)}\text{-a.s.}
        \]
        Hence, we can apply Theorem 2.1 from \cite{doi:10.1137/S0040585X97983821} and  complete the proof.
        \end{proof}
    \begin{proof}[Proof of Proposition \ref{eq:Prop 3.2}]
    \label{Proof of Proposition 3.2}
         We argue by contradiction and we suppose that $\mathcal{S}^{\theta}_{\sup}=\emptyset$. This implies that $\sigma^{*}=\infty,\;\widehat{\mathbb{P}}_{(x,i)}$-a.s. and for any $(x,i)\in\mathcal{O}$, and we thus obtain that
         \begin{align*}
            k_{2} & <v(x,i;\theta)=\inf_{\tau\in\mathcal{T}}\widehat{\mathbb{E}}_{(x,i)}\bigg[ \int_{0}^{\tau}
            e^{\int_{0}^{t}b_{x}(\widehat{X}_{s},Y_{s})ds}\pi_{x}(\widehat{X}_{t},\theta)dt+k_{2}e^{\int_{0}^{\tau}b_{x}(\widehat{X}_{s},Y_{s})ds}\bigg] \\
             & \leq \widehat{\mathbb{E}}_{(x,i)}\bigg[ \int_{0}^{\infty}
            e^{\int_{0}^{t}b_{x}(\widehat{X}_{s},Y_{s})ds}\pi_{x}(\widehat{X}_{t},\theta)dt\bigg].
        \end{align*}
        By Condition (\ref{eq:2.2-4}) and Condition (\ref{eq:2.2-5}) in Assumption \ref{eq:Ass 2.2} we reach a contradiction with $k_{2}>0$ by taking $x\uparrow \overline{x}$. The proof of $\mathcal{S}^{\theta}_{\inf}\neq \emptyset$ follows by similar arguments.
    \end{proof}
    The next result, which will be used in the proof of Proposition \ref{eq:Prop 3.3} below, provides the so-called \textit{semiharmonic characterization} of $v$ (see \cite{doi:10.1137/S0040585X97983821}). Its proof is direct consequence of Theorem 2.1 in \cite{doi:10.1137/S0040585X97983821}.
    \begin{proposition}
        \label{Prop A.1}
        For any $(x,i)\in\mathcal{O}$, we have under $\widehat{\mathbb{P}}_{(x,i)}$ that
        \begin{enumerate}
            \item $\bigg( \int_{0}^{t\wedge \tau^{*}(\theta)}\pi_{x}(\widehat{X}_{t},\theta)dt+e^{\int_{0}^{t\wedge \tau^{*}(\theta)}b_{x}(\widehat{X}_{t},Y_{t})dt}v(\widehat{X}_{t\wedge \tau^{*}(\theta)},Y_{t\wedge \tau^{*}(\theta)}) \bigg)_{t\geq 0}$ is a $\mathbb{F}$-submartingale;
            \item $\bigg( \int_{0}^{t\wedge \sigma^{*}(\theta)}\pi_{x}(\widehat{X}_{t},\theta)dt+e^{\int_{0}^{t\wedge \sigma^{*}(\theta)}b_{x}(\widehat{X}_{t},Y_{t})dt}v(\widehat{X}_{t\wedge \sigma^{*}(\theta)},Y_{t\wedge \sigma^{*}(\theta)}) \bigg)_{t\geq 0}$ is a $\mathbb{F}$-supermartingale;
            \item \label{Prop A.1-3} $\bigg( \int_{0}^{t\wedge \tau^{*}(\theta)\wedge \sigma^{*}(\theta)}\pi_{x}(\widehat{X}_{t},\theta)dt+e^{\int_{0}^{t\wedge \tau^{*}(\theta)\wedge\sigma^{*}(\theta)}b_{x}(\widehat{X}_{t},Y_{t})dt}v(\widehat{X}_{t\wedge \tau^{*}(\theta)\wedge \sigma^{*}(\theta)},Y_{t\wedge \tau^{*}(\theta)\wedge \sigma^{*}(\theta)}) \bigg)_{t\geq 0}$ is a $\mathbb{F}$-martingale.
        \end{enumerate}
    \end{proposition}
    
    \begin{proof}[Proof of Proposition \ref{eq:Prop 3.3}]
     \label{Proof of Proposition 3.3}
     The proof is organized in several steps.
     
        \textbf{Step 1:} We start by proving that $v(\cdot,i;\theta)\in C^{0}(\mathcal{I})$, for any $(i,\theta)\in\mathbb{Y}\times\mathbb{R}_{+}$. Letting $(x,i,\theta)\in\mathcal{O}\times \mathbb{R}_{+}$ and $\epsilon\in(0,1)$, we recall the optimal stopping times $\tau^{*}(\theta)=\inf\{t\geq 0:\widehat{X}_{t}\leq \alpha_{Y_{t}}(\theta)\}$, $\sigma^{*}(\theta)=\inf\{t\geq 0:\widehat{X}_{t}\geq \beta_{Y_{t}}(\theta)\}$ $\widehat{\mathbb{P}}_{(x,i)}$-a.s. and define $\sigma^{*}_{\epsilon}(\theta):=\inf\{t\geq 0:\widehat{X}_{t}\geq \beta_{Y_{t}}(\theta)\}$ $\widehat{\mathbb{P}}_{(x+\epsilon,i)}$-a.s., analogously $\tau^{*}_{\epsilon}(\theta)$. Then we have that
        \begin{align}
            \label{first ineq with est}
            \big|v(x+\epsilon,i;\theta)&-v(x,i;\theta) \big| \notag \\ 
            &\leq \big|\widehat{J}(x+\epsilon,i;\tau^{*}(\theta),\sigma^{*}_{\epsilon}(\theta),\theta)-\widehat{J}(x,i;\tau^{*}(\theta),\sigma^{*}_{\epsilon}(\theta),\theta)\big| \nonumber \\ 
            &\leq\bigg|\widehat{\mathbb{E}}\bigg[ \int_{0}^{\tau^{*}(\theta)\wedge \sigma^{*}_{\epsilon}(\theta)}\big(e^{\int_{0}^{t}b_{x}(\widehat{X}^{x+\epsilon}_{s},Y^{i}_{s})ds}\pi_{x}(\widehat{X}^{x+\epsilon}_{t},\theta)-e^{\int_{0}^{t}b_{x}(\widehat{X}^{x}_{s},Y^{i}_{s})ds}\pi_{x}(\widehat{X}^{x}_{t},\theta)\big)ds \bigg]\bigg|
            \nonumber \\
            &\quad +k_{1}\bigg|\widehat{\mathbb{E}}\bigg[ \bigg(e^{\int_{0}^{\tau^{*}(\theta)}b_{x}(\widehat{X}^{x+\epsilon}_{s},Y^{i}_{s})ds}-e^{\int_{0}^{\tau^{*}(\theta)}b_{x}(\widehat{X}^{x}_{s},Y^{i}_{s})ds}\bigg)\boldsymbol{1}_{\{\tau^{*}(\theta)<\sigma^{*}_{\epsilon}(\theta)\}} \bigg] \bigg| \notag \\
            &\quad +k_{2}\bigg|\widehat{\mathbb{E}}\bigg[ \bigg(e^{\int_{0}^{\sigma^{*}_{\epsilon}(\theta)}b_{x}(\widehat{X}^{x+\epsilon}_{s},Y^{i}_{s})ds}-e^{\int_{0}^{\sigma^{*}_{\epsilon}(\theta)}b_{x}(\widehat{X}^{x}_{s},Y^{i}_{s})ds}\bigg)\boldsymbol{1}_{\{\sigma^{*}_{\epsilon}(\theta)<\tau^{*}(\theta)\}} \bigg]\bigg| \notag\\
            &=:A^{\epsilon}+k_{1}B^{\epsilon}(\tau^{*}(\theta))+k_{2}B^{\epsilon}(\sigma^{*}_{\epsilon}(\theta)).
        \end{align}
        We bound the first term as follows, 
        \begin{align}
        \label{eq:(A.8)}
             A^{\epsilon}&\leq \widehat{\mathbb{E}}\bigg[ \int_{0}^{\infty}e^{\int_{0}^{t}b_{x}(\widehat{X}^{x+\epsilon}_{s},Y^{i}_{s})ds}\big|\pi_{x}(\widehat{X}^{x+\epsilon}_{t},\theta)-\pi_{x}(\widehat{X}^{x}_{t},\theta)\big|ds\bigg]   \\
            &\quad+\widehat{\mathbb{E}}\bigg[\int_{0}^{\infty}  \big| e^{\int_{0}^{t}b_{x}(\widehat{X}^{x+\epsilon}_{s},Y^{i}_{s})ds}-e^{\int_{0}^{t}b_{x}(\widehat{X}^{x}_{s},Y^{i}_{s})ds} \big|\cdot \big| \pi_{x}(\widehat{X}^{x}_{t},\theta)\big|dt\bigg]. \notag
        \end{align}
        Thanks to Condition (\ref{eq:(2.1-3)}) in Assumption \ref{eq:Ass 2.1}, (\ref{eq:(A.8)}) gives
        \begin{align}
        \label{eq:(A.9)}
            A^{\epsilon}&\leq \widehat{\mathbb{E}}\bigg[ \int_{0}^{\infty}e^{-ct}\big|\pi_{x}(\widehat{X}^{x+\epsilon}_{t},\theta)-\pi_{x}(\widehat{X}^{x}_{t},\theta)\big|dt \bigg]  \\
            &\quad +\widehat{\mathbb{E}}\bigg[ \int_{0}^{\infty}\big|e^{\int_{0}^{t}b_{x}(\widehat{X}^{x+\epsilon}_{s},Y^{i}_{s})ds}-e^{\int_{0}^{t}b_{x}(\widehat{X}^{x}_{s},Y^{i}_{s})ds} \big|\cdot\big|\pi_{x}(\widehat{X}^{x}_{t},\theta)\big|dt \bigg]. \notag
        \end{align}
        Following the same arguments as in Theorem 38 from Chapter V.7 in \cite{protter2005stochastic}, we can prove that $x\mapsto \widehat{X}^{x}_{t}$ is continuous $\widehat{\mathbb{P}}\otimes dt$-a.s., so that
        \begin{equation}
        \label{eq:(A.10)}
            \big|e^{\int_{0}^{t}b_{x}(\widehat{X}^{x+\epsilon}_{s},Y^{i}_{s})ds}-e^{\int_{0}^{t}b_{x}(\widehat{X}^{x}_{s},Y^{i}_{s})ds} \big|\to 0,\; \epsilon\downarrow 0 \; \widehat{\mathbb{P}}\otimes dt\text{-a.s.}
        \end{equation}
        Furthermore,
        \begin{equation}
        \label{eq:(A.11)}
            \big|e^{\int_{0}^{t}b_{x}(\widehat{X}^{x+\epsilon}_{s},Y^{i}_{s})ds}-e^{\int_{0}^{t}b_{x}(\widehat{X}^{x}_{s},Y^{i}_{s})ds} \big|\leq 2e^{-ct}.
        \end{equation}
        Hence, combining (\ref{eq:(A.10)}), (\ref{eq:(A.11)}) with Condition (\ref{eq:2.2-5}) in Assumption \ref{eq:Ass 2.2} we obtain that the second expectation on the right-hand side of (\ref{eq:(A.9)}) vanishes as $\epsilon\downarrow 0$. On the other hand, one can prove a version of comparison principle for hybrid-SDEs \footnote{ Writing $[0,\infty)=\bigcup_{n=1}^{\infty}[\tau_{n-1},\tau_{n})$, where $\tau_{n}$ is a $\mathbb{F}$-stopping when $Y$ makes a jump, thus applying Proposition 5.2.18 from \cite{karatzas1991brownian} to each $[\tau_{n-1},\tau_{n}),\; n\in\mathbb{N}$ and using the flow property of $(\widehat{X}_{t})_{t\geq 0}$ we obtain the result.}, i.e.$\;\widehat{X}^{x}_{t}\leq \widehat{X}^{x+\epsilon}_{t},\;\mathbb{P}\otimes dt\text{-a.s.}$ and using that $\pi_{x}(x,\theta)\geq 0$ for any $(x,\theta)\in\mathcal{I}\times\mathbb{R}_{+}$ and $x\mapsto \pi_{x}(x,\theta)$ is decreasing for any $\theta\in\mathbb{R}_{+}$, we obtain
        \begin{equation}
            \label{eq:(A.12)}
            \big|\pi_{x}(\widehat{X}^{x+\epsilon}_{t},\theta)-\pi_{x}(\widehat{X}^{x}_{t},\theta)\big|\leq 2\pi_{x}(\widehat{X}^{x}_{t},\theta),\quad \widehat{\mathbb{P}}\otimes dt\text{-a.s.}  
        \end{equation}
        
        Therefore, using again Condition (\ref{eq:2.2-5}) in Assumption \ref{eq:Ass 2.2}, we can let $\epsilon\downarrow 0$ and we observe that also the first expectation on the right-hand side of (\ref{eq:(A.8)}) converges to zero, due to the dominated convergence theorem.

        Now we focus on the two other terms on the right-hand side of (\ref{first ineq with est}). By the differentiability of the map $x\mapsto \widehat{X}^{x}$ (cf. Theorem 39 Chapter V.7 of \cite{protter2005stochastic}) and by the mean value theorem, for $\tilde{x}^{\epsilon}\in (x,x+\epsilon)$, we have that,
        \begin{equation}
            \label{mean value thm for exp}
            e^{\int_{0}^{t}b_{x}(\widehat{X}^{x+\epsilon}_{s},Y^{i}_{s})ds}-e^{\int_{0}^{t}b_{x}(\widehat{X}^{x}_{s},Y^{i}_{s})ds}=\epsilon e^{\int_{0}^{t} b_{x}(\widehat{X}^{\tilde{x}^{\epsilon}}_{s},Y^{i}_{s})ds} \Big(\int_{0}^{t} b_{xx}(\widehat{X}^{\tilde{x}^{\epsilon}}_{s},Y^{i}_{s}) \partial_{x}\widehat{X}^{x}_{s}\big|_{x=\tilde{x}^{\epsilon}} ds\Big),\quad \widehat{\mathbb{P}}\otimes dt\text{-a.s.}
        \end{equation}
        This, combined it with Assumption \ref{eq:Ass 2.1}-(\ref{eq:(2.1-3)}) gives us, for fixed $\eta\in\mathcal{T}$,  
        \begin{align}
            \label{estimate of E}
            B^{\epsilon}(\eta)&\leq \epsilon\widehat{\mathbb{E}}\bigg[ \bigg|e^{\int_{0}^{\eta} b_{x}(\widehat{X}^{\tilde{x}^{\epsilon}}_{s},Y^{i}_{s})ds} \Big(\int_{0}^{\eta} b_{xx}(\widehat{X}^{\tilde{x}^{\epsilon}}_{s},Y^{i}_{s}) \partial_{x}\widehat{X}^{x}_{s}\big|_{x=\tilde{x}^{\epsilon}}ds\Big) \bigg|\bigg]  \\
            &\leq \epsilon\widehat{\mathbb{E}}\bigg[e^{-c\eta}\int_{0}^{\eta}\big|b_{xx}(\widehat{X}^{\tilde{x}^{\epsilon}}_{s},Y^{i}_{s})\big| \partial_{x}\widehat{X}^{x}_{s}\big|_{x=\tilde{x}^{\epsilon}}ds\bigg] \notag \\
            &\leq \epsilon\widehat{\mathbb{E}}\bigg[\int_{0}^{\infty}e^{-cs}\big|b_{xx}(\widehat{X}^{\tilde{x}^{\epsilon}}_{s},Y^{i}_{s})\big| \partial_{x}\widehat{X}^{x}_{s}\big|_{x=\tilde{x}^{\epsilon}}ds\bigg] \notag,
        \end{align}
        hence sending $\epsilon\downarrow 0$ in (\ref{estimate of E}) by Assumption \ref{eq:Ass 2.2}-(\ref{assumption bxx}) invoking dominated convergence theorem we can exchange limit with expectation and we obtain (cf. (\ref{first ineq with est}))
        \begin{equation*}
          \lim_{\epsilon\downarrow 0}\big(k_{1}B^{\epsilon}(\tau^{*}(\theta))+k_{2}B^{\epsilon}(\sigma^{*}_{\epsilon}(\theta))\big)=0.
        \end{equation*}
        \textbf{Step 2:} Let $(i,\theta)\in \mathbb{Y}\times \mathbb{R}_{+}$. By Theorem \ref{eq:Thm 3.1}, we know that $v(\cdot,i;\theta)=k_{1},\;x\in(\underline{x},\alpha_{i}(\theta))$ and $v(\cdot,i;\theta)=k_{2},\;x\in(\beta_{i}(\theta),\overline{x})$ for any $(i,\theta)\in\mathbb{Y}\times\mathbb{R}_{+}$. Hence, it is sufficient to show that $v(\cdot,i;\theta)\in C^{2}(\mathcal{C}^{\theta}),\text{ for any }(i,\theta)\in\mathbb{Y}\times \mathbb{R}_{+}$. We then define the following boundary value problem:
        \begin{equation}
        \label{eq:(A.5)}
            \begin{cases}
                \mathcal{L}_{\widehat{X}}w(x,i;\theta)+(b_{x}(x,i)-q_{ii})w(x,i;\theta)=f(x,i;\theta),\;(x,i)\in\mathcal{C}^{\theta} \\
                w(\alpha_{i}(\theta),i;\theta)=v(\alpha_{i}(\theta),i;\theta),\; w(\beta_{i}(\theta),i;\theta)=v(\beta_{i}(\theta),i;\theta),
            \end{cases}
        \end{equation}
        where $\mathcal{L}_{\widehat{X}}w(x,i;\theta):=\frac{1}{2}\sigma^{2}(x,i)w_{xx}(x,i;\theta)+(b(x,i)+\sigma\sigma_{x}(x,i))w_{x}(x,i;\theta)$ and $f(x,i;\theta):=$
        \newline
        $-\sum_{j\neq i}q_{ij}v(x,j;\theta)-\pi_{x}(x,\theta)$. Using Condition (\ref{eq:(2.1-4)}) in Assumption \ref{eq:Ass 2.1} we conclude that (\ref{eq:(A.5)}) has a unique classical solution $w(\cdot,i;\theta)\in C^{2}(\alpha_{i}(\theta),\beta_{i}(\theta))$ (see Theorem 6.13 in \cite{gilbarg2001elliptic}).
        We then define the function $\overline{w}:(\alpha_{i}(\theta),\beta_{i}(\theta))\times \mathbb{Y}\times\mathbb{R}_{+}\mapsto \mathbb{R}$ by:
        \[
            \bar{w}(x,j;\theta):=\begin{cases}
                        w(x,i;\theta),\quad j=i \\
                        v(x,j;\theta),\quad j\neq i,
                    \end{cases}
        \]
        and the stopping time $\tau_{Y}:=\inf\{t\geq 0:Y_{t}^{i}\neq i\}$. Hence, setting $\eta(\theta):=\tau_{Y}\wedge \tau^{*}(\theta)\wedge \sigma^{*}(\theta)$, from Dynkin's formula we have:
        \begin{equation}
            \bar{w}(x,i;\theta)=w(x,i;\theta)=\widehat{\mathbb{E}}_{(x,i)}\bigg[ v(\widehat{X}_{\eta(\theta)},Y_{\eta(\theta)};\theta)e^{\int_{0}^{\eta(\theta)}b_{x}(\widehat{X}_{s},Y_{s})ds}+\int_{0}^{\eta(\theta)} e^{\int_{0}^{t}b_{x}(\widehat{X}_{s},Y_{s})ds}\pi_{x}(\widehat{X}_{t},\theta)dt \bigg]
        \end{equation}
        because $\bar{w}(\widehat{X}_{\eta(\theta)},Y_{\eta(\theta)};\theta)=v(\widehat{X}_{\eta(\theta)},Y_{\eta(\theta)};\theta)$, and, on $s<\eta(\theta)$,
        \[
            \frac{1}{2}\sigma^{2}(\widehat{X}_{s},Y_{s})\bar{w}_{xx}(\widehat{X}_{s},Y_{s};\theta)+(b(\widehat{X}_{s},Y_{s})+\sigma\sigma_{x}(\widehat{X}_{s},Y_{s}))\bar{w}_{x}(\widehat{X}_{s},Y_{s};\theta)+(b_{x}(\widehat{X}_{s},Y_{s})-q_{ii})\bar{w}(\widehat{X}_{s},Y_{s};\theta)=
        \]
        \[
            \frac{1}{2}\sigma^{2}(\widehat{X}_{s},Y_{s})w_{xx}(\widehat{X}_{s},Y_{s};\theta)+(b(\widehat{X}_{s},Y_{s})+\sigma\sigma_{x}(\widehat{X}_{s},Y_{s}))w_{x}(\widehat{X}_{s},Y_{s};\theta)+(b_{x}(\widehat{X}_{s},Y_{s})-q_{ii})w(\widehat{X}_{s},Y_{s};\theta)
        \]
        \[
            \hspace{-10cm}=f(\widehat{X}_{s},Y_{s};\theta), \quad \widehat{\mathbb{P}}_{(x,i)}\text{-a.s.}.
        \]
        However, $\eta(\theta)\leq \tau^{*}(\theta)\wedge \sigma^{*}(\theta),\; \widehat{\mathbb{P}}_{(x,i)}$-a.s., which, by the martingale property (\ref{Prop A.1-3}) in Proposition \ref{Prop A.1}, gives
        \[
            v(x,i;\theta)=\widehat{\mathbb{E}}_{(x,i)}\bigg[ v(\widehat{X}_{\eta(\theta)},Y_{\eta(\theta)};\theta)e^{\int_{0}^{\eta(\theta)}b_{x}(\widehat{X}_{s},Y_{s})ds}+\int_{0}^{\eta(\theta)} e^{\int_{0}^{t}b_{x}(\widehat{X}_{s},Y_{s})ds}\pi_{x}(\widehat{X}_{t},\theta)dt \bigg]=w(x,i;\theta).
        \]
        Since the latter holds for every $x\in(\alpha_{i}(\theta),\beta_{i}(\theta))$ and for arbitrary $i\in\mathbb{Y}$, we obtain $v\equiv w$, which finally yields that $v(\cdot,i;\theta)\in C^{2}((\alpha_{i}(\theta),\beta_{i}(\theta))$.
        \vspace{20px} 
        
        \textbf{Step 3:} Thanks to the previous steps, it is now sufficient to prove the continuity of $v_{x}(\cdot,i;\theta)$ only across the free boundaries $\alpha_{i}(\theta)$ and $\beta_{i}(\theta)$, for any $(i,\theta)\in\mathbb{Y}\times \mathbb{R}_{+}$. We prove this only for $\alpha_{i}(\theta)$, being the proof across $\beta_{i}(\theta)$ similar. Notice that from (\ref{first ineq with est}) we have
        \begin{align}
        \label{eq:(A.13)}
            \bigg| &\frac{v(x+\epsilon,i;\theta)-v(x,i;\theta)}{\epsilon} \bigg|\leq\frac{1}{\epsilon}A^{\epsilon}+\frac{k_{1}}{\epsilon}B^{\epsilon}(\tau^{*}(\theta))+\frac{k_{2}}{\epsilon}B^{\epsilon}(\sigma^{*}_{\epsilon}(\theta)).
%            &\leq \frac{1}{\epsilon}\widehat{\mathbb{E}}\bigg[\int_{0}^{\tau_{\epsilon}(\theta)\wedge \sigma^{*}(\theta)}e^{-ct}\big| \pi_{x}(\widehat{X}^{x+\epsilon}_{t},\theta)-\pi_{x}(\widehat{X}^{x}_{t},\theta) \big|ds\bigg] \notag \\
%            &\quad +\frac{1}{\epsilon}\widehat{\mathbb{E}}\bigg[\int_{0}^{\tau^{\epsilon}(\theta)\wedge \sigma^{*}(\theta)}\big|e^{\int_{0}^{t}b_{x}(\widehat{X}^{x+\epsilon}_{s},Y^{i}_{s})ds}-e^{\int_{0}^{t}b_{x}(\widehat{X}^{x}_{s},Y^{i}_{s})ds} \big|\cdot\big|\pi_{x}(\widehat{X}^{x}_{t},\theta)\big|dt \bigg] \notag \\ 
%            &\quad +\frac{k_{1}}{\epsilon}\widehat{\mathbb{E}}\bigg[ \bigg|e^{\int_{0}^{\tau^{*}(\theta)}b_{x}(\widehat{X}^{x+\epsilon}_{s},Y^{i}_{s})ds}-e^{\int_{0}^{\tau^{*}(\theta)}b_{x}(\widehat{X}^{x}_{s},Y^{i}_{s})ds}\bigg| \bigg] \notag \\
%            &\quad +\frac{k_{2}}{\epsilon}\widehat{\mathbb{E}}\bigg[ \bigg|e^{\int_{0}^{\sigma^{\epsilon}(\theta)}b_{x}(\widehat{X}^{x+\epsilon}_{s},Y^{i}_{s})ds}-e^{\int_{0}^{\sigma^{\epsilon}(\theta)}b_{x}(\widehat{X}^{x}_{s},Y^{i}_{s})ds}\bigg| \bigg] \notag \\
%            &\leq\frac{1}{\epsilon}A^{\epsilon}+\frac{k_{1}}{\epsilon}E^{\epsilon}(\tau^{*}(\theta))+\frac{k_{2}}{\epsilon}E^{\epsilon}(\sigma^{*}_{\epsilon}(\theta)). \notag
        \end{align}
%        We can prove that $x\mapsto \widehat{X}^{x}$ is differentiable in similar fashion as to Theorem 39, Chapter V.7 in \cite{protter2005stochastic}, then by mean value theorem we obtain, for some $x'\in (x,x+\epsilon)$,
%        \begin{equation}
%        \label{eq:(A.14)}
%            \bigg|e^{\int_{0}^{t}b_{x}(\widehat{X}^{x+\epsilon}_{s},Y^{i}_{s})ds}-e^{\int_{0}^{t}b_{x}(\widehat{X}^{x}_{s},Y^{i}_{s})ds} \bigg|=\bigg|e^{\int_{0}^{t}b_{x}(\widehat{X}^{x'}_{s},Y^{i}_{s})ds}\cdot \int_{0}^{t}b_{xx}(\widehat{X}^{x'}_{s},Y^{i}_{s})\partial_{x}\widehat{X}^{x}_{s}\big|_{x=x'}ds  \bigg|\cdot \epsilon,\; \widehat{\mathbb{P}}\otimes dt\text{-a.s.},
%        \end{equation}
        We recall that the map $x\mapsto \widehat{X}^{x}$ is differentiable and by regularity of profit the function (cf. Assumption \ref{eq:Ass 2.2}), for $x^{\epsilon}_{1}\in (x,x+\epsilon)$ we have that,
        \begin{equation}
        \label{eq:(A.15)}
            \big| \pi_{x}(\widehat{X}^{x+\epsilon}_{t},\theta)-\pi_{x}(\widehat{X}^{x}_{t},\theta) \big|\leq \epsilon\big| \pi_{xx}(\widehat{X}^{x^{\epsilon}_{1}}_{t},\theta) \big|\partial_{x}\widehat{X}^{x}\big|_{x=x^{\epsilon}_{1}},\quad \widehat{\mathbb{P}}\otimes dt\text{-a.s.}
        \end{equation}
        Using (\ref{eq:(A.15)}) in the first expectation on the right-hand side of (\ref{eq:(A.9)}), and applying the mean value theorem in the second expectation (cf.\ (\ref{mean value thm for exp})), we can find $x^{\epsilon}_{2}\in (x,x+\epsilon)$ such that
        \begin{align}
        \label{eq:(A.16)}
            \frac{1}{\epsilon}A_{\epsilon}&\leq \widehat{\mathbb{E}}\bigg[\int_{0}^{\tau^{*}(\theta)\wedge \sigma^{*}_{\epsilon}(\theta)}e^{-ct}\big|\pi_{xx}(\widehat{X}^{x^{\epsilon}_{1}}_{t},\theta)\big|\partial_{x}\widehat{X}^{x}_{s}\big|_{x=x^{\epsilon}_{1}} ds\bigg] \\
            &\quad +\widehat{\mathbb{E}}\bigg[\int_{0}^{\tau^{*}(\theta)\wedge \sigma^{*}_{\epsilon}(\theta)}\bigg(e^{\int_{0}^{t}b_{x}(\widehat{X}^{x^{\epsilon}_{2}}_{s},Y^{i}_{s})ds}\cdot\big|\pi_{x}(\widehat{X}^{x}_{t},\theta)\big|\cdot\bigg| \int_{0}^{t}b_{xx}(\widehat{X}^{x^{\epsilon}_{2}}_{s},Y^{i}_{s}) \partial_{x} \widehat{X}^{x}_{s}  \big|_{x=x^{\epsilon}_{2}}ds \bigg|\bigg) dt \bigg]. \notag
        \end{align}
        On the other hand, by (\ref{estimate of E}), for any $\eta\in\mathcal{T}$, we have
        \begin{align}
            \label{Estimation for B deriv}
            \frac{1}{\epsilon}B^{\epsilon}(\eta)\leq \widehat{\mathbb{E}}\bigg[\int_{0}^{\infty}e^{-cs}\big|b_{xx}(\widehat{X}^{\tilde{x}^{\epsilon}}_{s},Y^{i}_{s})\big| \partial_{x}\widehat{X}^{x}_{s}\big|_{x=\tilde{x}^{\epsilon}}ds\bigg],
        \end{align}
        for some $x_{\epsilon}\in (x,x+\epsilon)$. Hence, sending $\epsilon\to 0$, we can interchange limits and expectation thanks to Conditions (\ref{eq:2.2-5}) and (\ref{assumption bxx}) in Assumption \ref{eq:Ass 2.2}, and find from (\ref{eq:(A.13)}) that (cf.\ (\ref{eq:(A.16)}) and (\ref{Estimation for B deriv})) 
        \begin{align}
            \label{eq:(A.17)}
            0 \leq|v_{x}(x,i;\theta)|&
            \leq\widehat{\mathbb{E}}\bigg[\int_{0}^{\tau^{*}(\theta)\wedge \sigma^{*}(\theta)}e^{-ct}|\pi_{xx}(\widehat{X}^{x}_{t},\theta)|\partial_{x}\widehat{X}^{x}_{t} ds\bigg] \\
            &\quad +\widehat{\mathbb{E}}\bigg[\int_{0}^{\tau^{*}(\theta)\wedge \sigma^{*}(\theta)}e^{-ct}\cdot\big|\pi_{x}(\widehat{X}^{x}_{t},\theta)\big|\cdot \bigg(\int_{0}^{t}\big|b_{xx}(\widehat{X}^{x}_{s},Y^{i}_{s})\big|\partial_{x}\widehat{X}^{x}_{s}ds\bigg) dt\bigg] \notag \\
            &\quad + k_{1} \widehat{\mathbb{E}}\bigg[ e^{\int_{0}^{\tau^{*}(\theta)} b_{x}(\widehat{X}^{x}_{s},Y^{i}_{s})ds} \Big(\int_{0}^{\tau^{*}(\theta)}\big|b_{xx}(\widehat{X}^{x}_{s},Y^{i}_{s})\big| \partial_{x}\widehat{X}^{x}_{s}\big|_{x}ds\Big)\bigg] \notag \\
            &\quad + k_2 \widehat{\mathbb{E}}\bigg[ e^{\int_{0}^{\sigma^{*}(\theta)} b_{x}(\widehat{X}^{x}_{s},Y^{i}_{s})ds} \Big(\int_{0}^{\sigma^{*}(\theta)}\big|b_{xx}(\widehat{X}^{x}_{s},Y^{i}_{s})\big| \partial_{x}\widehat{X}^{x}_{s}ds\Big)\bigg], \notag 
        \end{align}
        where we have also used that $\sigma^{*}_{\epsilon}(\theta)\to \sigma^{*}(\theta)$ $\widehat{\mathbb{P}}_{(x,i)}\otimes d\theta\text{-a.s.}$
        Finally, sending $x\downarrow \alpha_{i}(\theta)$, $\tau^{*}(\theta)\to 0,\widehat{\mathbb{P}}_{(x,i)}\otimes d\theta\text{-a.s.}$ so that $v(\cdot,i;\theta)$ is $C^{1}$ across $\alpha_{i}(\theta)$. 
    \end{proof}
  Finally, arguing as in the proofs of Steps 1 and 3 in the proof of Proposition 3.3, one can evaluate upper and lower bounds for $\frac{1}{\epsilon}\big(v(x+\epsilon,i;\theta)-v(x,i;\theta) \big)$, which, after taking limits as $\epsilon\downarrow 0$ yield the following probabilistic representation of $v_{x}$.
\begin{lemma}
    \label{repr vx}
    Under the Assumptions \ref{eq:Ass 2.2}-(\ref{eq:2.2-5}) and (\ref{assumption bxx})-(\ref{eq:Ass 2.1}),  the  representation
    \begin{align*}
        v_{x}(x,i;\theta)=&\widehat{\mathbb{E}}\bigg[ \int_{0}^{\tau^{*}(\theta)\wedge\sigma^{*}(\theta)}e^{\int_{0}^{t}b_{x}(\widehat{X}^{x}_{s},Y^{i}_{s})ds}\bigg( \pi_{xx}(\widehat{X}^{x}_{t},\theta)\partial_{x}\widehat{X}^{x}_{t} \\
        & \qquad +\Big(\int_{0}^{t}b_{xx}(\widehat{X}^{x}_{s},Y^{i}_{s})\partial_{x}\widehat{X}^{x}_{s}ds\Big)\pi_{x}(\widehat{X}^{x}_{t},\theta) \bigg)dt \bigg] \\
        & \quad +k_1 \widehat{\mathbb{E}}\bigg[ e^{\int_{0}^{\tau^{*}(\theta)} b_{x}(\widehat{X}^{x}_{s},Y^{i}_{s})ds} \Big(\int_{0}^{\tau^{*}(\theta)}b_{xx}(\widehat{X}^{x}_{s},Y^{i}_{s}) \partial_{x}\widehat{X}^{x}_{s}ds\Big) \boldsymbol{1}_{ \{ \tau^{*}(\theta) < \sigma^{*}(\theta) \} } \bigg] \\
        & \quad + k_2  \widehat{\mathbb{E}}\bigg[ e^{\int_{0}^{\sigma^{*}(\theta)} b_{x}(\widehat{X}^{x}_{s},Y^{i}_{s})ds} \Big(\int_{0}^{\sigma^{*}(\theta)}b_{xx}(\widehat{X}^{x}_{s},Y^{i}_{s}) \partial_{x}\widehat{X}^{x}_{s}ds\Big) \boldsymbol{1}_{ \{\sigma^{*}(\theta)<\tau^{*}(\theta) \} } \bigg] 
    \end{align*}
    holds for any $(x,i,\theta)\in\mathcal{O}\times \mathbb{R}_{+}$, where $\tau^{*}(\theta)=\tau^{*}(x,i;\theta)$ and $\sigma^{*}(\theta)=\sigma^{*}(x,i;\theta)$ are as in Theorem \ref{eq:Thm 3.1}. 
\end{lemma}

\section{Skorokhod Reflection Problem with Regime-Switching and Proof of Theorem \ref{Thm 3.2}}
\label{Appendix B}
In this section we focus on the existence and uniqueness of a solution to $\mathbf{SP}(\alpha,\beta,x,i)$ as defined in Definition \ref{eq:Def 3.1}. Then we show a comparison result and finally we prove Theorem \ref{Thm 3.2}.

Recall $I$ as in  (\ref{eq:(3.18)}) and denote by $\Gamma:\mathcal{D}[0,\infty)\to \mathcal{D}[0,\infty)$ the following operator
\begin{align}
    \label{(B.1)}
    \Gamma (X)_{t}:=I(X)_{t}-\max \bigg\{ \big( x-\beta_{i}(\theta)\big)^{+}\wedge \inf_{s\leq t}\big( I(X)_{s}-\alpha_{Y_{s}}(\theta) \big), \nonumber \\
    \sup_{s\leq t}\bigg( \big(I(X)_{s}-\beta_{Y_{s}}(\theta)\big) \wedge \inf_{u\in [s,t]}\big(I(X)_{u}-\alpha_{Y_{u}}(\theta)\big)\bigg) \bigg\}=:I(X)_{t}-\Xi(X)_{t}.
\end{align}
We first prove that the operator has a fixed point.
\begin{proposition}
\label{eq:Prop B.1}
    There exists unique solution to Skorokhod problem $\mathbf{SP}(\alpha,\beta,x,i)$.
\end{proposition}
\begin{proof}
    Let $T>0$ and $X,X'\in\mathcal{D}[0,\infty)$, then using the inequalities
    \begin{align*}
        \max\big\{ \alpha_{1},\beta_{1} \big\}-\max\big\{ \alpha_{2},\beta_{2} \big\}&\leq \max\big\{ \alpha_{1}-\alpha_{2},\beta_{1}-\beta_{2} \big\} \\
        \big| \sup_{s\leq t}|f(s)|-\sup_{s\leq t}|g(s)|\big|&\leq \sup_{s\leq t}\big| f(s)-g(s) \big| \\
        \big| \inf_{s\leq t}|f(s)|-\inf_{s\leq t}|g(s)|\big|&\leq \sup_{s\leq t}\big| f(s)-g(s) \big|,
    \end{align*}
    we obtain that
    \begin{equation}
        \sup_{s\leq t}\big| \Gamma(X)_{s}-\Gamma(X')_{s}\big|^{2}\leq 4\sup_{s\leq t}\big| I(X)_{s}-I(X')_{s} \big|^{2},\; t\in [0,T].
    \end{equation}
    We  now derive an estimate for $\sup_{s\leq t}\big|I(X)_{s}-I(X')_{s} \big|^{2}$. In particular, we have that
    \begin{align*}
        \mathbb{E}\big[ \sup_{s\leq t}\big|I(X)_{s}-I(X')_{s} \big|^{2} \big] & \leq 2\mathbb{E}\bigg[ \sup_{s\leq t}\bigg|\int_{0}^{s}\big( b(X_{s},Y_{s})-b(X_{s}',Y_{s})\big)ds\bigg|^{2} \\
        & +\sup_{s\leq t}\bigg|\int_{0}^{s}\big( \sigma(X_{s},Y_{s})-\sigma(X_{s}',Y_{s})\big)dW_{s}\bigg|^{2}\bigg].
    \end{align*}
    By an application of Hölder's inequality in the drift term and an application of Burkholder-Davis-Gundy's inequality to the local martingale $M_{t}:=\int_{0}^{t}\big(\sigma(X_{s},Y_{s})-\sigma(X_{s}',Y_{s})\big)dW_{s}$, we find that
    \begin{align*}
        \mathbb{E}\bigg[ \sup_{s\leq t}\big|I(X)_{s}-I(X')_{s} \big|^{2} \bigg] & \leq \mathbb{E}\bigg[ C(T+1)\int_{0}^{t}\big|b(X_{s},Y_{s})-b(X_{s}',Y_{s})\big|^{2}ds\bigg] \\
        & \leq CK^{2}(1+T)\mathbb{E}\bigg[ \int_{0}^{t}\sup_{u\leq s}\big|X_{u}-X_{u}'\big|^{2}ds\bigg],
    \end{align*}
    where the last inequality is due to the Lipschitz property of $b$ and $\sigma$. Finally, for a constant $C_{1}=C_{1}(K,T)>0$ it holds 
    \begin{equation}
        \label{eq:(B.3)}
        \mathbb{E}\bigg[ \sup_{s\leq t}\big| \Gamma(X)_{s}-\Gamma(X')_{s}\big|^{2} \bigg]\leq C_{1}\mathbb{E}\bigg[ \int_{0}^{t}\sup_{u\leq s}\big|X_{u}-X_{u}'\big|^{2}ds\bigg],\; t\in [0,T].
    \end{equation}
    Let $\{X^{(n)}\}_{n\in\mathbb{N}}$ be a $\mathbb{F}$-adapted process such that
    \begin{equation*}
        \begin{cases}
        X^{(0)}_{t}:=x,\; t\in [0,T] \\
        X^{(n)}_{t}:=\Gamma(X^{(n-1)})_{t},\; t\in [0,T].
    \end{cases}
    \end{equation*}
    Then, by (\ref{eq:(B.3)}) we have that
    \begin{equation}
        \mathbb{E}\bigg[ \sup_{s\leq t}\big| \Gamma(X^{(n)})_{s}-\Gamma(X^{(n-1)})_{s}\big|^{2} \bigg]\leq C_{1}\mathbb{E}\bigg[ \int_{0}^{t}\sup_{u\leq s}\big|X^{(n-1)}_{u}-X_{u}^{(n-2)}\big|^{2}ds\bigg],
    \end{equation}
    which yields
    \begin{equation}
        \label{eq:(B.5)}
        \mathbb{E}\bigg[ \sup_{s\leq t}\big| \Gamma(X^{(n)})_{s}-\Gamma(X^{(n-1)})_{s}\big|^{2} \bigg]\leq \frac{(Rt)^{n}}{n!},
    \end{equation}
    where $\mathbb{E}\big[ \sup_{s\leq t}|X^{(1)}_{s}-x|^{2} \big]\leq Rt$. Hence, by using (\ref{eq:(B.5)}) and Markov inequality, we obtain that
    \begin{equation*}
        \int_{0}^{t}b(X^{(n)}_{s},Y_{s})ds+\int_{0}^{t}\sigma(X^{(n)}_{s},Y_{s})dW_{s}\to \int_{0}^{t}b(X_{s},Y_{s})ds+\int_{0}^{t}\sigma(X_{s},Y_{s})dW_{s},
    \end{equation*}
    uniformly in compact sets of $[0,T]$ as $n\uparrow \infty$. Defining $X:=\lim_{n\to\infty}\Gamma(X^{(n)})$, and taking $\xi=\xi^{+}-\xi^{-}=\Xi(X)$, we write $X^{\xi}=I(X^{\xi})+\xi$ so that the pair $(X^{\xi},\xi)$ satisfies the properties of Skorokhod problem in Definition \ref{eq:Def 3.1} for every $t\in [0,T]$. To extend the solution to $[0,\infty)$, we observe that (\ref{eq:(B.3)}) does not depend from the initial conditions. Hence, the result holds on each interval with length $T$ and, by writing $[0,\infty):=\bigcup_{n=1}^{\infty}[(n-1)T,nT)$, the proof is complete.
\end{proof}

The next Proposition establishes a comparison principle for the optimally controlled state with respect to the mean-field parameter.

\begin{proposition}
\label{Prop B.2}
    Let $x_{1},x_{2}\in\mathcal{I}$ and $\theta_{1},\theta_{2}\in\mathbb{R}_{+}$ such that $x_{1}\leq x_{2}$ and $\theta_{1}\leq \theta_{2}$ then it is true that:
    \begin{equation}
        X^{x_{1},\xi^{*}(\theta_{1})}_{t}\leq X^{x_{2},\xi^{*}(\theta_{2})}_{t},\; \mathbb{P}\otimes dt\text{-a.s.}
    \end{equation}
\end{proposition}

\begin{proof}
    Let $\{X^{(n,x,\theta)}\}_{n\in\mathbb{N}}$ to be a sequence of $\mathbb{F}$-adapted processes such that
    \begin{equation}
        dX^{(n,x,\theta)}_{t}=\big(b(X^{(n,x,\theta)}_{t},Y_{t})+g_{n}^{(\theta)}(X^{(n,x,\theta)}_{t},Y_{t})\big)dt+\sigma(X^{(n,x,\theta)}_{t},Y_{t})dW_{t},
    \end{equation}
        where,
    \begin{equation*}
        g_{n}^{(\theta)}(x,i):=\begin{cases}
            0,\quad x\in [\alpha_{i}(\theta),\beta_{i}(\theta)] \\
            -n(x-\alpha_{i}(\theta)),\quad x<\alpha_{i}(\theta) \\
            -n(x-\beta_{i}(\theta)),\quad x>\beta_{i}(\theta).
        \end{cases}
    \end{equation*}
    Let $\{\tau_{k}\}_{k\in\mathbb{N}}$ be the sequence of random times where the jumps of $Y$ occurred. 
    Thanks to Proposition 5.2.18 in \cite{karatzas1991brownian} and Lemma \ref{Lemma 4.1}, by an induction argument we find 
    \begin{equation*}
        X^{(n,x_{1},\theta_{1})}_{t}\leq X^{(n,x_{2},\theta_{2})}_{t},\quad \mathbb{P}\text{-a.s. for any }t\in [\tau_{k},\tau_{k+1}), \quad \text{for any $k\in\mathbb{N}_{0}$}. 
    \end{equation*} 
    Then, using Theorem 1.4.1 in \cite{pilipenko} (with estimates as in Proposition 2.3 in \cite{yin2009hybrid}), we obtain, $\mathbb{P}\otimes dt\text{-a.s.}$,
    \begin{equation*}
        X^{x_{1},\xi^{*}(\theta_{1})}_{t}=
        \lim_{n\uparrow\infty}X^{(n,x_{1},\theta_{1})}_{t}\leq \lim_{n\uparrow\infty}X^{(n,x_{2},\theta_{2})}_{t}=X^{x_{2},\xi^{*}(\theta_{2})}_{t},
    \end{equation*}
    thus completing the proof.
\end{proof}

We are now in the position to provide the proof of Theorem \ref{Thm 3.2}
\begin{proof}[Proof of Theorem \ref{Thm 3.2}]
\label{Proof of Theorem 3.2}
    By Proposition \ref{eq:Prop B.1}, there exists a unique solution to $\mathbf{SP}(\alpha(\theta),\beta(\theta),x,i)$, denoted by $(X^{\xi^{*}(\theta)},\xi^{*}(\theta))$. This satisfies the properties collected in Definition \ref{eq:Def 3.1}. 
    We prove that $\xi^{*}(\theta)\in\mathcal{A}_{e}$ and verify that the optimal policy has the form as in (\ref{eq:(3.20)}). By Property (\ref{eq:Def 3.1-1}) in Definition \ref{eq:Def 3.1}, we have that
    \[
        X^{\xi^{*}(\theta)}_{t}\in [\alpha_{Y_{t}}(\theta),\beta_{Y_{t}}(\theta)],\; \forall t\geq 0\quad \mathbb{P}_{(x,i)}-a.s.\;,
    \]
    so that
    \begin{equation}
    \label{eq:(B.6)}
        \limsup_{T\uparrow\infty}\frac{1}{T}\mathbb{E}_{(x,i)}\big[ |X^{\xi^{*}(\theta)}_{T}| \big]=0.
    \end{equation}
    It remains to show that $\mathbb{E}\big[|\xi^{*}_{T}(\theta)|\big]<\infty$ for any $T<\infty$. To that end, let $T<\infty$ and $\psi:\mathcal{O}\to \mathbb{R}$ be any classical solution of
    \begin{equation}
    \label{BVP}
    \begin{cases}
        \mathcal{L}_{(X,Y)}\psi(x,i)=0,\; x\in (\alpha_{i}(\theta),\beta_{i}(\theta)),\; i\in\mathbb{Y} \\
        \psi_{x}(\alpha_{i}(\theta),i)=1,\quad \psi_{x}(\beta_{i}(\theta),i)=-1.
    \end{cases}
    \end{equation}
    Applying Itô-Meyer's formula to $\{\psi(X^{\xi^{*}(\theta)}_{t},Y_{t})\}_{t\geq 0}$ on $[0,T]$ we obtain
    \begin{align*}
        \mathbb{E}_{(x,i)}\big[ \psi(X^{\xi^{*}(\theta)}_{T},Y_{T}) \big]=& \psi(x,i)+\mathbb{E}_{(x,i)}\bigg[ \int_{0}^{T}\mathcal{L}_{(X,Y)}\psi(X^{\xi^{*}(\theta)}_{t},Y_{t})dt \bigg] \\ 
        & \quad +\mathbb{E}_{(x,i)}\bigg[ \int_{0}^{T}\psi_{x}(X^{\xi^{*}(\theta)}_{t},Y_{t})d\xi^{*}_{t}(\theta) \bigg], 
    \end{align*}
    and by using the fact that $\psi$ is a solution to (\ref{BVP}) we conclude that
    \begin{equation*}
       \mathbb{E}_{(x,i)}\big[ \psi(X^{\xi^{*}(\theta)}_{T},Y_{T}) \big] =\psi(x,i)+\mathbb{E}_{(x,i)}\big[ |\xi^{*}_{T}(\theta)|\big].
    \end{equation*}
    Since $X^{\xi^{*}(\theta)}\in [\alpha_{Y_{t}}(\theta),\beta_{Y_{t}}(\theta)],\text{ for any }t\geq 0\quad \mathbb{P}_{(x,i)}$-a.s., we obtain that 
    \begin{equation*}
        \mathbb{E}_{(x,i)}\big[ \psi(X^{\xi^{*}(\theta)}_{T},Y_{T}) \big]\leq \max_{i\in\mathbb{Y}}\max_{ x\in [\alpha_{i}(\theta),\beta_{i}(\theta)]}|\psi(x,i)|<\infty,
    \end{equation*}
    and therefore
    \begin{equation}
        \label{eq:(B.7)}
        \mathbb{E}_{(x,i)}\big[ |\xi^{*}_{T}(\theta)|\big]\leq \mathbb{E}_{(x,i)}\big[ \psi(X^{\xi^{*}(\theta)}_{T},Y_{T}) \big]+\psi(x,i)<\infty,\text{ for any }T<\infty.
    \end{equation}
    Combining (\ref{eq:(B.6)}) and (\ref{eq:(B.7)}) we conclude that $\xi^{*}\in\mathcal{A}_{e}$. Finally, defining 
    \[
        \widehat{\xi}^{+}_{t}(\theta):=\sup_{0\leq s\leq t}\bigg( \alpha_{Y_{s}}(\theta)-I(X^{\widehat{\xi}(\theta)})_{s}+\widehat{\xi}^{-}_{s}(\theta) \bigg)^{+}\text{  and  } \widehat{\xi}^{\; -}_{t}(\theta):=\sup_{0\leq s\leq t}\bigg( I(X^{\widehat{\xi}(\theta)})_{s}+\widehat{\xi}^{+}_{s}(\theta)-\beta_{Y_{s}}(\theta) \bigg)^{+}
    \] 
    it is easy to see that $\widehat{\xi}(\theta):=\widehat{\xi}^{+}(\theta)-\widehat{\xi}^{-}(\theta)$ satisfies the properties in Definition \ref{eq:Def 3.1}. Since, however, the solution to $\mathbf{SP}(\alpha(\theta),\beta(\theta),x,i)$ is unique, it follows that $\xi^{*,+}\equiv\widehat{\xi}^{+}$ and $\xi^{*,-}\equiv\widehat{\xi}^{-}$.
\end{proof}

\section{Proof of (\ref{(1) in Thm 4.1}) in Theorem \ref{Thm 4.1}: The Fokker-Planck Coupled System}
\label{Appendix C}
We follow closely the proof of Theorem 1 in \cite{DAURIA20121566}.
Let $\mu^\theta$ be the cumulative distribution function of the stationary distribution $\nu^\theta$ (cf.\ Proposition \ref{Prop 4.1}).
To simplify the notation, we drop the dependence on $\theta$; that is, we set $\mu := \mu^\theta$ as well as $\alpha_i := \alpha_i (\theta)$, $\beta_i := \beta_i (\theta)$ for any $i \in \mathbb Y$.
We divide the rest of the proof in three steps.

\textbf{Step 1.}
We first derive a variational equation for the function $\mu$. 
To this end, introduce the generator $\mathcal L _{(X,Y)}^{\alpha,\beta}$ of the process $(X^\xi,Y)$, where $X^\xi$ is reflected at the boundaries $(\alpha,\beta) :=\{ (\alpha_i,\beta_i) \}_{i \in \mathbb Y}$.
Namely, for any $\phi:=(\phi(\cdot,i))_{i\in\mathbb{Y}}\in C^{2}(\mathcal{I};\mathbb{R}^{d})$ define
\begin{equation*}
    \mathcal{L}_{(X,Y)}^{\alpha,\beta} \phi (x,i) :=\frac{1}{2}\sigma^{2}(x,i)\phi_{xx}(x,i)+b(x,i)\phi_{x}(x,i)+\sum_{j \in \mathbb Y} q_{ij}\widehat \phi (x,j), \quad (x,i) \in \mathcal O,
\end{equation*}
where $\widehat \phi (x,i) := \phi (\alpha_i^- \lor x \land \beta_i,i)$.
From Chapter IV in \cite{ethier2009markov}, we have that $\mu$ solves the equation
\begin{equation}
\label{eq:(C.2)}
      \sum_{i\in\mathbb{Y}}\int_{\alpha_{i}^{-}}^{\beta_{i}}\mathcal{L}_{(X,Y)}^{\alpha,\beta} \phi(x,i) d\mu(x,i)=0,\quad\text{for any }\phi:=(\phi(\cdot,i))_{i\in\mathbb{Y}}\in C^{2}(\mathcal{I};\mathbb{R}^{d}).
\end{equation}
Hence, by Theorem 5.3.4 in \cite{Arapostathis_book} we obtain that $\mu$ is absolutely continuous with respect to the Lebesgue measure with density function $ \mu_x$.

Next, split the interval $[\min_{i\in\mathbb{Y}}\alpha_{i},\max_{i\in\mathbb{Y}}\beta_{i}]$ in the disjoint sub-intervals 
$$
I_{k}=(l_{k-1},l_{k}),\quad k=1,..,K,
$$
with $l_{0}=\min_{i\in\mathbb{Y}}\{\alpha_{i}\},\;l_{K}=\max_{i\in\mathbb{Y}}\{\beta_{i}\}$ and
\begin{equation*}
   l_{k+1}=\min_{\substack{i\in\mathbb{Y}: \\ \alpha_{i}>l_{k}}}\{\alpha_{i}\}\wedge \min_{\substack{i\in\mathbb{Y}: \\ \beta_{i}>l_{k}}}\{\beta_{i}\}.
\end{equation*}
Introduce the set of test functions 
     \begin{equation}
           \mathcal{V}:=\big\{\phi\in C^{2}(\mathcal{I};\mathbb{R}^{d}):\phi_{x}(x,i)=\phi_{xx}(x,i)=0,\ x\in\mathcal{I}\setminus (\alpha_{i}(\theta),\beta_{i}(\theta)),\;\text{ for any }i\in\mathbb{Y}\big\}.
     \end{equation}
Since $\mu(\alpha_{i}^{-},i)=0$, for $\phi \in \mathcal V$ we observe that
\begin{align*}
     \int_{\alpha_{i}^{-}}^{\beta_{i}}\widehat{\phi}(x,j)d\mu(x,i) 
         =\mu(\beta_{i},i) \widehat \phi (\beta_{j},j)-\int_{\alpha_{j}}^{\beta_{j}}\mu(\alpha_{i}^{-}\vee x\wedge \beta_{i},i)\phi_{x}(x,j)dx,
\end{align*}
so that, exchanging the indexes in th sum, we find
\begin{align}\label{eq int by parts 3}
\sum _{i \in \mathbb Y} \sum_{j\in\mathbb{Y}}q_{ij} \int_{\alpha_{i}^{-}}^{\beta_{i}}\widehat{\phi}(x,j)d\mu (x,i)
=& \sum _{i \in \mathbb Y} \sum_{j\in\mathbb{Y}}q_{ij} \mu(\beta_{i},i) \widehat \phi (\beta_{j},j) \\ \notag
&- \sum _{i \in \mathbb Y} \sum_{j\in\mathbb{Y}}q_{ji} \int_{\alpha_{i}}^{\beta_{i}}\mu(\alpha_{j}^{-}\vee x\wedge \beta_{j},j)\phi_{x}(x,i)dx. \notag
\end{align}
By substituting  \eqref{eq int by parts 3} in \eqref{eq:(C.2)} and rearranging the terms, for any $\phi\in\mathcal{V}$, we obtain 
\begin{equation}\label{(C.4)}
   \sum_{i\in\mathbb{Y}}  \big( I_w^i(\mu,\phi) +  \Sigma^i(\mu  ,\phi)  \big) =0,
\end{equation}
where 
\begin{align*}
     I_w^i(\mu,\phi) & :=  \int_{\alpha_{i}}^{\beta_{i}} \frac{1}{2}\sigma^{2}(x,i)\mu_{x}(x,i)\phi_{xx}(x,i)dx \\ 
     & \quad + \int_{\alpha_{i}}^{\beta_{i}} \Big( b(x,i) \mu_{x}(x,i)  -\sum_{j\in\mathbb{Y}}q_{ji}\mu(\alpha_{j}^{-}\vee x\wedge \beta_{j},j)\Big) \phi_{x}(x,i)dx, \\
     \Sigma^i(\mu,\phi) & := \sum_{j\in\mathbb{Y}}q_{ji}\mu(\beta_{j},j)\widehat{\phi}(\beta_{i},i).
\end{align*}  
   
\textbf{Step 2.}
The aim of this step is to gradually use the variational equation \eqref{(C.4)} to obtain the system \eqref{eq:(4.1)} and the proper regularity of $\mu$.
We argue as follows:
\begin{enumerate}
    \item First, for any $i \in \mathbb Y$ and $k$ such that $I_k \subset (\alpha_i, \beta_i)$, choose $\phi (\cdot, j) =0$ for $j \ne i$ and $\phi (\cdot, i)$ such that $\supp \phi (\cdot, i ) \subset I_k$. For such a choice of $\phi$, we have $\Sigma^i(\mu,\phi)=0$, so that from \eqref{(C.4)} we find $I_w^i(\mu,\phi)=0$.
Thus, setting $\psi := \phi_x (\cdot,i)$, we deduce that the function $\mu(\cdot, i)$ is a solution to the variational equation
\begin{align*}
    \int_{I_k} \bigg( & \frac{1}{2}\sigma^{2}(x,i)\mu_{x}(x,i)\psi_{x}(x,i)dx 
    \\ &+ \Big( b(x,i)\mu_{x}(x,i)  - \sum_{j\in\mathbb{Y}}q_{ji}\mu(\alpha_{j}^{-}\vee x\wedge \beta_{j},j) \Big) \psi (x,i) \bigg) dx=0, 
\end{align*}
for any $\psi \in C^2_c (I_k)$. 
Therefore, by the interior regularity for elliptic equations (see Theorem 8.10 in \cite{gilbarg2001elliptic}), we obtain that $\mu(\cdot, i) \in C^2 (I_k;\mathbb R)$. 
Since the interval $I_k$ is arbitrary, we obtain  
\begin{equation}
\label{eq C 2 regularity}
\mu (\cdot,i)\in C^{2}\Big( (\alpha_{i},\beta_{i})\setminus \bigcup_{j\in\mathbb{Y}}\{\alpha_{j},\beta_{j}\} \Big), 
\end{equation}
for any $i \in \mathbb Y$.
\item
Next, for $i\in \mathbb{Y}$ we use the enumeration $l_{k}=\alpha_{i}<l_{k+1}<...<l_{m-1}<l_{m}=\beta_{i}$  and define $\Delta f(l_{j},i)=f(l_{j}^{+},i)-f(l_{j}^{-},i)$ for $k<j<m$, $\Delta f(l_{k},i)=-f(l_{k}^{-},i)$ and $\Delta f(l_{m},i)=f(l_{m}^{+},i)$.
Since $\mu(\cdot,i) \in C^2 (I_\ell)$ for $\ell = k+1,...,m$, 
dividing the domain of integration into the $I_\ell$'s  and
applying integration by parts, for $\phi \in \mathcal V$ we obtain
\begin{align}\label{eq int by parts 1}
    I_w^i(\mu,\phi) =  I_s^i(\mu,\phi) + \Psi_{\sigma}^i(\mu,\phi),
\end{align}
where 
\begin{align*}
I_s^i(\mu,\phi) & :=  - \int_{\alpha_{i}}^{\beta_{i}}\bigg( \frac{1}{2}\sigma^{2}(x,i)\mu_{xx}(x,i) - (b(x,i)-\sigma\sigma_{x}(x,i))\mu_{x}(x,i) \\
     & \quad \quad \quad \quad \quad \quad  + \sum_{j\in\mathbb{Y}}q_{ji}\mu(\alpha_{j}^{-}\vee x\wedge \beta_{j},j)\bigg)\phi_{x}(x,i)dx, \\
\Psi_{\sigma}^i(\mu,\phi) &:= - \sum_{j=k}^{m}\frac{1}{2}\sigma^{2}(l_{j},i)\Delta \mu_{x}(l_{j},i)\phi_{x}(l_{j},i).
\end{align*}
In particular, for any $i \in \mathbb Y$ and $\ell$ such that $I_\ell \subset (\alpha_i, \beta_i)$, by choosing $\phi (\cdot, j) =0$ for $j \ne i$ and $\phi (\cdot, i)$ such that $\supp \phi (\cdot, i ) \subset I_\ell$, $\mu$ solves the equation $I_s^i(\mu, \phi) =0$.
Hence, since the $i, I_\ell$ are arbitrary, we conclude that
$$
 \frac{1}{2}\sigma^{2}(x,i)\mu_{xx}(x,i)-(b(x,i)-\sigma\sigma_{x}(x,i))\mu_{x}(x,i)  +\sum_{j\in\mathbb{Y}}q_{ji}\mu(\alpha_{j}^{-}\vee x\wedge \beta_{j},j)=0, 
$$
in $(\alpha_i,\beta_i)$, for any $i\in \mathbb Y$; that is, $\mu$ solves the equation \eqref{eq:(4.1)} pointwise. 
\item Furthermore, using the latter equation, we deduce that $I_s^i(\mu,\phi)=0$ for any $\phi \in \mathcal V$.
Therefore, by choosing $\phi \in \mathcal V$ such that $\phi_x (l_j,i) = \phi_{xx} (l_j,i)=0$ for any $j,i \in \mathbb Y$, from \eqref{(C.4)} and \eqref{eq int by parts 1} we obtain $\Sigma^i(\mu,\phi)=0$.
Since the values of $\phi(\beta_i,i)$ are arbitrary, this in turn implies that
\begin{equation*}
    \sum_{j\in\mathbb{Y}}q_{ji}\mu (\beta_{j},j)=0,\quad i\in\mathbb{Y}.
\end{equation*}
Such a system of equations corresponds to the eigenvector problem with zero eigenvalue of the matrix $\mathbb Q$ and its solution is given by the stationary distribution of $Y$.
Hence, we have $\mu (\beta_{i} ,i)=p(i)$ for any $i\in\mathbb{Y}$, proving that the boundary conditions of \eqref{eq:(4.1)} are met.
\item Finally, using $I_s^i(\mu,\phi)=\Sigma^i(\mu,\phi)=0$ in \eqref{(C.4)},  we have 
$$ 
\sum_{i \in \mathbb Y} \sum_{\substack{j=k: \\l_{k}=\alpha_{i},\\l_{m}=\beta_{i}}}^{m}\frac{1}{2}\sigma^{2}(l_{j},i) \Delta \mu_{x}(l_{j},i)\phi_{x}(l_{j},i)  = 0, \quad \text{for any $\phi \in \mathcal V$,}
$$ 
from which we deduce (thanks also to \eqref{eq C 2 regularity}) that $\mu \in C^1 ( \mathcal I; \mathbb R ^d)$. 
\end{enumerate}
Concluding, $\mu (\cdot,i)\in C^{1}([\alpha_{i},\beta_{i}])\cap C^{2}((\alpha_{i},\beta_{i})\setminus \bigcup_{j\in\mathbb{Y}}\{\alpha_{j},\beta_{j}\})$ is a classical solution to the boundary value problem in \eqref{eq:(4.1)}.

\textbf{Step 3.} 
Now we prove that \eqref{eq:(4.1)} admits a unique solution. 
Assume that \eqref{eq:(4.1)} has two solutions, denoted by $\mu_{1}=(\mu_{1}(\cdot,i))_{i\in\mathbb{Y}}$ and $\mu_{2}=(\mu_{2}(\cdot,i))_{i\in\mathbb{Y}}$, then the difference $\tilde{\mu}:=\mu_{2}-\mu_{1}=(\mu_{2}(\cdot,i)-\mu_{1}(\cdot,i))_{i\in\mathbb{Y}}$ solves \eqref{eq:(4.1)} with $\tilde{\mu}(x,i)=0$ for $x\in \{\alpha_{i},\beta_{i}\}$.
Therefore, by Theorem 1 in \cite{maxprinc}, we obtain that
\begin{equation}
    0\leq \sup_{(\alpha_{i},\beta_{i})}\big|\tilde{\mu}(x,i)\big|\leq C\max_{\{\alpha_{i},\beta_{i}\}}\big|\tilde{\mu}(x,i)\big|=0,
\end{equation}
which in turn implies that, for any fixed $i\in\mathbb{Y}$, $\mu_{2}(x,i)=\mu_{1}(x,i),\; x\in [\alpha_{i},\beta_{i}]$. 
Hence, by Proposition \ref{Prop 4.1} the unique stationary distribution $\nu$ admits a cumulative function $\mu$ which coincides with the unique solution to \eqref{eq:(4.1)}.

\subsection*{Acknowledgements}
Funded by the Deutsche Forschungsgemeinschaft (DFG, German Research Foundation) – Project-ID 317210226 – SFB 1283. We thank the anonymous Associate Editor and Referees for constructive comments.

\scriptsize \bibliographystyle{apalike}

\bibliography{references}

\end{document}